\documentclass{amsart}
\title{Computing the symmetric $\mathfrak{gl}_1$-homology}
\author{Laura Marino}
\address{IMJ-PRG, Universit\'{e} Paris Cit\'{e}, France}
\email{\myemail{marino@imj-prg.fr}}
\usepackage{thmtools,enumerate,graphicx,xcolor,enumitem,afterpage,float,ucs,amsmath,amssymb,amscd,tabu,diagbox,sseq,amssymb,epsfig,psfrag,microtype,thm-restate,url,stmaryrd,mathtools,tikz-cd,amsthm,caption,soul,placeins,skak,verbatim,transparent,MnSymbol,mathrsfs}
\usepackage[pagebackref]{hyperref}
\usepackage[utf8]{inputenc}
\usepackage[nameinlink]{cleveref}
\usepackage[textsize=footnotesize, color=white, bordercolor=white]{todonotes}
\usepackage{subfig}
\usepackage{changepage}
\usepackage{nicematrix}
\usepackage{bbm}
\usepackage{multirow}
\usepackage{wrapfig}
\let\cref\Cref
\crefname{subsection}{section}{sections}
\Crefname{subsection}{Section}{Sections}
\Crefname{enumi}{}{}
\Crefname{equation}{}{}
\definecolor{darkblue}{RGB}{0,0,96}
\definecolor{medblue}{RGB}{0,128,128}
\definecolor{gray}{RGB}{127,127,127}
\definecolor{darkred}{RGB}{160,0,0}
\definecolor{lightyellow}{RGB}{255,255,128}
\hypersetup{colorlinks={true},linkcolor={medblue},citecolor={darkblue},filecolor={gray},urlcolor={darkblue}}
\newcommand{\myemail}[1]{\href{mailto:#1}{#1}}
\newcommand{\qua}{\hskip 0.4em \ignorespaces}
\def\arxiv#1{\relax\ifhmode\unskip\qua\fi
\href{http://arxiv.org/abs/#1}%
{\tt arXiv:\penalty -100\unskip#1}}

\def\ZB#1{\relax\ifhmode\unskip\qua\fi
\href{https://zbmath.org/?q=an:#1}{\tt zb#1}}
\def\xox#1{\csname xx#1\endcsname}

\numberwithin{equation}{section}
\declaretheorem[numberwithin=section,name={Lemma}]{lem}
\newtheorem{theorem}[lem]{Theorem}

\newtheorem{prop}[lem]{Proposition}
\newtheorem{conjecture}[lem]{Conjecture}
\theoremstyle{definition}
\newtheorem{dfn}[lem]{Definition}

\newtheorem{remark}[lem]{Remark}

\newtheorem{example}[lem]{Example}

\DeclareMathAlphabet{\mathpzc}{OT1}{pzc}{m}{it}

\newcommand{\N}{\mathbb{N}}
\newcommand{\Z}{\mathbb{Z}}

\newcommand{\Q}{\mathbb{Q}}

\newcommand{\rk}{\operatorname{rk}}

\graphicspath{{figures/}}
\definecolor{annotation}{rgb}{.2,.7,.2}
\definecolor{firstmap}{rgb}{0,0,1}
\definecolor{secondmap}{rgb}{.7,.7,1}
\definecolor{homotopy}{rgb}{1,.1,.1}

\renewcommand*{\backrefalt}[4]{%
\tiny
\ifcase #1 %
No citations.%
\or
Cited on page~#2.%
\else
Cited on pages~#2.%
\fi
}


\begin{document}
\begin{abstract}
The symmetric $\mathfrak{gl}_n$-homologies, introduced by Robert and Wagner, provide a categorification of the Reshetikhin--Turaev invariants corresponding to symmetric powers of the standard representation of quantum $\mathfrak{gl}_n$. Unlike in the exterior setting, these homologies are already non-trivial when $n=1$. Moreover, in this case, their construction can be greatly simplified.
Our first aim is giving a down-to-earth description of the non-equivariant symmetric $\mathfrak{gl}_1$-homology, together with relations that hold in this setting. We then find a basis for the state spaces of graphs, and use it to construct an algorithm and a program computing the invariant for uncolored links. 
\end{abstract}

\maketitle
    
\section{Introduction}

In the last 20 years, several categorifications of the $\mathfrak{gl}_n$-polynomials have emerged, declined in a variety of flavours: see, for instance, \cite{KH1, khr1, zbMATH05530200, QuefRose} and \cite{zbMATH06893251} for an overview. The $\mathfrak{gl}_n$-polynomials constitute a particular subset of the Reshetikhin--Turaev invariants, those corresponding to exterior powers of the standard representation of $U_q(\mathfrak{gl}_n)$, and so their categorifications, which are (almost) all isomorphic, will be denoted here by \emph{exterior $\mathfrak{gl}_n$-homologies}. Parallel to these, categorifications of the Reshetikhin--Turaev invariants arising from symmetric powers of the standard representation of $U_q(\mathfrak{gl}_n)$ have started to gain interest in recent years (see Cooper--Krushkal \cite{cooperkrushkal} and, later, Cautis \cite{zbMATH06791663}). 
The categorification due to Cautis, in particular, specifically applies to the symmetric case, and categorifies the \emph{symmetric} MOY calculus \cite{QuefRose, QRS}.
Building on Cautis' construction and on the work of Queffelec--Rose--Sartori \cite{QRS}, Robert and Wagner introduced in 2019 a new such categorification, referred to as the \emph{symmetric $\mathfrak{gl}_n$-homologies} (\cite{zbMATH07206869}). 

\bigskip

Although the proof that they are indeed link invariants is algebraically involved, and requires redefining them in terms of Soergel bimodules and Hochschild homology, the definition of the symmetric $\mathfrak{gl}_n$-homologies is completely combinatorial and geometric in nature. The construction, in its full generality, is done in an equivariant setting and for colored links, and is defined over the rationals. It relies on the universal construction of \cite{zbMATH00863709} to define a monoidal functor $\mathscr{S}_n$ from a category of \emph{vinyl graphs} (plane, trivalent graphs with a labeling on their edges, see \cref{def:vinyl}) and \emph{foams} to one of modules over a ring.
Graphs and foams constitute respectively the objects and morphisms of the \emph{hypercube of resolutions} obtained from a braid diagram of a link, and the functor $\mathscr{S}_n$, followed by a flattening, turns this hypercube into the symmetric $\mathfrak{gl}_n$-chain complex. The functor yields, in particular, a categorification of the symmetric MOY calculus (see \cite{zbMATH07206869}).
The image of a vinyl graph $\Gamma$ under $\mathscr{S}_n$ must thus have finite rank, and this is achieved by quotienting by a family of relations, given implicitly by an evaluation $\llangle \cdot \rrangle$ of foams. 

\bigskip

The structure of the symmetric $\mathfrak{gl}_n$-homologies strongly resembles the exterior construction, and it is therefore likely that similar numerical invariants, such as Rasmussen type ones, can be extracted from them. The first step toward achieving this, however, is dealing with the difficulty of carrying out computations both of the symmetric $\mathfrak{gl}_n$-homologies themselves and inside the modules $\mathscr{S}_n(\Gamma)$.
This is partly due to the non-local character of these invariants, which makes it often necessary to work with the entire link diagram, rather than tangles inside of it (as in the exterior case). It is also hard, in general, to explicitly spell out relations that hold in $\mathscr{S}_n(\Gamma)$.
This paper addresses both of these computability issues, for the case $n=1$.
The $\mathfrak{gl}_1$-polynomial $P_1$ is defined by the skein relation 
\[
qP_1(\;\raisebox{-0.41\baselineskip}{\def\svgwidth{12pt}
\begingroup%
  \makeatletter%
  \providecommand\color[2][]{%
    \errmessage{(Inkscape) Color is used for the text in Inkscape, but the package 'color.sty' is not loaded}%
    \renewcommand\color[2][]{}%
  }%
  \providecommand\transparent[1]{%
    \errmessage{(Inkscape) Transparency is used (non-zero) for the text in Inkscape, but the package 'transparent.sty' is not loaded}%
    \renewcommand\transparent[1]{}%
  }%
  \providecommand\rotatebox[2]{#2}%
  \newcommand*\fsize{\dimexpr\f@size pt\relax}%
  \newcommand*\lineheight[1]{\fontsize{\fsize}{#1\fsize}\selectfont}%
  \ifx\svgwidth\undefined%
    \setlength{\unitlength}{59.75938772bp}%
    \ifx\svgscale\undefined%
      \relax%
    \else%
      \setlength{\unitlength}{\unitlength * \real{\svgscale}}%
    \fi%
  \else%
    \setlength{\unitlength}{\svgwidth}%
  \fi%
  \global\let\svgwidth\undefined%
  \global\let\svgscale\undefined%
  \makeatother%
  \begin{picture}(1,1.29229881)%
    \lineheight{1}%
    \setlength\tabcolsep{0pt}%
    \put(0,0){\includegraphics[width=\unitlength,page=1]{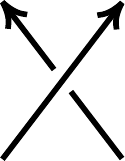}}%
  \end{picture}%
\endgroup%
}\;) - q^{-1}P_1(\;\raisebox{-0.41\baselineskip}{\def\svgwidth{12pt}
\begingroup%
  \makeatletter%
  \providecommand\color[2][]{%
    \errmessage{(Inkscape) Color is used for the text in Inkscape, but the package 'color.sty' is not loaded}%
    \renewcommand\color[2][]{}%
  }%
  \providecommand\transparent[1]{%
    \errmessage{(Inkscape) Transparency is used (non-zero) for the text in Inkscape, but the package 'transparent.sty' is not loaded}%
    \renewcommand\transparent[1]{}%
  }%
  \providecommand\rotatebox[2]{#2}%
  \newcommand*\fsize{\dimexpr\f@size pt\relax}%
  \newcommand*\lineheight[1]{\fontsize{\fsize}{#1\fsize}\selectfont}%
  \ifx\svgwidth\undefined%
    \setlength{\unitlength}{59.75939876bp}%
    \ifx\svgscale\undefined%
      \relax%
    \else%
      \setlength{\unitlength}{\unitlength * \real{\svgscale}}%
    \fi%
  \else%
    \setlength{\unitlength}{\svgwidth}%
  \fi%
  \global\let\svgwidth\undefined%
  \global\let\svgscale\undefined%
  \makeatother%
  \begin{picture}(1,1.29229857)%
    \lineheight{1}%
    \setlength\tabcolsep{0pt}%
    \put(0,0){\includegraphics[width=\unitlength,page=1]{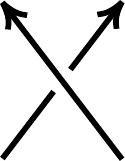}}%
  \end{picture}%
\endgroup%
}\;) = (q-q^{-1})P_1(\;\raisebox{-0.41\baselineskip}{\def\svgwidth{12pt}
\begingroup%
  \makeatletter%
  \providecommand\color[2][]{%
    \errmessage{(Inkscape) Color is used for the text in Inkscape, but the package 'color.sty' is not loaded}%
    \renewcommand\color[2][]{}%
  }%
  \providecommand\transparent[1]{%
    \errmessage{(Inkscape) Transparency is used (non-zero) for the text in Inkscape, but the package 'transparent.sty' is not loaded}%
    \renewcommand\transparent[1]{}%
  }%
  \providecommand\rotatebox[2]{#2}%
  \newcommand*\fsize{\dimexpr\f@size pt\relax}%
  \newcommand*\lineheight[1]{\fontsize{\fsize}{#1\fsize}\selectfont}%
  \ifx\svgwidth\undefined%
    \setlength{\unitlength}{62.04436251bp}%
    \ifx\svgscale\undefined%
      \relax%
    \else%
      \setlength{\unitlength}{\unitlength * \real{\svgscale}}%
    \fi%
  \else%
    \setlength{\unitlength}{\svgwidth}%
  \fi%
  \global\let\svgwidth\undefined%
  \global\let\svgscale\undefined%
  \makeatother%
  \begin{picture}(1,1.25792275)%
    \lineheight{1}%
    \setlength\tabcolsep{0pt}%
    \put(0,0){\includegraphics[width=\unitlength,page=1]{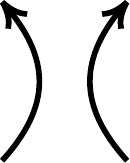}}%
  \end{picture}%
\endgroup%
}\;) \qquad \text{ and } \qquad P_1(\text{unknot})=1.
\]
While one can easily verify that this polynomial is equal to $1$ for all links, its symmetric categorification is far from being trivial, and is therefore the natural starting point.

\bigskip

A combinatorial and foam-free description of the symmetric $\mathfrak{gl}_1$-homology, in the non-equivariant case, is given in Section 2.5 of \cite{zbMATH07632766}. In particular, foams are replaced, in some cases, by \emph{decorations} by symmetric polynomials on the edges of vinyl graphs. In this paper we review this description and list a series of relations that hold in this particular setting.
We aim at providing a down-to-earth treatment of this case, in which computations are more manageable. Our main result is the definition of a basis for $\mathscr{S}_1(\Gamma)$, for all \emph{elementary vinyl graphs} $\Gamma$ (\cref{def:elementary}), that we call the \emph{d.u.r.\ basis} in what follows. It is intended at gaining a better understanding of the modules $\mathscr{S}_1(\Gamma)$, and of the morphisms between such modules. The \emph{d.u.r.\ set} of a vinyl graph $\Gamma$ consists of all possible \textbf{d}ecorations by Schur polynomials on the \textbf{u}pper-\textbf{r}ight edges of split vertices of $\Gamma$ (hence the name d.u.r.): see \cref{def:d.u.r.} and \cref{ex:d.u.r.} for more details.

\begin{equation*}
{\scriptsize \raisebox{-1\baselineskip}{\def\svgwidth{150pt}
\begingroup%
  \makeatletter%
  \providecommand\color[2][]{%
    \errmessage{(Inkscape) Color is used for the text in Inkscape, but the package 'color.sty' is not loaded}%
    \renewcommand\color[2][]{}%
  }%
  \providecommand\transparent[1]{%
    \errmessage{(Inkscape) Transparency is used (non-zero) for the text in Inkscape, but the package 'transparent.sty' is not loaded}%
    \renewcommand\transparent[1]{}%
  }%
  \providecommand\rotatebox[2]{#2}%
  \newcommand*\fsize{\dimexpr\f@size pt\relax}%
  \newcommand*\lineheight[1]{\fontsize{\fsize}{#1\fsize}\selectfont}%
  \ifx\svgwidth\undefined%
    \setlength{\unitlength}{235.41757022bp}%
    \ifx\svgscale\undefined%
      \relax%
    \else%
      \setlength{\unitlength}{\unitlength * \real{\svgscale}}%
    \fi%
  \else%
    \setlength{\unitlength}{\svgwidth}%
  \fi%
  \global\let\svgwidth\undefined%
  \global\let\svgscale\undefined%
  \makeatother%
  \begin{picture}(1,0.25447915)%
    \lineheight{1}%
    \setlength\tabcolsep{0pt}%
    \put(0.45620694,0.00706858){\makebox(0,0)[lt]{\lineheight{1.25}\smash{\begin{tabular}[t]{l}\tiny $a+b$\end{tabular}}}}%
    \put(0.56273155,0.22739194){\makebox(0,0)[lt]{\lineheight{1.25}\smash{\begin{tabular}[t]{l}\tiny $a$\end{tabular}}}}%
    \put(0.41945314,0.22866059){\makebox(0,0)[lt]{\lineheight{1.25}\smash{\begin{tabular}[t]{l}\tiny $b$\end{tabular}}}}%
    \put(0,0){\includegraphics[width=\unitlength,page=1]{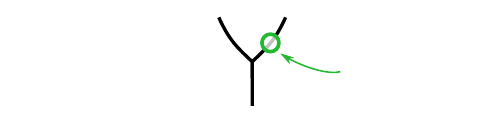}}%
    \put(0.71845261,0.10603913){\color[rgb]{0.14117647,0.7254902,0.19607843}\makebox(0,0)[lt]{\lineheight{1.25}\smash{\begin{tabular}[t]{l}Schur polynomial $s_{\lambda}$\\with $\lambda \in T(a,b)$ \end{tabular}}}}%
    \put(0,0){\includegraphics[width=\unitlength,page=2]{durintro.pdf}}%
    \put(-0.00184181,0.14266018){\makebox(0,0)[lt]{\lineheight{1.25}\smash{\begin{tabular}[t]{l}split vertex of $\Gamma$\end{tabular}}}}%
  \end{picture}%
\endgroup%
}}  
\end{equation*}

\begin{theorem} \label{thm:intro}
The d.u.r.\ set is a basis for the module $\mathscr{S}_1(\Gamma)$, for all elementary vinyl graphs $\Gamma$.    
\end{theorem} 

With some more work, we expect that this result can be extended to general vinyl graphs, but for now it only holds for the graphs appearing in the hypercubes of resolutions of uncolored links. Other than over $\Q$, the d.u.r.\ set turns out to be a basis of $\mathscr{S}_1(\Gamma)$ also as a $\Z$-module. The proof is therefore done in this more general setting. 

\bigskip

Our second result is an algorithm and a computer program, computing the non-equivariant $\mathfrak{gl}_1$-homology of uncolored links. The program is implemented in Python and is effective for links having braid length up to 12 and braid index up to 6. In \cref{sec:computations} we list the Poincar\'e polynomials associated to all prime knots with up to 10 crossings, braid length up to 12 and braid index up to 4, obtained by running our program on the braid representations of knots provided by KnotInfo \cite{chalivingston}.

\bigskip

One of our motivations for implementing this program is investigating the relations between the symmetric $\mathfrak{gl}_n$-homologies and Khovanov--Rozansky's triply graded homology \cite{khr2}, which yields a categorification of the HOMFLYPT polynomial.
Indeed, although the HOMFLYPT polynomial specialises to the Reshetikhin--Turaev invariants, the relations between their categorifications are in general more complex. It is known (\cite{ras4, zbMATH07096528, zbMATH07305663} and \cite{zbMATH07206869}, using results of Cautis \cite{zbMATH06791663}), that there exist spectral sequences from the triply graded homology to both the exterior and the symmetric $\mathfrak{gl}_n$-homologies. In the exterior case, this spectral sequence is known to be non-trivial. We formulate the following conjecture for the symmetric case.

\begin{conjecture}
The total rank of the symmetric $\mathfrak{gl}_1$-homology is equal to the total rank of the reduced triply graded homology.   
\end{conjecture}

More generally, for any $n \geq 1$, we expect the rank of the symmetric $\mathfrak{gl}_n$-homologies to be $n$ times the rank of the reduced triply graded homology.
While several computer programs have been implemented to compute the exterior $\mathfrak{gl}_n$-homologies (see for instance \cite{khoca, khoho}), computing the triply graded homology appears to be a more difficult task, as the complexity grows very quickly. However, advances have been made in \cite{computedW, ras4}, and more recently by Nakagane--Sano in \cite{computHOMFLY}, who compute the reduced triply graded homology of all prime knots with up to 10 crossings and, among the 11 crossings knots, of all those having braid length up to 13. Comparing the results of our program with the computations of Nakagane-Sano \cite{computedHOMFLY}, one verifies that the results indeed coincide for all the knots in \cref{sec:computations}.

\bigskip

Other than over the rationals, our program also computes the symmetric $\mathfrak{gl}_1$-homology of the closure of a braid diagram over $\mathbb{F}_p$, with $p$ prime. Running the program on the closure of the braids
$\sigma_1$, $\sigma_1\sigma_2^{-1}$ and $\sigma_1\sigma_2$,
which are all representatives of the unknot, one obtains three pairwise-distinct results over $\mathbb{F}_3$. This shows, in particular, that the symmetric $\mathfrak{gl}_1$-homology is not a link invariant over $\Z$. 

\subsection*{Outline of the paper} 
The first section is a review of the definition of $\mathscr{S}_1(\Gamma)$, and of all the ingredients involved in the construction of the symmetric $\mathfrak{gl}_1$-homology of \cite{zbMATH07206869}, in the non-equivariant case. As stated above, in this setting, the definitions and main properties of the objects involved in the construction can be rephrased in a more elementary way (see also \cite{zbMATH07632766}). In particular, foams are omitted and replaced, in some cases, by graphs equipped with \emph{decorations} by symmetric polynomials on their edges. We deduce some relations that hold in this setting and give a description of the morphisms used throughout the paper.

In the second section, after defining the d.u.r.\ set of $\mathscr{S}_1(\Gamma)$, we state \cref{thm:intro}, and prove it for a particularly simple case.

The third and fourth sections are concerned with the proof of \cref{thm:intro}: we first prove that the d.u.r.\ set is a basis for a special class of graphs, then show that a series of isomorphisms preserve the d.u.r.\ basis.

The fifth section is a description of an algorithm and a computer program which uses d.u.r.\ bases to compute the symmetric $\mathfrak{gl}_1$-homology of links. The outputs of the program for small knots are listed in the appendix.

\subsection*{Acknowledgments}
I would like to thank my advisor Emmanuel Wagner for guiding me through the symmetric link homologies, and for his precious support and advice. I am very grateful to my advisor Louis-Hadrien Robert for the many suggestions and helpful discussions, and for his invaluable support in programming.
The author has received funding from the European Union's Horizon 2020 research and innovation programme under the Marie Skłodowska-Curie grant agreement No 945332 \includegraphics[width=0.04\textwidth]{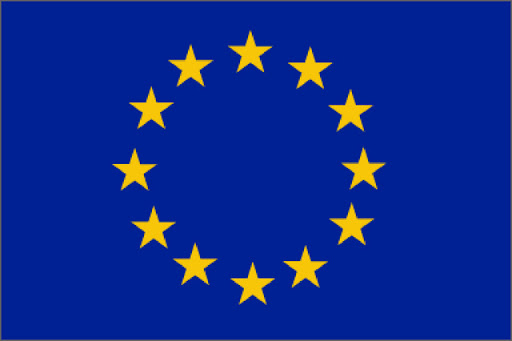}.

\section{The symmetric $\mathfrak{gl}_1$-construction} \label{sec:gl1eval}

In this section we describe the main ingredients of the non-equivariant symmetric $\mathfrak{gl}_1$-homology of links introduced by Robert and Wagner in \cite{zbMATH07206869}. Namely, we will define the graphs $\Gamma$ appearing in the hypercube of resolutions of a colored link, called \emph{vinyl graphs}, the state spaces $\mathscr{S}_1(\Gamma)$ associated to them, and some morphisms between these spaces. We defer the description of the symmetric $\mathfrak{gl}_1$-chain complex and the homology itself until \cref{sec:comp-gll_1-homol}. The construction is done here in a foam-free setting, and holds over $\Q,\,\Z$ and $\mathbb{F}_p$, with $p$ prime. For more details and a general treatment of the subject, as well as proofs of all the statements given in this section, refer to \cite{zbMATH07206869}, \cite{zbMATH07305663}, \cite{zbMATH07632766}, and \cite{rootofunity}.

\bigskip

Given an integer $k\geq 0$, consider the plane, oriented graphs $\rotatebox[origin=c]{180}{Y}^k_i,\, Y^k_i$ and $I^k$, for $i=0,\ldots ,k$, represented below.
\begin{equation*}
\raisebox{-1.2\baselineskip}{\def\svgwidth{360pt}
\begingroup%
  \makeatletter%
  \providecommand\color[2][]{%
    \errmessage{(Inkscape) Color is used for the text in Inkscape, but the package 'color.sty' is not loaded}%
    \renewcommand\color[2][]{}%
  }%
  \providecommand\transparent[1]{%
    \errmessage{(Inkscape) Transparency is used (non-zero) for the text in Inkscape, but the package 'transparent.sty' is not loaded}%
    \renewcommand\transparent[1]{}%
  }%
  \providecommand\rotatebox[2]{#2}%
  \newcommand*\fsize{\dimexpr\f@size pt\relax}%
  \newcommand*\lineheight[1]{\fontsize{\fsize}{#1\fsize}\selectfont}%
  \ifx\svgwidth\undefined%
    \setlength{\unitlength}{324.9370786bp}%
    \ifx\svgscale\undefined%
      \relax%
    \else%
      \setlength{\unitlength}{\unitlength * \real{\svgscale}}%
    \fi%
  \else%
    \setlength{\unitlength}{\svgwidth}%
  \fi%
  \global\let\svgwidth\undefined%
  \global\let\svgscale\undefined%
  \makeatother%
  \begin{picture}(1,0.10506862)%
    \lineheight{1}%
    \setlength\tabcolsep{0pt}%
    \put(-0.00102183,0.06048599){\color[rgb]{0,0,0}\makebox(0,0)[lt]{\lineheight{1.25}\smash{\begin{tabular}[t]{l}$\rotatebox[origin=c]{180}{Y}^k_i=$\end{tabular}}}}%
    \put(0,0){\includegraphics[width=\unitlength,page=1]{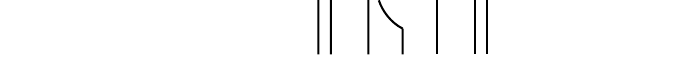}}%
    \put(0.50132862,0.06051375){\color[rgb]{0,0,0}\makebox(0,0)[lt]{\lineheight{1.25}\smash{\begin{tabular}[t]{l}$\cdots$\end{tabular}}}}%
    \put(0.66186696,0.06040985){\color[rgb]{0,0,0}\makebox(0,0)[lt]{\lineheight{1.25}\smash{\begin{tabular}[t]{l}$\cdots$\end{tabular}}}}%
    \put(0,0){\includegraphics[width=\unitlength,page=2]{mergesplit.pdf}}%
    \put(0.10007495,0.06057831){\color[rgb]{0,0,0}\makebox(0,0)[lt]{\lineheight{1.25}\smash{\begin{tabular}[t]{l}$\cdots$\end{tabular}}}}%
    \put(0.26102581,0.06047448){\color[rgb]{0,0,0}\makebox(0,0)[lt]{\lineheight{1.25}\smash{\begin{tabular}[t]{l}$\cdots$\end{tabular}}}}%
    \put(0,0){\includegraphics[width=\unitlength,page=3]{mergesplit.pdf}}%
    \put(0.40533718,0.06053871){\color[rgb]{0,0,0}\makebox(0,0)[lt]{\lineheight{1.25}\smash{\begin{tabular}[t]{l}$Y^k_i=$\end{tabular}}}}%
    \put(0,0){\includegraphics[width=\unitlength,page=4]{mergesplit.pdf}}%
    \put(0.92180723,0.06051368){\color[rgb]{0,0,0}\makebox(0,0)[lt]{\lineheight{1.25}\smash{\begin{tabular}[t]{l}$\cdots$\end{tabular}}}}%
    \put(0.81658335,0.06053864){\color[rgb]{0,0,0}\makebox(0,0)[lt]{\lineheight{1.25}\smash{\begin{tabular}[t]{l}$I^k=$\end{tabular}}}}%
    \put(0,0){\includegraphics[width=\unitlength,page=5]{mergesplit.pdf}}%
    \put(0.93081993,0.00308217){\color[rgb]{0,0.53333333,0.63137255}\makebox(0,0)[lt]{\lineheight{1.25}\smash{\begin{tabular}[t]{l}\tiny $k$\end{tabular}}}}%
    \put(0,0){\includegraphics[width=\unitlength,page=6]{mergesplit.pdf}}%
    \put(0.66780578,0.00266115){\color[rgb]{0,0.53333333,0.63137255}\makebox(0,0)[lt]{\lineheight{1.25}\smash{\begin{tabular}[t]{l}\tiny $k-i$\end{tabular}}}}%
    \put(0,0){\includegraphics[width=\unitlength,page=7]{mergesplit.pdf}}%
    \put(0.50506274,0.00256578){\color[rgb]{0,0.53333333,0.63137255}\makebox(0,0)[lt]{\lineheight{1.25}\smash{\begin{tabular}[t]{l}\tiny $i$\end{tabular}}}}%
    \put(0,0){\includegraphics[width=\unitlength,page=8]{mergesplit.pdf}}%
    \put(0.26617032,0.00256598){\color[rgb]{0,0.53333333,0.63137255}\makebox(0,0)[lt]{\lineheight{1.25}\smash{\begin{tabular}[t]{l}\tiny $k-i$\end{tabular}}}}%
    \put(0,0){\includegraphics[width=\unitlength,page=9]{mergesplit.pdf}}%
    \put(0.10431229,0.00299874){\color[rgb]{0,0.53333333,0.63137255}\makebox(0,0)[lt]{\lineheight{1.25}\smash{\begin{tabular}[t]{l}\tiny $i$\end{tabular}}}}%
  \end{picture}%
\endgroup%
}
\end{equation*}

The trivalent vertices of $\rotatebox[origin=c]{180}{Y}^k_i$ and $Y^k_i$ are called \emph{merge vertex} and \emph{split vertex} respectively. We further denote the two edges departing from a split vertex $v$ on the left and on the right by $e^l(v)$ and $e^r(v)$ respectively.

\begin{dfn} \label{def:vinyl}
A \emph{vinyl graph} $\Gamma$ is a labeled, closed, finite, oriented, trivalent, plane graph $(V(\Gamma),E(\Gamma))$, obtained as follows: 
\begin{enumerate}
    \item Glue on top of each other finitely many graphs of type $\rotatebox[origin=c]{180}{Y}^k_i$, $Y^{k'}_j$ and $I^{k''}$, for any $k,k',k'',i,j$, so that, at the end of this process, the number of split and merge vertices coincides. 
    One is only allowed to glue a graph $\Theta_1$ on top of another graph $\Theta_0$ if the number of strands on the bottom of $\Theta_1$ and on top of $\Theta_0$ coincide. We call the resulting graph $\Gamma_{\text{\tiny open}}$.
    \item $\Gamma$ is obtained by closing up the strands of $\Gamma_{\text{\tiny open}}$ as follows 
    \begin{equation*}
    \Gamma = \quad {\tiny \raisebox{-8\baselineskip}{\def\svgwidth{80pt}
\begingroup%
  \makeatletter%
  \providecommand\color[2][]{%
    \errmessage{(Inkscape) Color is used for the text in Inkscape, but the package 'color.sty' is not loaded}%
    \renewcommand\color[2][]{}%
  }%
  \providecommand\transparent[1]{%
    \errmessage{(Inkscape) Transparency is used (non-zero) for the text in Inkscape, but the package 'transparent.sty' is not loaded}%
    \renewcommand\transparent[1]{}%
  }%
  \providecommand\rotatebox[2]{#2}%
  \newcommand*\fsize{\dimexpr\f@size pt\relax}%
  \newcommand*\lineheight[1]{\fontsize{\fsize}{#1\fsize}\selectfont}%
  \ifx\svgwidth\undefined%
    \setlength{\unitlength}{178.15109263bp}%
    \ifx\svgscale\undefined%
      \relax%
    \else%
      \setlength{\unitlength}{\unitlength * \real{\svgscale}}%
    \fi%
  \else%
    \setlength{\unitlength}{\svgwidth}%
  \fi%
  \global\let\svgwidth\undefined%
  \global\let\svgscale\undefined%
  \makeatother%
  \begin{picture}(1,1.50351855)%
    \lineheight{1}%
    \setlength\tabcolsep{0pt}%
    \put(0,0){\includegraphics[width=\unitlength,page=1]{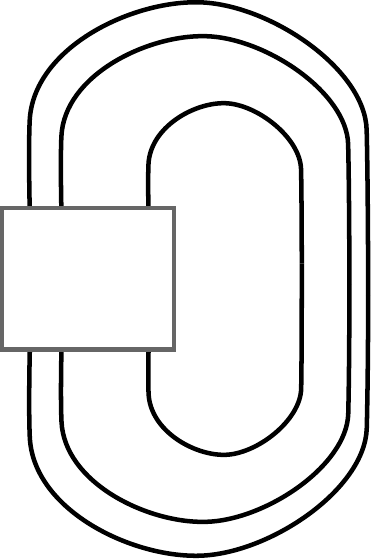}}%
    \put(0.12452252,0.71061793){\color[rgb]{0.4,0.4,0.4}\makebox(0,0)[lt]{\lineheight{1.25}\smash{\begin{tabular}[t]{l}\normalsize $\Gamma_{\text{open}}$\end{tabular}}}}%
    \put(0.82235921,0.73935748){\makebox(0,0)[lt]{\lineheight{1.25}\smash{\begin{tabular}[t]{l}$\ldots$\end{tabular}}}}%
    \put(0.21831028,1.02624051){\makebox(0,0)[lt]{\lineheight{1.25}\smash{\begin{tabular}[t]{l}$\ldots$\end{tabular}}}}%
    \put(0.2119314,0.46610953){\makebox(0,0)[lt]{\lineheight{1.25}\smash{\begin{tabular}[t]{l}$\ldots$\end{tabular}}}}%
  \end{picture}%
\endgroup%
}}
    \end{equation*}
    
    \item The edges of $\Gamma$ are given a labeling by positive integers. Around each vertex, the labeling $\ell\thinspace\colon E(\Gamma) \to \N_{>0}$ is required to follow one of the two models below:
    \begin{equation*}
    {\scriptsize \raisebox{-1.2\baselineskip}{\def\svgwidth{136pt}
\begingroup%
  \makeatletter%
  \providecommand\color[2][]{%
    \errmessage{(Inkscape) Color is used for the text in Inkscape, but the package 'color.sty' is not loaded}%
    \renewcommand\color[2][]{}%
  }%
  \providecommand\transparent[1]{%
    \errmessage{(Inkscape) Transparency is used (non-zero) for the text in Inkscape, but the package 'transparent.sty' is not loaded}%
    \renewcommand\transparent[1]{}%
  }%
  \providecommand\rotatebox[2]{#2}%
  \newcommand*\fsize{\dimexpr\f@size pt\relax}%
  \newcommand*\lineheight[1]{\fontsize{\fsize}{#1\fsize}\selectfont}%
  \ifx\svgwidth\undefined%
    \setlength{\unitlength}{131.02873246bp}%
    \ifx\svgscale\undefined%
      \relax%
    \else%
      \setlength{\unitlength}{\unitlength * \real{\svgscale}}%
    \fi%
  \else%
    \setlength{\unitlength}{\svgwidth}%
  \fi%
  \global\let\svgwidth\undefined%
  \global\let\svgscale\undefined%
  \makeatother%
  \begin{picture}(1,0.200038)%
    \lineheight{1}%
    \setlength\tabcolsep{0pt}%
    \put(0,0){\includegraphics[width=\unitlength,page=1]{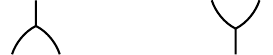}}%
    \put(0.22094468,0.03543583){\color[rgb]{0,0,0}\makebox(0,0)[lt]{\lineheight{1.25}\smash{\begin{tabular}[t]{l}$a$\end{tabular}}}}%
    \put(0,0){\includegraphics[width=\unitlength,page=2]{labelsvertices.pdf}}%
    \put(-0.00253403,0.03853651){\color[rgb]{0,0,0}\makebox(0,0)[lt]{\lineheight{1.25}\smash{\begin{tabular}[t]{l}$b$\end{tabular}}}}%
    \put(0.15589647,0.15097115){\color[rgb]{0,0,0}\makebox(0,0)[lt]{\lineheight{1.25}\smash{\begin{tabular}[t]{l}$a+b$\end{tabular}}}}%
    \put(0.89040834,0.03240381){\color[rgb]{0,0,0}\makebox(0,0)[lt]{\lineheight{1.25}\smash{\begin{tabular}[t]{l}$a+b$\end{tabular}}}}%
    \put(0.94498165,0.13461708){\color[rgb]{0,0,0}\makebox(0,0)[lt]{\lineheight{1.25}\smash{\begin{tabular}[t]{l}$a$\end{tabular}}}}%
    \put(0.72795451,0.12848422){\color[rgb]{0,0,0}\makebox(0,0)[lt]{\lineheight{1.25}\smash{\begin{tabular}[t]{l}$b$\end{tabular}}}}%
  \end{picture}%
\endgroup%
}}
    \end{equation*}   
    Sometimes it will be convenient to also allow edges labeled by $0$. A graph having some $0$-labeled edges is equivalent to the same graph with those edges removed.
\end{enumerate}
\end{dfn}

The labeling of the edges of a vinyl graph $\Gamma$ induces a labeling of the edges of $\Gamma_{\text{\tiny open}}$, where the labels on the top and bottom strands agree. We call the graph $\Gamma_{\text{\tiny open}}$ equipped with this labeling an \emph{open vinyl graph}. The \emph{level} of $\Gamma$ is the sum of the labels of all edges at a fixed height. The graphs in \cref{fig:carreequalsgammakl}, for instance, are both of level $k+\ell+1$. We consider vinyl graphs up to planar isotopy. 
\\\\
Let $R$ denote any of $\Q, \,\Z$ or $\mathbb{F}_p$, with $p$ prime.
\begin{dfn}
A \emph{decoration} of a (open or closed) vinyl graph $\Gamma$ of level $n$ is a choice of a homogeneous symmetric polynomial in $\ell(e)$ variables for each edge $e$ of $\Gamma$. In other words, it is a map 
\[
d\thinspace\colon E(\Gamma) \to R[x_1,\ldots,x_n], \qquad d(e) \in R[x_1,\ldots,x_{\ell(e)}]^{S_{\ell(e)}} \; \text{ homogeneous}. 
\]
We call a vinyl graph $\Gamma$ equipped with a decoration $d$ a \emph{decorated graph of shape $\Gamma$}, and denote it by $\Gamma^d$.  
\end{dfn}

The space generated by decorated graphs of shape $\Gamma$ is called $\mathscr{D}(\Gamma)$. 
It has the structure of a commutative $R$-algebra, which can be deduced from the following:
\[
\mathscr{D}(\Gamma) = \bigotimes_{e \text{ edge of } \Gamma} R[x_1,\ldots,x_{\ell(e)}]^{S_{\ell(e)}}.
\]

\begin{remark}[Notation] 

\begin{itemize}
    \item A decoration on a graph will be represented by a colored bullet ($\color{blue} \bullet$) labeled by a homogeneous symmetric polynomial on each edge. When the decoration on an edge $e$ is \emph{trivial}, i.e.\ $d(e)=1$, it will be omitted.
    \item For edges labeled $1$, the only possible decorations can be (multiples of) powers of $x_1$. On these edges, a decoration by $x_1^n$ will simply be denoted by a bullet labeled $n$ ($\color{blue} \bullet^n$). 
    \item A hollow bullet on an edge $e$ means that the decoration on $e$ might be trivial.
    \item Lastly, bullets on the same edge can be ``multiplied" in the following sense: having two bullets labeled $p$ and $p'$ on the same edge is equivalent to having one bullet labeled $p \cdot p'$.
\end{itemize}
 
\end{remark}

A special kind of homogeneous symmetric polynomials will be particularly relevant in this paper: \emph{Schur polynomials}, whose definition and main properties can be found in \cite{macdonald}. This family of homogeneous symmetric polynomials are an integral linear basis for the space of all symmetric polynomials. We will denote the set of Young diagrams with at most $a$ columns and at most $b$ rows by $T(a,b)$, and the Schur polynomial corresponding to a Young diagram $\lambda$ by $s_{\lambda}$.  
Other than on $\lambda$, the polynomial $s_{\lambda}$ also depends on the set $T(a,b)$ we consider $\lambda$ to be part of: when $\lambda$ is viewed as an element of $T(a,b)$, it is implicit that $s_{\lambda}$ is a symmetric polynomial in $a$ variables.

\begin{remark} \label{rem:schur}
For a Young diagram $\lambda$, let $|\lambda|$ be the number of boxes it consists of. 
\begin{enumerate}
    \item When $\lambda \in T(a,1)$, one has $s_{\lambda}=e_i(x_1,\ldots,x_a)$, where $i=|\lambda|$  and $e_i$ is the $i$\textit{-th elementary polynomial}:
    \[
    e_i(x_1,\ldots,x_a)=\sum_{1\leq j_1<\ldots<j_i\leq a} x_{j_1} \cdots x_{j_i} \qquad \text{ for } 0\leq i\leq a.
    \]
    \item When $\lambda \in T(1,b)$, one has $0 \leq |\lambda| \leq b$ and $s_{\lambda}=x_1^{|\lambda|}$. 
\end{enumerate}

In this paper, all relevant instances of Schur polynomials will be of one of these two kinds.
\end{remark}

Next, we would like to define a pairing of decorated graphs. We start by introducing the notion of coloring of a graph.
\begin{dfn} \label{def:coloring}
A \emph{coloring} of a (open or closed) vinyl graph $\Gamma$ of level $n$ is a map
$$c\thinspace\colon E(\Gamma) \to \mathcal{P}(\{1, \ldots ,n\}), \qquad c(e)=\{i_1, \ldots ,i_{\ell(e)}\} \textnormal{ with } i_j \in \{1, \ldots ,n\}$$ 
which satisfies the two conditions below:
\begin{enumerate}
    \item Denote the edges around a vertex $v$ by $e_a,\, e_b,\, e_{a+b}$, with labels $a,\, b$ and $a+b$ respectively. At each vertex, $c(e_{a+b})=c(e_a) \cup c(e_b)$. Since, in particular, this implies that $c(e_a) \cap c(e_b) = \emptyset$, we can rephrase the condition as $c(e_{a+b})=c(e_a) \sqcup c(e_b)$.
    \item Let $e_1,\ldots,e_k$ be the edges at a fixed height of $\Gamma$. Then $\displaystyle\bigcup_i c(e_i) = \{1, \ldots ,n\}$.
\end{enumerate}
\end{dfn}

\begin{dfn} \label{def:evaldecoratedgraphs}
Let $\Gamma$ be a (open or closed) vinyl graph of level $n$, and $\Gamma^{d}$ a decorated graph. Fix a coloring $c$ of $\Gamma$.
\begin{itemize}
    \item For each edge $e$ of $\Gamma$, $d(e)$ is a symmetric polynomial $p_e(x_1,\ldots,x_{\ell(e)})$. If $c(e)=\{i_1,\ldots ,i_{\ell(e)}\}$, with $i_j \in \{1,\ldots , n\}$, let $d_c(e)=p_e(x_{i_1},\ldots,x_{i_{\ell(e)}})$. We define 
    \[
    P(\Gamma^{d},c)=\prod_{e\in E(\Gamma)} d_c(e), \qquad Q(\Gamma,c)=\displaystyle \prod_{\substack{v \text{ split}\\ \text{vertex}}} \prod_{\substack{i \in c(e^l(v)),\\ j \in c(e^r(v))}} (x_i-x_j).
    \]
    \item Let 
    $$\langle \Gamma^d, c \rangle_\infty=\frac{P(\Gamma^{d},c)}{Q(\Gamma,c)}.$$
\end{itemize}

We define the \emph{$\infty$-evaluation $\langle \Gamma^d \rangle_\infty$ of the decorated graph} $\Gamma^d$ as 
\[
\langle \Gamma^d \rangle_\infty=\sum_{c \text{ coloring of } \Gamma} \langle \Gamma^d, c \rangle_\infty. 
\]
\end{dfn}

From its definition it is clear that $\langle \Gamma^d \rangle_\infty \in R(x_1,\ldots,x_n)$. However, more is true: $\langle \Gamma^d \rangle_\infty$ turns out to be a homogeneous symmetric polynomial in $R[x_1,\ldots,x_n]$, for all decorated graphs $\Gamma^d$.

We define the \emph{$1$-evaluation} $\langle \Gamma^d \rangle_1$ of $\Gamma^d$ as the specialisation of $\langle \Gamma^d \rangle_\infty$ at $(0,\ldots ,0)$.

\begin{dfn} \label{def:pairing}
Given a vinyl graph $\Gamma$ of level $n$, we define the $R$-module map $(\cdot \, , \cdot)_1$ as 
\[(\cdot \, , \cdot)_1 \colon \thinspace \mathscr{D}(\Gamma) \otimes_R \mathscr{D}(\Gamma) \to R, \qquad \qquad
(\Gamma^d , \Gamma^{d'})_1 = \langle \Gamma^{d \cdot d'} \rangle_1.
\]
\end{dfn}

\begin{dfn} \label{def:statespaces}
Given a (closed) vinyl graph $\Gamma$, we define $\mathscr{S}_1(\Gamma)$ to be the $R$-module
\[
\mathscr{S}_1(\Gamma)=\mathscr{D}(\Gamma)/\bigcap_{\Gamma^d} \ker(\cdot \, , \Gamma^d)_1 .
\]
\end{dfn}
A grading on $\mathscr{S}_1(\Gamma)$ is defined as follows:
\begin{dfn} \label{def:degdecoratedgraph}
Let $\Gamma$ be a vinyl graph and define $t_{\Gamma}=\displaystyle\sum_{\substack{v \textnormal{ split, with} \\ \textnormal{labels } a,b,a+b}} ab$. Let $\Gamma^d$ be a decorated graph.
\begin{enumerate}
    \item The \emph{degree of the decoration} $d$ is
    \[
    \deg(d)=\sum_{e \in E(\Gamma)} \deg(d(e)).
    \]
    \item The \emph{degree of the decorated graph} $\Gamma^{d}$ is defined as
    \[
    \deg(\Gamma^{d}) = 2 \cdot \deg(d) - t_{\Gamma}.
    \]
\end{enumerate}
\end{dfn}

Given a graded $R$-module $M=\displaystyle\bigoplus_{i \in \Z} M_i$, its \emph{graded rank} is $\rk_q(M)=\displaystyle\sum_{i \in \Z} q^i\rk_{R}(M)$. Moreover, we denote by $q^jM$ the graded $R$-module with $(q^jM)_i=M_{i-j}$. More generally, given a Laurent polynomial $P(q)=\displaystyle\sum_{i \in \Z} a_iq^i$, with $a_i \in \N$, we define
\[
P(q)M = \bigoplus_{i \in \Z} q^i(\bigoplus_{j=1}^{a_i} M).
\]

Let $[n]=\frac{q^n-q^{-n}}{q-q^{-1}}\in \Z[q,q^{-1}]^{S_2}$ and $\left[ \substack {n \\ k} \right] = \frac{[n][n-1]\cdots[n-k+1]}{[k][k-1]\cdots[1]}$ be the \emph{quantum integers} and \emph{quantum binomials} respectively.

\begin{dfn}
For all vinyl graphs $\Gamma$, we define\footnote{This definition encodes the \emph{symmetric MOY calculus} \cite{zbMATH07206869} in the case $\mathfrak{gl}_1$. In the foamy setting, $\mathscr{S}_1(\cdot)$ is really a functor categorifying this calculus.} 
\[
\llangle \Gamma \rrangle_1 = \prod_{\substack{v \textnormal{ split, with} \\ \textnormal{labels } a,b,a+b}} \left[ \substack {a + b \\ a} \right]
\]
\end{dfn}

Then the following holds.
\begin{prop}[\cite{zbMATH07206869}, Theorem 5.14] \label{prop:S1}
Given a vinyl graph $\Gamma$, $\mathscr{S}_1(\Gamma)$ is a finitely generated, free, graded $R$-module of graded rank $\llangle \Gamma \rrangle_1$.
\end{prop}

\begin{remark} \label{rem:degreemaxdecoration}
Let $\Gamma^d$ be a decorated graph. Then, by the definition of the $\infty$-evaluation:
\begin{equation} \label{eq:degevalgraph}
\deg (d) - t_{\Gamma} = \deg(\langle \Gamma^d \rangle_{\infty} ).
\end{equation}
This is due to the fact that for each coloring $c$ of $\Gamma$, one has $\deg(P(\Gamma^d,c))=\deg(d)$ and $\deg(Q(\Gamma,c))=t_{\Gamma}$. As a consequence, since $\langle \Gamma^d \rangle_{\infty}$ is a homogeneous polynomial and the 1-evaluation is the specialisation of the $\infty$-evaluation at $(0,\ldots,0)$:
\begin{equation} \label{eq:degreeeval1}
\deg(d) \ne t_{\Gamma} \quad\Rightarrow\quad \langle \Gamma^d \rangle_1 = 0.
\end{equation}
Moreover, since $\mathscr{S}_1(\Gamma)$ is obtained by quotienting by the pairing $(\cdot \, ,\cdot)_1$, either the class of $\Gamma^d$ in $\mathscr{S}_1(\Gamma)$ is zero, or it satisfies: 
\begin{equation} \label{eq:degreemaxdecoration}
0 \leq \deg(d) \leq t_{\Gamma}.
\end{equation} 
\end{remark}

\bigskip

In the rest of the section, we will state the main properties and relations that hold in $\mathscr{S}_1(\Gamma)$, and define some morphisms that will be used throughout this paper. We will represent graphs only locally, in the tangles where they differ from one another.

\begin{lem}[\cite{zbMATH07206869}, Theorem 5.12]  \label{lem:monoidal}
Let $\Gamma,\, \Theta$ be vinyl graphs. Then
\[
\mathscr{S}_1(\Gamma \sqcup \Theta) \simeq \mathscr{S}_1(\Gamma) \otimes_R \mathscr{S}_1(\Theta).
\]
\end{lem}

\begin{lem}[\cite{zbMATH07206869}]  \label{lem:dotmigration}
Inside $\mathscr{S}_1(\Gamma)$ the following (local) equality of decorated graphs holds, where $c^{\lambda}_{\mu\nu}$ are the Littlewood--Richardson coefficients (see \cite{macdonald}):
\begin{equation} \label{eq:dotmigration}
\raisebox{-1.2\baselineskip}{\def\svgwidth{150pt}
\begingroup%
  \makeatletter%
  \providecommand\color[2][]{%
    \errmessage{(Inkscape) Color is used for the text in Inkscape, but the package 'color.sty' is not loaded}%
    \renewcommand\color[2][]{}%
  }%
  \providecommand\transparent[1]{%
    \errmessage{(Inkscape) Transparency is used (non-zero) for the text in Inkscape, but the package 'transparent.sty' is not loaded}%
    \renewcommand\transparent[1]{}%
  }%
  \providecommand\rotatebox[2]{#2}%
  \newcommand*\fsize{\dimexpr\f@size pt\relax}%
  \newcommand*\lineheight[1]{\fontsize{\fsize}{#1\fsize}\selectfont}%
  \ifx\svgwidth\undefined%
    \setlength{\unitlength}{157.81683597bp}%
    \ifx\svgscale\undefined%
      \relax%
    \else%
      \setlength{\unitlength}{\unitlength * \real{\svgscale}}%
    \fi%
  \else%
    \setlength{\unitlength}{\svgwidth}%
  \fi%
  \global\let\svgwidth\undefined%
  \global\let\svgscale\undefined%
  \makeatother%
  \begin{picture}(1,0.26510798)%
    \lineheight{1}%
    \setlength\tabcolsep{0pt}%
    \put(0.01337543,0.24585109){\makebox(0,0)[lt]{\lineheight{1.25}\smash{\begin{tabular}[t]{l}\tiny $a$\end{tabular}}}}%
    \put(0,0){\includegraphics[width=\unitlength,page=1]{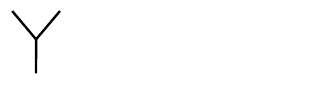}}%
    \put(0.1805483,0.2443541){\makebox(0,0)[lt]{\lineheight{1.25}\smash{\begin{tabular}[t]{l}\tiny $b$\end{tabular}}}}%
    \put(0.06482912,0.00703442){\makebox(0,0)[lt]{\lineheight{1.25}\smash{\begin{tabular}[t]{l}\tiny $a+b$\end{tabular}}}}%
    \put(0,0){\includegraphics[width=\unitlength,page=2]{dotmigration.pdf}}%
    \put(0.72195132,0.2440888){\makebox(0,0)[lt]{\lineheight{1.25}\smash{\begin{tabular}[t]{l}\tiny $a$\end{tabular}}}}%
    \put(0,0){\includegraphics[width=\unitlength,page=3]{dotmigration.pdf}}%
    \put(0.88912433,0.24259181){\makebox(0,0)[lt]{\lineheight{1.25}\smash{\begin{tabular}[t]{l}\tiny $b$\end{tabular}}}}%
    \put(0.77340515,0.00527213){\makebox(0,0)[lt]{\lineheight{1.25}\smash{\begin{tabular}[t]{l}\tiny $a+b$\end{tabular}}}}%
    \put(0,0){\includegraphics[width=\unitlength,page=4]{dotmigration.pdf}}%
    \put(0.13138394,0.0983362){\color[rgb]{0,0.53333333,0.63137255}\makebox(0,0)[lt]{\lineheight{1.25}\smash{\begin{tabular}[t]{l}$s_{\lambda}$\end{tabular}}}}%
    \put(0.72478003,0.14803761){\color[rgb]{0,0.33333333,0.83137255}\makebox(0,0)[lt]{\lineheight{1.25}\smash{\begin{tabular}[t]{l}$s_{\mu}$\end{tabular}}}}%
    \put(0.85051779,0.14330548){\color[rgb]{0,0.33333333,0.83137255}\makebox(0,0)[lt]{\lineheight{1.25}\smash{\begin{tabular}[t]{l}$s_{\nu}$\end{tabular}}}}%
    \put(0.24630543,0.12132047){\makebox(0,0)[lt]{\lineheight{1.25}\smash{\begin{tabular}[t]{l}$=\displaystyle\sum_{|\mu|+|\nu|=|\lambda|}c_{\mu\nu}^\lambda$\end{tabular}}}}%
    \put(0,0){\includegraphics[width=\unitlength,page=5]{dotmigration.pdf}}%
  \end{picture}%
\endgroup%
}
\end{equation}
The equality still holds for merge vertices (i.e.\ rotating both tangles by $\pi$).
We call this identity \emph{dot migration}, as it allows us to migrate decorations past vertices.
\end{lem}

In addition to the represented ones, other decorations $p_a,\, p_b$ and $p_{a+b}$ can be present on the edges of the decorated graphs of \cref{eq:dotmigration}. These simply get multiplied with the decorations $s_{\mu},\, s_{\nu}$ and $s_{\lambda}$ that are drawn.

Given two vinyl graphs $\Gamma$ and $\Theta$, let us discuss some ``elementary" $R$-module morphisms between $\mathscr{S}_1(\Gamma)$ and $\mathscr{S}_1(\Theta)$, that will come up often in this paper. Most of the maps we will encounter later on will be the composition of elementary ones, and all intermediate vinyl graphs will create a ``movie" (see \cref{eq:moviephi0alpha} for an example of a movie). Here is a description\footnote{The fact that these maps are indeed well-defined is not trivial, but can be checked easily. For more details see \cite{zbMATH07632766}.} of some of these elementary morphisms. They are the identity on the portion of the graphs that is outside the represented tangles. For all $i\in\N$, let $p_i$ be a homogeneous symmetric polynomial in $i$ variables.

\begin{itemize} 
    \item The map \emph{zip}: $\mathscr{S}_1 (\;\;\raisebox{-0.65\baselineskip}{\def\svgwidth{25pt}
\begingroup%
  \makeatletter%
  \providecommand\color[2][]{%
    \errmessage{(Inkscape) Color is used for the text in Inkscape, but the package 'color.sty' is not loaded}%
    \renewcommand\color[2][]{}%
  }%
  \providecommand\transparent[1]{%
    \errmessage{(Inkscape) Transparency is used (non-zero) for the text in Inkscape, but the package 'transparent.sty' is not loaded}%
    \renewcommand\transparent[1]{}%
  }%
  \providecommand\rotatebox[2]{#2}%
  \newcommand*\fsize{\dimexpr\f@size pt\relax}%
  \newcommand*\lineheight[1]{\fontsize{\fsize}{#1\fsize}\selectfont}%
  \ifx\svgwidth\undefined%
    \setlength{\unitlength}{88.88106094bp}%
    \ifx\svgscale\undefined%
      \relax%
    \else%
      \setlength{\unitlength}{\unitlength * \real{\svgscale}}%
    \fi%
  \else%
    \setlength{\unitlength}{\svgwidth}%
  \fi%
  \global\let\svgwidth\undefined%
  \global\let\svgscale\undefined%
  \makeatother%
  \begin{picture}(1,0.88151274)%
    \lineheight{1}%
    \setlength\tabcolsep{0pt}%
    \put(0,0){\includegraphics[width=\unitlength,page=1]{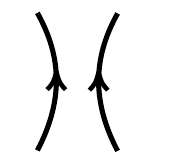}}%
    \put(0.6082359,0.44849416){\makebox(0,0)[lt]{\lineheight{1.25}\smash{\begin{tabular}[t]{l}\tiny $a$\end{tabular}}}}%
    \put(0.05875792,0.43643931){\makebox(0,0)[lt]{\lineheight{1.25}\smash{\begin{tabular}[t]{l}\tiny $b$\end{tabular}}}}%
    \put(0,0){\includegraphics[width=\unitlength,page=2]{twostrands.pdf}}%
  \end{picture}%
\endgroup%
}) \longrightarrow \mathscr{S}_1 (\;\;\raisebox{-0.65\baselineskip}{\def\svgwidth{26pt}
\begingroup%
  \makeatletter%
  \providecommand\color[2][]{%
    \errmessage{(Inkscape) Color is used for the text in Inkscape, but the package 'color.sty' is not loaded}%
    \renewcommand\color[2][]{}%
  }%
  \providecommand\transparent[1]{%
    \errmessage{(Inkscape) Transparency is used (non-zero) for the text in Inkscape, but the package 'transparent.sty' is not loaded}%
    \renewcommand\transparent[1]{}%
  }%
  \providecommand\rotatebox[2]{#2}%
  \newcommand*\fsize{\dimexpr\f@size pt\relax}%
  \newcommand*\lineheight[1]{\fontsize{\fsize}{#1\fsize}\selectfont}%
  \ifx\svgwidth\undefined%
    \setlength{\unitlength}{91.57316269bp}%
    \ifx\svgscale\undefined%
      \relax%
    \else%
      \setlength{\unitlength}{\unitlength * \real{\svgscale}}%
    \fi%
  \else%
    \setlength{\unitlength}{\svgwidth}%
  \fi%
  \global\let\svgwidth\undefined%
  \global\let\svgscale\undefined%
  \makeatother%
  \begin{picture}(1,0.8166754)%
    \lineheight{1}%
    \setlength\tabcolsep{0pt}%
    \put(0,0){\includegraphics[width=\unitlength,page=1]{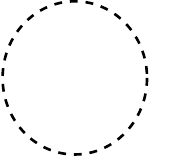}}%
    \put(0.48769331,0.37608225){\makebox(0,0)[lt]{\lineheight{1.25}\smash{\begin{tabular}[t]{l}\tiny $a+b$\end{tabular}}}}%
    \put(0.61631075,0.61301214){\makebox(0,0)[lt]{\lineheight{1.25}\smash{\begin{tabular}[t]{l}\tiny $a$\end{tabular}}}}%
    \put(0,0.61593708){\makebox(0,0)[lt]{\lineheight{1.25}\smash{\begin{tabular}[t]{l}\tiny $b$\end{tabular}}}}%
    \put(0,0){\includegraphics[width=\unitlength,page=2]{dumble.pdf}}%
    \put(0.61975313,0.14036107){\makebox(0,0)[lt]{\lineheight{1.25}\smash{\begin{tabular}[t]{l}\tiny $a$\end{tabular}}}}%
    \put(0.00256478,0.14182389){\makebox(0,0)[lt]{\lineheight{1.25}\smash{\begin{tabular}[t]{l}\tiny $b$\end{tabular}}}}%
    \put(0,0){\includegraphics[width=\unitlength,page=3]{dumble.pdf}}%
  \end{picture}%
\endgroup%
})$ is defined as
    \begin{equation} \label{eq:zip}
    {\small \raisebox{-1.1\baselineskip}{\def\svgwidth{230pt}
\begingroup%
  \makeatletter%
  \providecommand\color[2][]{%
    \errmessage{(Inkscape) Color is used for the text in Inkscape, but the package 'color.sty' is not loaded}%
    \renewcommand\color[2][]{}%
  }%
  \providecommand\transparent[1]{%
    \errmessage{(Inkscape) Transparency is used (non-zero) for the text in Inkscape, but the package 'transparent.sty' is not loaded}%
    \renewcommand\transparent[1]{}%
  }%
  \providecommand\rotatebox[2]{#2}%
  \newcommand*\fsize{\dimexpr\f@size pt\relax}%
  \newcommand*\lineheight[1]{\fontsize{\fsize}{#1\fsize}\selectfont}%
  \ifx\svgwidth\undefined%
    \setlength{\unitlength}{348.15769295bp}%
    \ifx\svgscale\undefined%
      \relax%
    \else%
      \setlength{\unitlength}{\unitlength * \real{\svgscale}}%
    \fi%
  \else%
    \setlength{\unitlength}{\svgwidth}%
  \fi%
  \global\let\svgwidth\undefined%
  \global\let\svgscale\undefined%
  \makeatother%
  \begin{picture}(1,0.15171479)%
    \lineheight{1}%
    \setlength\tabcolsep{0pt}%
    \put(0,0){\includegraphics[width=\unitlength,page=1]{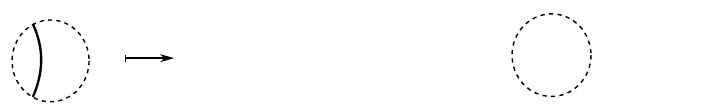}}%
    \put(0.17621824,0.09279196){\color[rgb]{0,0,0}\makebox(0,0)[lt]{\lineheight{1.25}\smash{\begin{tabular}[t]{l}$\textit{zip}$\end{tabular}}}}%
    \put(0.28805418,0.06983923){\color[rgb]{0,0,0}\makebox(0,0)[lt]{\lineheight{1.25}\smash{\begin{tabular}[t]{l}$\displaystyle\sum_{\lambda \in T(a,b)}(-1)^{|\widehat{\lambda}|}$\end{tabular}}}}%
    \put(0.79047917,0.10312331){\color[rgb]{0,0.33333333,0.83137255}\makebox(0,0)[lt]{\lineheight{1.25}\smash{\begin{tabular}[t]{l}$s_{\lambda}\cdot p_a$\end{tabular}}}}%
    \put(0.59443964,0.03651762){\color[rgb]{0,0.33333333,0.83137255}\makebox(0,0)[lt]{\lineheight{1.25}\smash{\begin{tabular}[t]{l}$s_{\widehat{\lambda}}\cdot p_b$\end{tabular}}}}%
    \put(0,0){\includegraphics[width=\unitlength,page=2]{zip.pdf}}%
    \put(0.09590229,0.12299614){\makebox(0,0)[lt]{\lineheight{1.25}\smash{\begin{tabular}[t]{l}\tiny $a$\end{tabular}}}}%
    \put(0.02732723,0.12162037){\makebox(0,0)[lt]{\lineheight{1.25}\smash{\begin{tabular}[t]{l}\tiny $b$\end{tabular}}}}%
    \put(0,0){\includegraphics[width=\unitlength,page=3]{zip.pdf}}%
    \put(0.71797594,0.13117923){\makebox(0,0)[lt]{\lineheight{1.25}\smash{\begin{tabular}[t]{l}\tiny $b$\end{tabular}}}}%
    \put(0.7203664,0.00477964){\makebox(0,0)[lt]{\lineheight{1.25}\smash{\begin{tabular}[t]{l}\tiny $b$\end{tabular}}}}%
    \put(0.78879418,0.13425677){\makebox(0,0)[lt]{\lineheight{1.25}\smash{\begin{tabular}[t]{l}\tiny $a$\end{tabular}}}}%
    \put(0.78851421,0.00541696){\makebox(0,0)[lt]{\lineheight{1.25}\smash{\begin{tabular}[t]{l}\tiny $a$\end{tabular}}}}%
    \put(0.77016556,0.07040015){\makebox(0,0)[lt]{\lineheight{1.25}\smash{\begin{tabular}[t]{l}\tiny $a+b$\end{tabular}}}}%
    \put(0,0){\includegraphics[width=\unitlength,page=4]{zip.pdf}}%
    \put(0.10160423,0.03825641){\color[rgb]{0,0.53333333,0.63137255}\makebox(0,0)[lt]{\lineheight{1.25}\smash{\begin{tabular}[t]{l}$p_a$\end{tabular}}}}%
    \put(-0.00143052,0.0524009){\color[rgb]{0,0.53333333,0.63137255}\makebox(0,0)[lt]{\lineheight{1.25}\smash{\begin{tabular}[t]{l}$p_b$\end{tabular}}}}%
    \put(0,0){\includegraphics[width=\unitlength,page=5]{zip.pdf}}%
  \end{picture}%
\endgroup%
}}
    \end{equation}
    where $\widehat{\lambda}$ is the Young diagram dual to $\lambda$.

    \item The map \emph{unzip}: $\mathscr{S}_1 (\;\;\raisebox{-0.65\baselineskip}{\def\svgwidth{26pt}}) \longrightarrow \mathscr{S}_1 (\;\;\raisebox{-0.65\baselineskip}{\def\svgwidth{25pt}})$ only changes the shape of the underlying graph, but preserves the decorations of the decorated graphs:
    \begin{equation} \label{eq:unzip}
    {\small \raisebox{-1.3\baselineskip}{\def\svgwidth{154pt}
\begingroup%
  \makeatletter%
  \providecommand\color[2][]{%
    \errmessage{(Inkscape) Color is used for the text in Inkscape, but the package 'color.sty' is not loaded}%
    \renewcommand\color[2][]{}%
  }%
  \providecommand\transparent[1]{%
    \errmessage{(Inkscape) Transparency is used (non-zero) for the text in Inkscape, but the package 'transparent.sty' is not loaded}%
    \renewcommand\transparent[1]{}%
  }%
  \providecommand\rotatebox[2]{#2}%
  \newcommand*\fsize{\dimexpr\f@size pt\relax}%
  \newcommand*\lineheight[1]{\fontsize{\fsize}{#1\fsize}\selectfont}%
  \ifx\svgwidth\undefined%
    \setlength{\unitlength}{223.65727051bp}%
    \ifx\svgscale\undefined%
      \relax%
    \else%
      \setlength{\unitlength}{\unitlength * \real{\svgscale}}%
    \fi%
  \else%
    \setlength{\unitlength}{\svgwidth}%
  \fi%
  \global\let\svgwidth\undefined%
  \global\let\svgscale\undefined%
  \makeatother%
  \begin{picture}(1,0.23517582)%
    \lineheight{1}%
    \setlength\tabcolsep{0pt}%
    \put(0,0){\includegraphics[width=\unitlength,page=1]{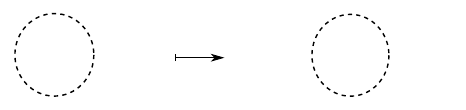}}%
    \put(0.35013818,0.14440581){\color[rgb]{0,0,0}\makebox(0,0)[lt]{\lineheight{1.25}\smash{\begin{tabular}[t]{l}$\textit{unzip}$\end{tabular}}}}%
    \put(0.1654664,0.15953555){\color[rgb]{0,0.53333333,0.63137255}\makebox(0,0)[lt]{\lineheight{1.25}\smash{\begin{tabular}[t]{l}$p_a$\end{tabular}}}}%
    \put(-0.00061144,0.15312987){\color[rgb]{0,0.53333333,0.63137255}\makebox(0,0)[lt]{\lineheight{1.25}\smash{\begin{tabular}[t]{l}$p_b$\end{tabular}}}}%
    \put(0,0){\includegraphics[width=\unitlength,page=2]{unzip.pdf}}%
    \put(0.05260391,0.203209){\makebox(0,0)[lt]{\lineheight{1.25}\smash{\begin{tabular}[t]{l}\tiny $b$\end{tabular}}}}%
    \put(0.04961844,0.00814188){\makebox(0,0)[lt]{\lineheight{1.25}\smash{\begin{tabular}[t]{l}\tiny $b$\end{tabular}}}}%
    \put(0.16284363,0.20799967){\makebox(0,0)[lt]{\lineheight{1.25}\smash{\begin{tabular}[t]{l}\tiny $a$\end{tabular}}}}%
    \put(0.16240773,0.00744026){\makebox(0,0)[lt]{\lineheight{1.25}\smash{\begin{tabular}[t]{l}\tiny $a$\end{tabular}}}}%
    \put(0.13384526,0.1085968){\makebox(0,0)[lt]{\lineheight{1.25}\smash{\begin{tabular}[t]{l}\tiny $a+b$\end{tabular}}}}%
    \put(0,0){\includegraphics[width=\unitlength,page=3]{unzip.pdf}}%
    \put(0.79605301,0.20593039){\makebox(0,0)[lt]{\lineheight{1.25}\smash{\begin{tabular}[t]{l}\tiny $a$\end{tabular}}}}%
    \put(0.68739947,0.19994011){\makebox(0,0)[lt]{\lineheight{1.25}\smash{\begin{tabular}[t]{l}\tiny $b$\end{tabular}}}}%
    \put(0,0){\includegraphics[width=\unitlength,page=4]{unzip.pdf}}%
    \put(0.16954588,0.05427216){\color[rgb]{0,0.53333333,0.63137255}\makebox(0,0)[lt]{\lineheight{1.25}\smash{\begin{tabular}[t]{l}$q_a$\end{tabular}}}}%
    \put(-0.00222683,0.04665102){\color[rgb]{0,0.53333333,0.63137255}\makebox(0,0)[lt]{\lineheight{1.25}\smash{\begin{tabular}[t]{l}$q_b$\end{tabular}}}}%
    \put(0.56163086,0.05617745){\color[rgb]{0,0.53333333,0.63137255}\makebox(0,0)[lt]{\lineheight{1.25}\smash{\begin{tabular}[t]{l}$p_b\cdot q_b$\end{tabular}}}}%
    \put(0.79805267,0.08570977){\color[rgb]{0,0.53333333,0.63137255}\makebox(0,0)[lt]{\lineheight{1.25}\smash{\begin{tabular}[t]{l}$p_a\cdot q_a$\end{tabular}}}}%
  \end{picture}%
\endgroup%
}}
    \end{equation}
    By \cref{lem:dotmigration}, if a decorated graph in the domain has a non-trivial decoration on the edge labeled $a+b$, this decoration can be pushed on two of the other four edges, so it is enough to define the \emph{unzip} map on decorated graphs with trivial decoration on that edge.

    \item For $0 \leq \alpha \leq k$, the map $\emph{bubble}_{\alpha}$: $\mathscr{S}_1 (\;\;\raisebox{-0.65\baselineskip}{\def\svgwidth{24pt}
\begingroup%
  \makeatletter%
  \providecommand\color[2][]{%
    \errmessage{(Inkscape) Color is used for the text in Inkscape, but the package 'color.sty' is not loaded}%
    \renewcommand\color[2][]{}%
  }%
  \providecommand\transparent[1]{%
    \errmessage{(Inkscape) Transparency is used (non-zero) for the text in Inkscape, but the package 'transparent.sty' is not loaded}%
    \renewcommand\transparent[1]{}%
  }%
  \providecommand\rotatebox[2]{#2}%
  \newcommand*\fsize{\dimexpr\f@size pt\relax}%
  \newcommand*\lineheight[1]{\fontsize{\fsize}{#1\fsize}\selectfont}%
  \ifx\svgwidth\undefined%
    \setlength{\unitlength}{105.04305612bp}%
    \ifx\svgscale\undefined%
      \relax%
    \else%
      \setlength{\unitlength}{\unitlength * \real{\svgscale}}%
    \fi%
  \else%
    \setlength{\unitlength}{\svgwidth}%
  \fi%
  \global\let\svgwidth\undefined%
  \global\let\svgscale\undefined%
  \makeatother%
  \begin{picture}(1,0.91738924)%
    \lineheight{1}%
    \setlength\tabcolsep{0pt}%
    \put(0.55881063,0.62745381){\makebox(0,0)[lt]{\lineheight{1.25}\smash{\begin{tabular}[t]{l}\tiny $k+1$\end{tabular}}}}%
    \put(0,0){\includegraphics[width=\unitlength,page=1]{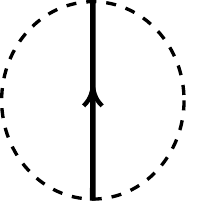}}%
  \end{picture}%
\endgroup%
}\;) \longrightarrow \mathscr{S}_1 (\;\;\raisebox{-0.65\baselineskip}{\def\svgwidth{26pt}
\begingroup%
  \makeatletter%
  \providecommand\color[2][]{%
    \errmessage{(Inkscape) Color is used for the text in Inkscape, but the package 'color.sty' is not loaded}%
    \renewcommand\color[2][]{}%
  }%
  \providecommand\transparent[1]{%
    \errmessage{(Inkscape) Transparency is used (non-zero) for the text in Inkscape, but the package 'transparent.sty' is not loaded}%
    \renewcommand\transparent[1]{}%
  }%
  \providecommand\rotatebox[2]{#2}%
  \newcommand*\fsize{\dimexpr\f@size pt\relax}%
  \newcommand*\lineheight[1]{\fontsize{\fsize}{#1\fsize}\selectfont}%
  \ifx\svgwidth\undefined%
    \setlength{\unitlength}{110.47261234bp}%
    \ifx\svgscale\undefined%
      \relax%
    \else%
      \setlength{\unitlength}{\unitlength * \real{\svgscale}}%
    \fi%
  \else%
    \setlength{\unitlength}{\svgwidth}%
  \fi%
  \global\let\svgwidth\undefined%
  \global\let\svgscale\undefined%
  \makeatother%
  \begin{picture}(1,0.87230418)%
    \lineheight{1}%
    \setlength\tabcolsep{0pt}%
    \put(0,0){\includegraphics[width=\unitlength,page=1]{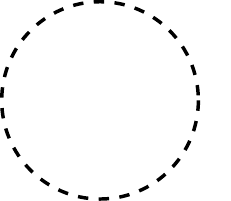}}%
    \put(0.51435908,0.08312927){\makebox(0,0)[lt]{\lineheight{1.25}\smash{\begin{tabular}[t]{l}\tiny $k+1$\end{tabular}}}}%
    \put(0.5303417,0.72465114){\makebox(0,0)[lt]{\lineheight{1.25}\smash{\begin{tabular}[t]{l}\tiny $k+1$\end{tabular}}}}%
    \put(0.68310781,0.40792159){\makebox(0,0)[lt]{\lineheight{1.25}\smash{\begin{tabular}[t]{l}\tiny $1$\end{tabular}}}}%
    \put(0.03203106,0.41130009){\makebox(0,0)[lt]{\lineheight{1.25}\smash{\begin{tabular}[t]{l}\tiny $k$\end{tabular}}}}%
    \put(0,0){\includegraphics[width=\unitlength,page=2]{bigon.pdf}}%
  \end{picture}%
\endgroup%
}\;)$, is defined as 
    \begin{equation} \label{eq:bubblemap}
    {\small \raisebox{-1.6\baselineskip}{\def\svgwidth{145pt}
\begingroup%
  \makeatletter%
  \providecommand\color[2][]{%
    \errmessage{(Inkscape) Color is used for the text in Inkscape, but the package 'color.sty' is not loaded}%
    \renewcommand\color[2][]{}%
  }%
  \providecommand\transparent[1]{%
    \errmessage{(Inkscape) Transparency is used (non-zero) for the text in Inkscape, but the package 'transparent.sty' is not loaded}%
    \renewcommand\transparent[1]{}%
  }%
  \providecommand\rotatebox[2]{#2}%
  \newcommand*\fsize{\dimexpr\f@size pt\relax}%
  \newcommand*\lineheight[1]{\fontsize{\fsize}{#1\fsize}\selectfont}%
  \ifx\svgwidth\undefined%
    \setlength{\unitlength}{187.53747138bp}%
    \ifx\svgscale\undefined%
      \relax%
    \else%
      \setlength{\unitlength}{\unitlength * \real{\svgscale}}%
    \fi%
  \else%
    \setlength{\unitlength}{\svgwidth}%
  \fi%
  \global\let\svgwidth\undefined%
  \global\let\svgscale\undefined%
  \makeatother%
  \begin{picture}(1,0.27879566)%
    \lineheight{1}%
    \setlength\tabcolsep{0pt}%
    \put(0,0){\includegraphics[width=\unitlength,page=1]{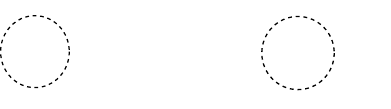}}%
    \put(0.11387894,0.07755285){\color[rgb]{0,0.53333333,0.63137255}\makebox(0,0)[lt]{\lineheight{1.25}\smash{\begin{tabular}[t]{l}$p_{k+1}$\end{tabular}}}}%
    \put(0,0){\includegraphics[width=\unitlength,page=2]{bubblemap.pdf}}%
    \put(0.28286403,0.17417609){\color[rgb]{0,0,0}\makebox(0,0)[lt]{\lineheight{1.25}\smash{\begin{tabular}[t]{l}$\textit{bubble}_{\alpha}$\end{tabular}}}}%
    \put(0.72408368,0.00887325){\makebox(0,0)[lt]{\lineheight{1.25}\smash{\begin{tabular}[t]{l}\tiny $k+1$\end{tabular}}}}%
    \put(0.71687313,0.24638538){\makebox(0,0)[lt]{\lineheight{1.25}\smash{\begin{tabular}[t]{l}\tiny $k+1$\end{tabular}}}}%
    \put(0.81332846,0.14016391){\makebox(0,0)[lt]{\lineheight{1.25}\smash{\begin{tabular}[t]{l}\tiny $1$\end{tabular}}}}%
    \put(0.67270849,0.10243869){\makebox(0,0)[lt]{\lineheight{1.25}\smash{\begin{tabular}[t]{l}\tiny $k$\end{tabular}}}}%
    \put(0,0){\includegraphics[width=\unitlength,page=3]{bubblemap.pdf}}%
    \put(0.11012888,0.186372){\makebox(0,0)[lt]{\lineheight{1.25}\smash{\begin{tabular}[t]{l}\tiny $k+1$\end{tabular}}}}%
    \put(0,0){\includegraphics[width=\unitlength,page=4]{bubblemap.pdf}}%
    \put(0.78439343,0.19152722){\color[rgb]{0,0.53333333,0.63137255}\makebox(0,0)[lt]{\lineheight{1.25}\smash{\begin{tabular}[t]{l}$p_{k+1}$\end{tabular}}}}%
    \put(0.81752307,0.09046258){\color[rgb]{0,0.33333333,0.83137255}\makebox(0,0)[lt]{\lineheight{1.25}\smash{\begin{tabular}[t]{l}$\alpha$\end{tabular}}}}%
    \put(0,0){\includegraphics[width=\unitlength,page=5]{bubblemap.pdf}}%
  \end{picture}%
\endgroup%
}}
    \end{equation}
    We will call \emph{bubble} $=$ $\emph{bubble}_{0}$.
    
    \item For $0 \leq \alpha \leq k$, the map $\emph{unbubble}_{\alpha}$: $\mathscr{S}_1 (\;\;\raisebox{-0.65\baselineskip}{\def\svgwidth{26pt}}\;) \longrightarrow \mathscr{S}_1 (\;\;\raisebox{-0.65\baselineskip}{\def\svgwidth{24pt}}\;)$ is defined as
    \begin{equation} \label{eq:unbubble}
    {\small \raisebox{-2.4\baselineskip}{\def\svgwidth{210pt}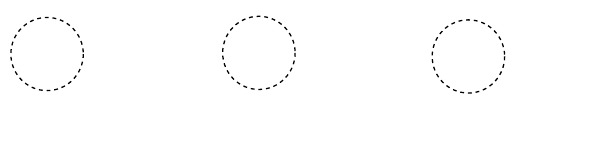}}
    \end{equation}
    where $\widetilde{p}_{\alpha}$ is the homogeneous symmetric polynomial
    \[
    \widetilde{p}_{\alpha}(x_1,\ldots,x_{k+1})=\sum_{i=1}^{k+1} \frac{p_k(x_1,\ldots,\widehat{x_i},\ldots,x_{k+1}) \cdot p_1(x_i) \cdot x_i^{\alpha}}{\displaystyle\prod_{\substack{j\in\{1,\ldots,k+1\} \\ j\ne i}} (x_j-x_i)}.
    \]
    We will call \emph{unbubble} $=$ $\emph{unbubble}_{0}$. A more general version of the bubble and unbubble maps exists, replacing the label $1$ by any $h\in \N$ and $k+1$ by $k+h$, but we won't need it.

    \item The \emph{tree isomorphism}: $\mathscr{S}_1 (\;\raisebox{-1\baselineskip}{\def\svgwidth{42pt}
\begingroup%
  \makeatletter%
  \providecommand\color[2][]{%
    \errmessage{(Inkscape) Color is used for the text in Inkscape, but the package 'color.sty' is not loaded}%
    \renewcommand\color[2][]{}%
  }%
  \providecommand\transparent[1]{%
    \errmessage{(Inkscape) Transparency is used (non-zero) for the text in Inkscape, but the package 'transparent.sty' is not loaded}%
    \renewcommand\transparent[1]{}%
  }%
  \providecommand\rotatebox[2]{#2}%
  \newcommand*\fsize{\dimexpr\f@size pt\relax}%
  \newcommand*\lineheight[1]{\fontsize{\fsize}{#1\fsize}\selectfont}%
  \ifx\svgwidth\undefined%
    \setlength{\unitlength}{58.12059785bp}%
    \ifx\svgscale\undefined%
      \relax%
    \else%
      \setlength{\unitlength}{\unitlength * \real{\svgscale}}%
    \fi%
  \else%
    \setlength{\unitlength}{\svgwidth}%
  \fi%
  \global\let\svgwidth\undefined%
  \global\let\svgscale\undefined%
  \makeatother%
  \begin{picture}(1,0.85444464)%
    \lineheight{1}%
    \setlength\tabcolsep{0pt}%
    \put(0.10203668,0.70298682){\makebox(0,0)[lt]{\lineheight{1.25}\smash{\begin{tabular}[t]{l}\tiny $a$\end{tabular}}}}%
    \put(0,0){\includegraphics[width=\unitlength,page=1]{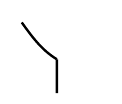}}%
    \put(0.46797916,0.80215574){\makebox(0,0)[lt]{\lineheight{1.25}\smash{\begin{tabular}[t]{l}\tiny $b$\end{tabular}}}}%
    \put(0.14743763,0.0143156){\makebox(0,0)[lt]{\lineheight{1.25}\smash{\begin{tabular}[t]{l}\tiny $a+b+c$\end{tabular}}}}%
    \put(0,0){\includegraphics[width=\unitlength,page=2]{treel.pdf}}%
    \put(0.70478275,0.69794488){\makebox(0,0)[lt]{\lineheight{1.25}\smash{\begin{tabular}[t]{l}\tiny $c$\end{tabular}}}}%
    \put(0,0.36087916){\makebox(0,0)[lt]{\lineheight{1.25}\smash{\begin{tabular}[t]{l}\tiny $a+b$\end{tabular}}}}%
    \put(0,0){\includegraphics[width=\unitlength,page=3]{treel.pdf}}%
  \end{picture}%
\endgroup%
}) \simeq \mathscr{S}_1 (\;\;\;\raisebox{-1\baselineskip}{\def\svgwidth{39.5pt}
\begingroup%
  \makeatletter%
  \providecommand\color[2][]{%
    \errmessage{(Inkscape) Color is used for the text in Inkscape, but the package 'color.sty' is not loaded}%
    \renewcommand\color[2][]{}%
  }%
  \providecommand\transparent[1]{%
    \errmessage{(Inkscape) Transparency is used (non-zero) for the text in Inkscape, but the package 'transparent.sty' is not loaded}%
    \renewcommand\transparent[1]{}%
  }%
  \providecommand\rotatebox[2]{#2}%
  \newcommand*\fsize{\dimexpr\f@size pt\relax}%
  \newcommand*\lineheight[1]{\fontsize{\fsize}{#1\fsize}\selectfont}%
  \ifx\svgwidth\undefined%
    \setlength{\unitlength}{55.16977617bp}%
    \ifx\svgscale\undefined%
      \relax%
    \else%
      \setlength{\unitlength}{\unitlength * \real{\svgscale}}%
    \fi%
  \else%
    \setlength{\unitlength}{\svgwidth}%
  \fi%
  \global\let\svgwidth\undefined%
  \global\let\svgscale\undefined%
  \makeatother%
  \begin{picture}(1,0.90014446)%
    \lineheight{1}%
    \setlength\tabcolsep{0pt}%
    \put(0.04762221,0.74058731){\makebox(0,0)[lt]{\lineheight{1.25}\smash{\begin{tabular}[t]{l}\tiny $a$\end{tabular}}}}%
    \put(0.29719344,0.84505882){\makebox(0,0)[lt]{\lineheight{1.25}\smash{\begin{tabular}[t]{l}\tiny $b$\end{tabular}}}}%
    \put(0.09545187,0.01508129){\makebox(0,0)[lt]{\lineheight{1.25}\smash{\begin{tabular}[t]{l}\tiny $a+b+c$\end{tabular}}}}%
    \put(0.68260687,0.73527569){\makebox(0,0)[lt]{\lineheight{1.25}\smash{\begin{tabular}[t]{l}\tiny $c$\end{tabular}}}}%
    \put(0.58222002,0.40372874){\makebox(0,0)[lt]{\lineheight{1.25}\smash{\begin{tabular}[t]{l}\tiny $b+c$\end{tabular}}}}%
    \put(0,0){\includegraphics[width=\unitlength,page=1]{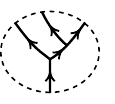}}%
  \end{picture}%
\endgroup%
})$ is an isomorphism of $R$-modules, realised by 
    \begin{equation} \label{eq:treeiso}
    {\small \raisebox{-1.7\baselineskip}{\def\svgwidth{170pt}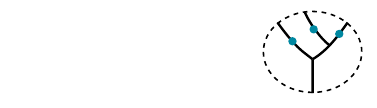}}
    \end{equation}    
    It still holds for vinyl graphs containing the tangles rotated by $\pi$ (i.e.\ with merge vertices instead of split ones). Again, by \cref{lem:dotmigration}, it is enough to define the tree isomorphism on decorated graphs with trivial decorations on the edges labeled $a+b$ (or $b+c$ for its inverse) and $a+b+c$. 
    
\end{itemize}

\begin{lem}[\cite{zbMATH07206869}, identity (31)] \label{lem:un-bubble}
There is an isomorphism of graded $R$-modules
\begin{equation} \label{eq:un-bubble}
\mathscr{S}_1 (\;\;\raisebox{-0.65\baselineskip}{\def\svgwidth{26pt}}\;) \quad\simeq\quad [k+1] \; \mathscr{S}_1 (\;\;\raisebox{-0.65\baselineskip}{\def\svgwidth{24pt}}\;)    
\end{equation}

realised by the morphisms $\arraycolsep=1.4pt\def\arraystretch{1.4} \left(\begin{array}{c}
    \textit{unbubble}_{0} \\
    \vdots       \\
    \textit{unbubble}_{k}
\end{array}\right)$ and its inverse $\left(\textit{bubble}_{k} \quad \cdots \quad \textit{bubble}_{0} \right).$
\end{lem}

\begin{lem}[\cite{zbMATH07206869}, identity (32)] \label{lem:id17general}
The following, as well as its mirror image, are isomorphisms of graded $R$-modules
\begin{equation} \label{eq:id17general}
\mathscr{S}_1(\;\;\raisebox{-1.1\baselineskip}{\def\svgwidth{35pt}
\begingroup%
  \makeatletter%
  \providecommand\color[2][]{%
    \errmessage{(Inkscape) Color is used for the text in Inkscape, but the package 'color.sty' is not loaded}%
    \renewcommand\color[2][]{}%
  }%
  \providecommand\transparent[1]{%
    \errmessage{(Inkscape) Transparency is used (non-zero) for the text in Inkscape, but the package 'transparent.sty' is not loaded}%
    \renewcommand\transparent[1]{}%
  }%
  \providecommand\rotatebox[2]{#2}%
  \newcommand*\fsize{\dimexpr\f@size pt\relax}%
  \newcommand*\lineheight[1]{\fontsize{\fsize}{#1\fsize}\selectfont}%
  \ifx\svgwidth\undefined%
    \setlength{\unitlength}{56.34317135bp}%
    \ifx\svgscale\undefined%
      \relax%
    \else%
      \setlength{\unitlength}{\unitlength * \real{\svgscale}}%
    \fi%
  \else%
    \setlength{\unitlength}{\svgwidth}%
  \fi%
  \global\let\svgwidth\undefined%
  \global\let\svgscale\undefined%
  \makeatother%
  \begin{picture}(1,0.90660172)%
    \lineheight{1}%
    \setlength\tabcolsep{0pt}%
    \put(0,0){\includegraphics[width=\unitlength,page=1]{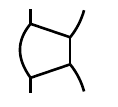}}%
    \put(0.00643734,0.85109416){\color[rgb]{0,0,0}\makebox(0,0)[lt]{\lineheight{1.25}\smash{\begin{tabular}[t]{l}\tiny $k+1$\end{tabular}}}}%
    \put(0.00713094,0.01595557){\color[rgb]{0,0,0}\makebox(0,0)[lt]{\lineheight{1.25}\smash{\begin{tabular}[t]{l}\tiny $k+1$\end{tabular}}}}%
    \put(0,0.43312184){\color[rgb]{0,0,0}\makebox(0,0)[lt]{\lineheight{1.25}\smash{\begin{tabular}[t]{l}\tiny $k$\end{tabular}}}}%
    \put(0.68933403,0.01476721){\color[rgb]{0,0,0}\makebox(0,0)[lt]{\lineheight{1.25}\smash{\begin{tabular}[t]{l}\tiny $1$\end{tabular}}}}%
    \put(0.67283822,0.85266329){\color[rgb]{0,0,0}\makebox(0,0)[lt]{\lineheight{1.25}\smash{\begin{tabular}[t]{l}\tiny $1$\end{tabular}}}}%
    \put(0.36644135,0.33209884){\color[rgb]{0,0,0}\makebox(0,0)[lt]{\lineheight{1.25}\smash{\begin{tabular}[t]{l}\tiny $1$\end{tabular}}}}%
    \put(0.33910636,0.54721834){\color[rgb]{0,0,0}\makebox(0,0)[lt]{\lineheight{1.25}\smash{\begin{tabular}[t]{l}\tiny $1$\end{tabular}}}}%
    \put(0.63742114,0.43074434){\color[rgb]{0,0,0}\makebox(0,0)[lt]{\lineheight{1.25}\smash{\begin{tabular}[t]{l}\tiny $2$\end{tabular}}}}%
    \put(0,0){\includegraphics[width=\unitlength,page=2]{square.pdf}}%
  \end{picture}%
\endgroup%
}) \quad \simeq \quad [k] \; \mathscr{S}_1(\raisebox{-0.8\baselineskip}{\def\svgwidth{32pt}
\begingroup%
  \makeatletter%
  \providecommand\color[2][]{%
    \errmessage{(Inkscape) Color is used for the text in Inkscape, but the package 'color.sty' is not loaded}%
    \renewcommand\color[2][]{}%
  }%
  \providecommand\transparent[1]{%
    \errmessage{(Inkscape) Transparency is used (non-zero) for the text in Inkscape, but the package 'transparent.sty' is not loaded}%
    \renewcommand\transparent[1]{}%
  }%
  \providecommand\rotatebox[2]{#2}%
  \newcommand*\fsize{\dimexpr\f@size pt\relax}%
  \newcommand*\lineheight[1]{\fontsize{\fsize}{#1\fsize}\selectfont}%
  \ifx\svgwidth\undefined%
    \setlength{\unitlength}{50.09681663bp}%
    \ifx\svgscale\undefined%
      \relax%
    \else%
      \setlength{\unitlength}{\unitlength * \real{\svgscale}}%
    \fi%
  \else%
    \setlength{\unitlength}{\svgwidth}%
  \fi%
  \global\let\svgwidth\undefined%
  \global\let\svgscale\undefined%
  \makeatother%
  \begin{picture}(1,0.90186219)%
    \lineheight{1}%
    \setlength\tabcolsep{0pt}%
    \put(0,0){\includegraphics[width=\unitlength,page=1]{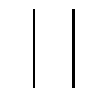}}%
    \put(0,0.8411984){\color[rgb]{0,0,0}\makebox(0,0)[lt]{\lineheight{1.25}\smash{\begin{tabular}[t]{l}\tiny $k+1$\end{tabular}}}}%
    \put(0.65059843,0.84066137){\color[rgb]{0,0,0}\makebox(0,0)[lt]{\lineheight{1.25}\smash{\begin{tabular}[t]{l}\tiny $1$\end{tabular}}}}%
    \put(0,0){\includegraphics[width=\unitlength,page=2]{twostrands_k+1.pdf}}%
  \end{picture}%
\endgroup%
}) \; \oplus \; \mathscr{S}_1(\;\raisebox{-1.1\baselineskip}{\def\svgwidth{30pt}
\begingroup%
  \makeatletter%
  \providecommand\color[2][]{%
    \errmessage{(Inkscape) Color is used for the text in Inkscape, but the package 'color.sty' is not loaded}%
    \renewcommand\color[2][]{}%
  }%
  \providecommand\transparent[1]{%
    \errmessage{(Inkscape) Transparency is used (non-zero) for the text in Inkscape, but the package 'transparent.sty' is not loaded}%
    \renewcommand\transparent[1]{}%
  }%
  \providecommand\rotatebox[2]{#2}%
  \newcommand*\fsize{\dimexpr\f@size pt\relax}%
  \newcommand*\lineheight[1]{\fontsize{\fsize}{#1\fsize}\selectfont}%
  \ifx\svgwidth\undefined%
    \setlength{\unitlength}{49.2042829bp}%
    \ifx\svgscale\undefined%
      \relax%
    \else%
      \setlength{\unitlength}{\unitlength * \real{\svgscale}}%
    \fi%
  \else%
    \setlength{\unitlength}{\svgwidth}%
  \fi%
  \global\let\svgwidth\undefined%
  \global\let\svgscale\undefined%
  \makeatother%
  \begin{picture}(1,1.01052298)%
    \lineheight{1}%
    \setlength\tabcolsep{0pt}%
    \put(0.00974256,0.94789204){\color[rgb]{0,0,0}\makebox(0,0)[lt]{\lineheight{1.25}\smash{\begin{tabular}[t]{l}\tiny $k+1$\end{tabular}}}}%
    \put(0,0.01690973){\color[rgb]{0,0,0}\makebox(0,0)[lt]{\lineheight{1.25}\smash{\begin{tabular}[t]{l}\tiny $k+1$\end{tabular}}}}%
    \put(0.61646793,0.0171418){\color[rgb]{0,0,0}\makebox(0,0)[lt]{\lineheight{1.25}\smash{\begin{tabular}[t]{l}\tiny $1$\end{tabular}}}}%
    \put(0.5896638,0.94875879){\color[rgb]{0,0,0}\makebox(0,0)[lt]{\lineheight{1.25}\smash{\begin{tabular}[t]{l}\tiny $1$\end{tabular}}}}%
    \put(0,0){\includegraphics[width=\unitlength,page=1]{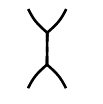}}%
    \put(0.52906791,0.4756937){\color[rgb]{0,0,0}\makebox(0,0)[lt]{\lineheight{1.25}\smash{\begin{tabular}[t]{l}\tiny $k+2$\end{tabular}}}}%
    \put(0,0){\includegraphics[width=\unitlength,page=2]{dumble_k+1.pdf}}%
  \end{picture}%
\endgroup%
}).
\end{equation}
\end{lem}

The maps realising isomorphism \cref{eq:id17general} are compositions of the elementary morphisms defined above, plus two maps $\varphi^{\alpha},\xi^{\alpha}$ that add decorations on some edges of a decorated graph. For $0 \leq \alpha \leq k-1$, let 
$$\varphi^{\alpha},\xi^{\alpha}\thinspace\colon \, \mathscr{S}_1(\;\;\raisebox{-1.1\baselineskip}{\def\svgwidth{33pt}}) \longrightarrow \mathscr{S}_1(\;\;\raisebox{-1.1\baselineskip}{\def\svgwidth{33pt}})$$ 
be the $R$-module morphisms given by 
\begin{equation*}
{\small \raisebox{-1.3\baselineskip}{\def\svgwidth{270pt}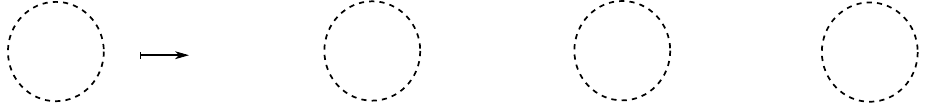}}
\end{equation*}

Then we define the morphisms of $R$-modules $\phi_0^{\alpha},\phi_1$ via the movies \cref{eq:moviephi0alpha} and \cref{eq:moviephi1} respectively (each graph $\Gamma$ stands for $\mathscr{S}_1(\Gamma)$). 
\begin{equation} \label{eq:moviephi0alpha}
{\small \raisebox{-1.5\baselineskip}{\def\svgwidth{260pt}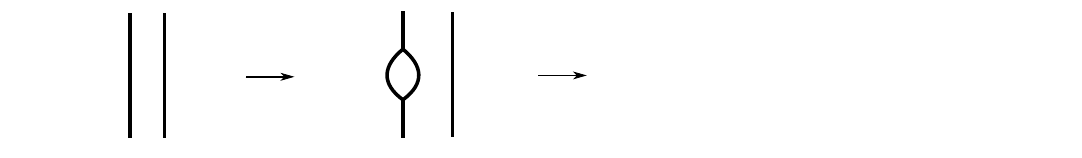}}
\end{equation}
\begin{equation} \label{eq:moviephi1}
{\small \raisebox{-1.5\baselineskip}{\def\svgwidth{260pt}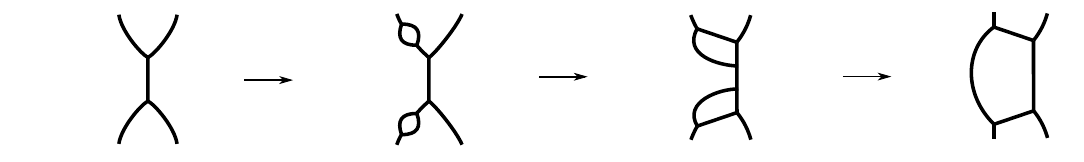}}
\end{equation}

Similarly, we define 
\[
\psi_0^{\alpha}= \textit{unbubble $\circ$ unzip $\circ$ $\xi^{\alpha}$}, \qquad  \psi_1= \textit{unbubble $\circ$ tree iso. $\circ$ zip},
\]

The isomorphism of \cref{lem:id17general} is then realised by the maps $\psi$ and its inverse $\phi$ defined as follows: 
$$\psi=\arraycolsep=1.4pt\def\arraystretch{1.4} \left(\begin{array}{c}
    \psi_0^{k-1} \\
    \psi_0^{k-2} \\
    \vdots       \\
    \psi_0^{0}   \\
    \psi_1
\end{array}\right), \qquad 
\phi=\left(\phi_0^0 \quad \phi_0^1 \quad \cdots \quad \phi_0^{k-1} \quad \phi_1 \right) .$$

\section{The d.u.r.\ set}

Given a vinyl graph $\Gamma$, we are interested in a special family of decorated graphs in $\mathscr{S}_1(\Gamma)$. First, order the split vertices of $\Gamma$ as $v_1,\ldots , v_n$ and consider the family of decorations $d$ on $\Gamma$ satisfying:
\begin{align*}
&d(e^r(v_i))=p_i \quad &&\text{for $i=1,\ldots,n$ and $p_i$ a homogeneous symmetric} \\ 
&&&\text{polynomial in $\ell(e^r(v_i))$ variables}, \\
&d(e)=1 \quad &&\text{if $e \ne e^r(v_i)$ for all $i$}.
\end{align*}
i.e. the only edges where $d$ is (possibly) non-trivial are the upper-right edges of split vertices. We call $\overline{\textit{DUR}}(\Gamma)$ the family of decorated graphs with decoration $d$ as above. In order to make the decorations of these graphs explicit, we denote the decorated graphs $\Gamma^d \in \overline{\textit{DUR}}(\Gamma)$ by $\gamma_{p_1,\ldots,p_n}(\Gamma)$, where $p_i=d(e^r(v_i))$. When the underlying vinyl graph $\Gamma$ is clear, we omit it and write $\gamma_{p_1,\ldots,p_n}$. Lastly, if $\ell(e^r(v_i))=1$ for all $i$, and $p_i=x^{k_i}$ for some $k_i \in \N$, we simply write $\gamma_{k_1,\ldots,k_n}(\Gamma)$ instead of $\gamma_{p_1,\ldots,p_n}(\Gamma)$.

\begin{dfn} \label{def:d.u.r.}
Given a vinyl graph $\Gamma$ with split vertices $v_1,\ldots,v_n$, let $a_i=\ell(e^r(v_i))$ and $b_i=\ell(e^l(v_i))$. The \emph{d.u.r.\ set} (d.u.r.\ = decoration on upper-right) of $\Gamma$, $\textit{DUR}(\Gamma)$, consists of all decorated graphs $\gamma_{p_1,\ldots,p_n}(\Gamma) \in \overline{\textit{DUR}}(\Gamma)$ such that, for all $i$, $p_i$ is a Schur polynomial $s_{\lambda_i}$ with $\lambda_i \in T(a_i,b_i)$.
\begin{equation*}
{\scriptsize \raisebox{-1\baselineskip}{\def\svgwidth{85pt}
\begingroup%
  \makeatletter%
  \providecommand\color[2][]{%
    \errmessage{(Inkscape) Color is used for the text in Inkscape, but the package 'color.sty' is not loaded}%
    \renewcommand\color[2][]{}%
  }%
  \providecommand\transparent[1]{%
    \errmessage{(Inkscape) Transparency is used (non-zero) for the text in Inkscape, but the package 'transparent.sty' is not loaded}%
    \renewcommand\transparent[1]{}%
  }%
  \providecommand\rotatebox[2]{#2}%
  \newcommand*\fsize{\dimexpr\f@size pt\relax}%
  \newcommand*\lineheight[1]{\fontsize{\fsize}{#1\fsize}\selectfont}%
  \ifx\svgwidth\undefined%
    \setlength{\unitlength}{124.04465576bp}%
    \ifx\svgscale\undefined%
      \relax%
    \else%
      \setlength{\unitlength}{\unitlength * \real{\svgscale}}%
    \fi%
  \else%
    \setlength{\unitlength}{\svgwidth}%
  \fi%
  \global\let\svgwidth\undefined%
  \global\let\svgscale\undefined%
  \makeatother%
  \begin{picture}(1,0.34496755)%
    \lineheight{1}%
    \setlength\tabcolsep{0pt}%
    \put(-0.00535343,0.02256542){\makebox(0,0)[lt]{\lineheight{1.25}\smash{\begin{tabular}[t]{l}$a_i+b_i$\end{tabular}}}}%
    \put(0.4243218,0.28801519){\makebox(0,0)[lt]{\lineheight{1.25}\smash{\begin{tabular}[t]{l}$a_i$\end{tabular}}}}%
    \put(0.01192683,0.27973448){\makebox(0,0)[lt]{\lineheight{1.25}\smash{\begin{tabular}[t]{l}$b_i$\end{tabular}}}}%
    \put(0,0){\includegraphics[width=\unitlength,page=1]{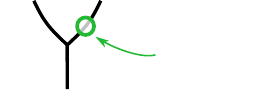}}%
    \put(0.62390715,0.11585143){\color[rgb]{0.14117647,0.7254902,0.19607843}\makebox(0,0)[lt]{\lineheight{1.25}\smash{\begin{tabular}[t]{l}\small $s_{\lambda_i}$\end{tabular}}}}%
  \end{picture}%
\endgroup%
}}    
\end{equation*}
We call \emph{d.u.r.\ decorated graphs} the elements of $\textit{DUR}(\Gamma)$.
\end{dfn}

The d.u.r.\ set behaves well with respect to disjoint union: the d.u.r set of $\Gamma\sqcup\Theta$ is canonically isomorphic to the Cartesian product of the d.u.r sets of $\Gamma$ and $\Theta$.

\begin{dfn} \label{def:elementary}
An \emph{elementary (vinyl) graph} is a vinyl graph $\Gamma$ having only $1$- and $2$-labeled edges, i.e.\ satisfying $\ell(E(\Gamma)) \subseteq \{1,2\}$.
\end{dfn}

Sometimes, when the labels on a vinyl graph are small, it will be convenient to represent edges labeled $\ell$ by $\ell$-uple strands, rather than writing down the label. See \cref{eq:twodumbles} for an example.

\begin{example} \label{ex:d.u.r.}
Consider the elementary graph $\Gamma$ obtained as the closure of the graph $\Gamma_{\text{open}}$ of \cref{eq:twodumbles}. 
\begin{equation} \label{eq:twodumbles}
{\small \raisebox{-2\baselineskip}{\def\svgwidth{210pt}
\begingroup%
  \makeatletter%
  \providecommand\color[2][]{%
    \errmessage{(Inkscape) Color is used for the text in Inkscape, but the package 'color.sty' is not loaded}%
    \renewcommand\color[2][]{}%
  }%
  \providecommand\transparent[1]{%
    \errmessage{(Inkscape) Transparency is used (non-zero) for the text in Inkscape, but the package 'transparent.sty' is not loaded}%
    \renewcommand\transparent[1]{}%
  }%
  \providecommand\rotatebox[2]{#2}%
  \newcommand*\fsize{\dimexpr\f@size pt\relax}%
  \newcommand*\lineheight[1]{\fontsize{\fsize}{#1\fsize}\selectfont}%
  \ifx\svgwidth\undefined%
    \setlength{\unitlength}{226.45982268bp}%
    \ifx\svgscale\undefined%
      \relax%
    \else%
      \setlength{\unitlength}{\unitlength * \real{\svgscale}}%
    \fi%
  \else%
    \setlength{\unitlength}{\svgwidth}%
  \fi%
  \global\let\svgwidth\undefined%
  \global\let\svgscale\undefined%
  \makeatother%
  \begin{picture}(1,0.23085383)%
    \lineheight{1}%
    \setlength\tabcolsep{0pt}%
    \put(0.55559054,0.10840195){\color[rgb]{0,0,0}\makebox(0,0)[lt]{\lineheight{1.25}\smash{\begin{tabular}[t]{l}$\gamma_{i,j}=$\end{tabular}}}}%
    \put(0,0){\includegraphics[width=\unitlength,page=1]{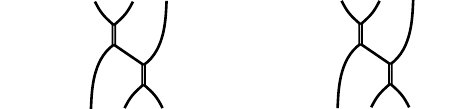}}%
    \put(0.93119295,0.18551604){\color[rgb]{0.04705882,0.04705882,0.69019608}\makebox(0,0)[lt]{\lineheight{1.25}\smash{\begin{tabular}[t]{l}$i=0,1$\end{tabular}}}}%
    \put(0,0){\includegraphics[width=\unitlength,page=2]{twodumbles.pdf}}%
    \put(0.92975875,0.04222704){\color[rgb]{0.04705882,0.04705882,0.69019608}\makebox(0,0)[lt]{\lineheight{1.25}\smash{\begin{tabular}[t]{l}$j=0,1$\end{tabular}}}}%
    \put(-0.00146618,0.10806579){\color[rgb]{0,0,0}\makebox(0,0)[lt]{\lineheight{1.25}\smash{\begin{tabular}[t]{l}$\Gamma_{\text{open}}=$\end{tabular}}}}%
    \put(0,0){\includegraphics[width=\unitlength,page=3]{twodumbles.pdf}}%
  \end{picture}%
\endgroup%
}}
\end{equation}

The decorated graphs $\gamma_{0,0}, \gamma_{0,1}, \gamma_{1,0}$ and $\gamma_{1,1}$ constitute the d.u.r.\ set of $\Gamma$.
\end{example}

Let $\Gamma$ be an elementary graph. By \cref{prop:S1}, $\mathscr{S}_{1}(\Gamma)$ is a finitely generated, free, graded $R$-module of graded rank $[2]^n$, where $n$ is the number of split vertices of $\Gamma$. Observe that the cardinality of the d.u.r.\ set of $\Gamma$ is $2^n$, i.e.\ the (ungraded) rank of the $R$-module $\mathscr{S}_{1}(\Gamma)$. Moreover, the degrees of the decorated graphs of $\textit{DUR}(\Gamma)$ correspond precisely to those of $\mathscr{S}_{1}(\Gamma)$. This hints to the fact that the d.u.r.\ set might be a good candidate for a basis of $\mathscr{S}_{1}(\Gamma)$. The goal of the next two sections will be to prove the following result.

\begin{theorem} \label{thm:d.u.r.}
Let $\Gamma$ be an elementary graph\footnote{It seems reasonable to expect that the theorem holds for all vinyl graphs, and not just elementary ones. However, in this setting, the proof would require working with general Schur polynomials, rather than elementary ones.}. The set $\textit{DUR}(\Gamma)$ is a basis for the graded $R$-module $\mathscr{S}_{1}(\Gamma)$.
\end{theorem}

A vinyl graph with the property that the d.u.r.\ set is a basis will be called a \emph{d.u.r.\ graph}.
From now on, let $R=\Z$. We will prove the theorem in this setting, as the results for $R=\Q$ and $R=\mathbb{F}_p$, with $p$ prime, follow immediately. The proof will be divided into 2 parts:
\begin{enumerate}
    \item In \cref{sec:basisgammakl} we show by induction over $\ell \geq 0$ that the vinyl graph $\Theta_{\ell}$, given by the closure of the graph represented in \cref{fig:basisthetal}, is d.u.r.
    \item In \cref{sec:basisgeneral} we deduce from this and \cref{lem:moveselementary} that all elementary graphs are d.u.r.
\end{enumerate}

\begin{remark} \label{rem:unionstrandsisdur}
The vinyl graphs $\Gamma$ consisting of the closure of vertical strands (without any trivalent vertex) are clearly d.u.r.: by \cref{eq:degreemaxdecoration}, the only (up to taking multiples in $R$) non-zero decorated graph in $\mathscr{S}_1(\Gamma)$ is the trivially decorated one, which also constitutes its d.u.r.\ set. 
\end{remark}

\begin{lem} \label{lem:un-bubblepreservesdur}
A vinyl graph containing the tangle $\raisebox{-0.65\baselineskip}{\def\svgwidth{26pt}}$ is d.u.r.\ if and only if the graph obtained by replacing $\raisebox{-0.65\baselineskip}{\def\svgwidth{26pt}}$ with $\;\;\raisebox{-0.65\baselineskip}{\def\svgwidth{24pt}}$ is d.u.r.
\end{lem}

\begin{proof}
The two graphs are respectively the LHS and RHS of identity \cref{eq:un-bubble}. It is easy to check that the isomorphism $\left(\textit{bubble}_{k} \quad \cdots \quad \textit{bubble}_{0} \right)$ realising identity \cref{eq:un-bubble} maps the d.u.r.\ set to the d.u.r.\ set. This implies that if $\;\;\raisebox{-0.65\baselineskip}{\def\svgwidth{24pt}}$ is d.u.r., then so is $\raisebox{-0.65\baselineskip}{\def\svgwidth{26pt}}$. 
Conversely, let $\raisebox{-0.65\baselineskip}{\def\svgwidth{26pt}}$ be d.u.r., with d.u.r.\ set locally given by $\biggl\{ \gamma_0(\;\;\raisebox{-0.65\baselineskip}{\def\svgwidth{26pt}}\;),\ldots,\gamma_k(\;\;\raisebox{-0.65\baselineskip}{\def\svgwidth{26pt}}\;) \biggr\}$. The map \emph{unbubble} is surjective and maps $\gamma_k(\;\;\raisebox{-0.65\baselineskip}{\def\svgwidth{26pt}}\;)$ to $\gamma(\;\;\raisebox{-0.65\baselineskip}{\def\svgwidth{24pt}}\;)$ and $\gamma_i(\;\;\raisebox{-0.65\baselineskip}{\def\svgwidth{26pt}}\;)$ to $0$ for all $i<k$. Therefore $\gamma(\;\;\raisebox{-0.65\baselineskip}{\def\svgwidth{24pt}}\;)$, which constitutes (locally) the d.u.r.\ set of $\;\;\raisebox{-0.65\baselineskip}{\def\svgwidth{24pt}}$, must be a basis.
\end{proof}

\section{Basis for $\Gamma_{k,\ell}$} \label{sec:basisgammakl}

In this section we aim at finding a basis over the integers of the symmetric $\mathfrak{gl}_1$-state spaces $\mathscr{S}_1(\Theta_\ell)$, for the vinyl graphs $\Theta_{\ell}$ on the LHS of \cref{fig:basisthetal} (with $\ell \in \N$). 

\begin{equation} \label{fig:basisthetal}
{\scriptsize \raisebox{-7\baselineskip}{\def\svgwidth{300pt}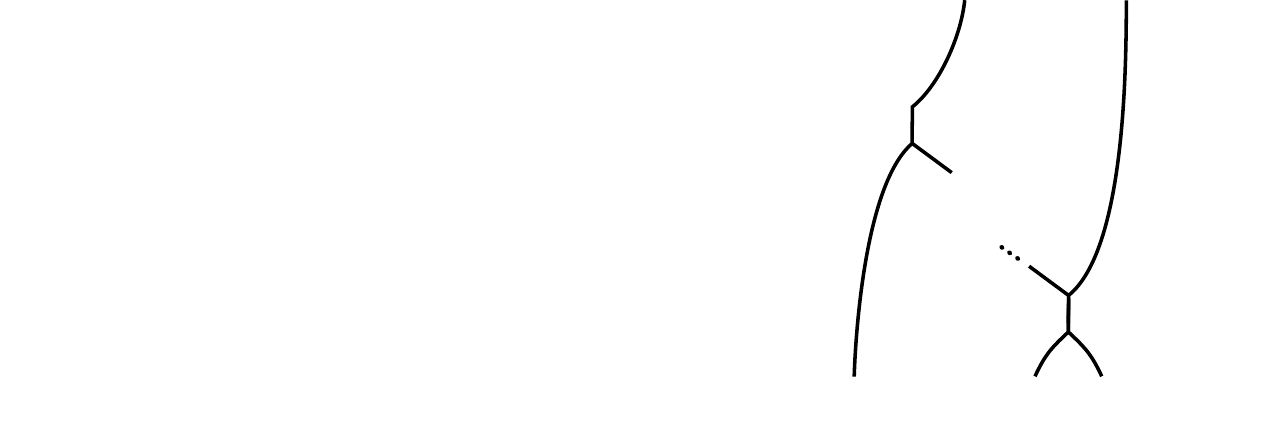}}
\end{equation}

We will actually prove a slightly stronger result:

\begin{prop} \label{prop:gammaklisdur}
Consider the vinyl graph $\Gamma_{k,\ell}$ on the LHS of \cref{fig:carreequalsgammakl}, for all $k,\ell\in \N$. The d.u.r.\ set is a basis for the graded $\Z$-module $\mathscr{S}_1(\Gamma_{k,\ell})$. 
\end{prop}

\begin{remark}
\begin{enumerate}
    \item All vinyl graphs in this section are closed, however they will be represented as open vinyl graphs for simplicity.
    \item The index $\ell$ in $\Theta_{\ell}$ and $\Gamma_{k,\ell}$ represents the number of split vertices of type ${\tiny \raisebox{-0.7\baselineskip}{\def\svgwidth{20pt}
\begingroup%
  \makeatletter%
  \providecommand\color[2][]{%
    \errmessage{(Inkscape) Color is used for the text in Inkscape, but the package 'color.sty' is not loaded}%
    \renewcommand\color[2][]{}%
  }%
  \providecommand\transparent[1]{%
    \errmessage{(Inkscape) Transparency is used (non-zero) for the text in Inkscape, but the package 'transparent.sty' is not loaded}%
    \renewcommand\transparent[1]{}%
  }%
  \providecommand\rotatebox[2]{#2}%
  \newcommand*\fsize{\dimexpr\f@size pt\relax}%
  \newcommand*\lineheight[1]{\fontsize{\fsize}{#1\fsize}\selectfont}%
  \ifx\svgwidth\undefined%
    \setlength{\unitlength}{30.20477935bp}%
    \ifx\svgscale\undefined%
      \relax%
    \else%
      \setlength{\unitlength}{\unitlength * \real{\svgscale}}%
    \fi%
  \else%
    \setlength{\unitlength}{\svgwidth}%
  \fi%
  \global\let\svgwidth\undefined%
  \global\let\svgscale\undefined%
  \makeatother%
  \begin{picture}(1,0.88593063)%
    \lineheight{1}%
    \setlength\tabcolsep{0pt}%
    \put(0,0){\includegraphics[width=\unitlength,page=1]{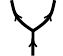}}%
    \put(-0.01648901,0.55929025){\color[rgb]{0,0,0}\makebox(0,0)[lt]{\lineheight{1.25}\smash{\begin{tabular}[t]{l}$1$\end{tabular}}}}%
    \put(0.52133101,0.68341105){\color[rgb]{0,0,0}\makebox(0,0)[lt]{\lineheight{1.25}\smash{\begin{tabular}[t]{l}$1$\end{tabular}}}}%
    \put(0.63733881,0.10337565){\color[rgb]{0,0,0}\makebox(0,0)[lt]{\lineheight{1.25}\smash{\begin{tabular}[t]{l}$2$\end{tabular}}}}%
  \end{picture}%
\endgroup%
}}$. Observe that $\Theta_{\ell}=\Gamma_{0,\ell}$, so \cref{prop:gammaklisdur} implies that $\Theta_{\ell}$ is d.u.r. 
\end{enumerate}
\end{remark}

\begin{remark} \label{rem:carreequalsgammakl}
We notice that $\Gamma_{k,\ell}$ is equal to the vinyl graph $\Gamma_{k,\ell}^{\textnormal{sq}}$ on the right-hand side of \cref{fig:carreequalsgammakl}: one simply has to slide along the annulus the portion of $\Gamma_{k,\ell}$ that is above the dotted line, until it reaches the bottom of the graph (remember that $\Gamma_{k,l}$ is closed).
\begin{equation} \label{fig:carreequalsgammakl}
{\scriptsize \raisebox{-6.75\baselineskip}{\def\svgwidth{316pt}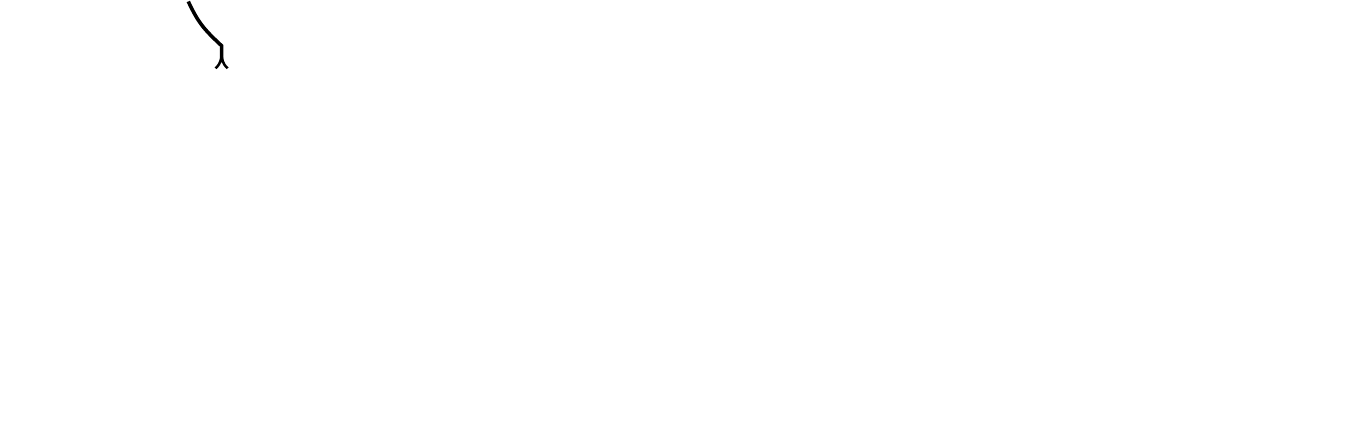}}
\end{equation}

\end{remark}

Let us examine in more detail the d.u.r.\ sets of $\Theta_{\ell}$ and $\Gamma_{k,\ell}$. $\textit{DUR}(\Theta_{\ell})$ consists of the decorated graphs $\theta_{\mathbf{a}}=\gamma_{a_1,\ldots,a_{\ell}}(\Theta_{\ell})$ (see \cref{fig:basisthetal}), with decorations determined by $\mathbf{a}=(a_1,\ldots,a_{\ell})\in \{0,1\}^{\ell}$. The degree of their decoration is $|\mathbf{a}|=\sum_{i=1}^{\ell} a_i$.

As for $\textit{DUR}(\Gamma_{k,\ell})$, it is made of the decorated graphs $\gamma_{i,\theta_{\mathbf{a}}}(\Gamma_{k,\ell})$ of \cref{fig:basisgammakl}, for $0 \leq i \leq k$ and $\mathbf{a} \in \{0,1\}^\ell$. By \cref{rem:carreequalsgammakl} these coincide with the decorated graphs $\gamma_{i,j,\theta_{\widetilde{\mathbf{a}}}}(\Gamma_{k,\ell}^{\textnormal{sq}})$, for $0 \leq i \leq k, \; 0 \leq j \leq 1 $ and $\widetilde{\mathbf{a}} \in \{0,1\}^{\ell-1}$, constituting  $\textit{DUR}(\Gamma_{k,\ell}^{\textnormal{sq}})$.
\begin{equation} \label{fig:basisgammakl}
{\scriptsize \raisebox{-7\baselineskip}{\def\svgwidth{370pt}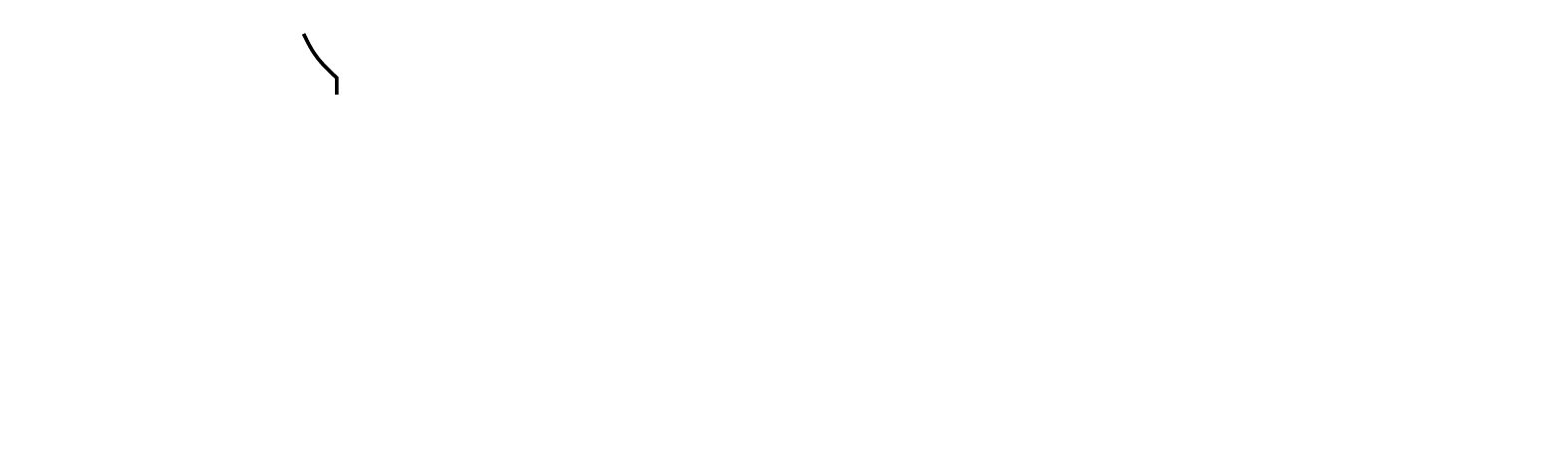}}
\end{equation}

The following result will allow us to express elements of $\overline{\textit{DUR}}(\Gamma_{k,\ell}^{\text{sq}})$ as linear combinations of the d.u.r.\ decorated graphs.
\begin{lem} \label{lem:rewritingdots}
Suppose that the graph $\Theta_{\ell}$ is d.u.r.\ for some $\ell$ and consider $\Gamma_{k,\ell+1}$. Let $r,s \in \N$ and $\mathbf{a} \in \{0,1\}^{\ell}$, and write $\gamma_{r,s,\theta_\mathbf{a}}=\gamma_{r,s,\theta_\mathbf{a}}(\Gamma_{k,\ell+1}^{\textnormal{sq}})$.
\begin{enumerate}
    \item We have 
    $$\gamma_{s,r+1,\theta_\mathbf{a}}-\gamma_{s+1,r,\theta_\mathbf{a}}=\sum_{|\widetilde{\mathbf{a}}|=|\mathbf{a}|+r} b_{\widetilde{\mathbf{a}}}(\gamma_{s,1,\theta_{\widetilde{\mathbf{a}}}}-\gamma_{s+1,0,\theta_{\widetilde{\mathbf{a}}}}),$$ 
    for $b_{\widetilde{\mathbf{a}}} \in \Z$ such that $b_{\widetilde{\mathbf{a}}}=0$ if $|\widetilde{\mathbf{a}}| >\ell$.
    
    \item As a consequence 
    $$\gamma_{0,r+1,\theta_\mathbf{a}} = \gamma_{r,1,\theta_{\mathbf{a}}} + \sum_{t=1}^m \sum_{|\widetilde{\mathbf{a}}|=|\mathbf{a}|+t} b_{\widetilde{\mathbf{a}}}(\gamma_{r-t,1,\theta_{\widetilde{\mathbf{a}}}}-\gamma_{r-t+1,0,\theta_{\widetilde{\mathbf{a}}}}),$$ 
    where $m=\min \{r,\ell-|\mathbf{a}|\}$.
\end{enumerate}
\end{lem}

\begin{proof}
We start by noticing, as shown in \cref{fig:zipgammakl}, that 
$$\gamma_{s,r+1,\theta_\mathbf{a}}-\gamma_{s+1,r,\theta_\mathbf{a}}=\textit{zip}(f_{s,r,\mathbf{a}}).$$ 
\begin{equation} \label{fig:zipgammakl}
{\scriptsize \raisebox{-7\baselineskip}{\def\svgwidth{370pt}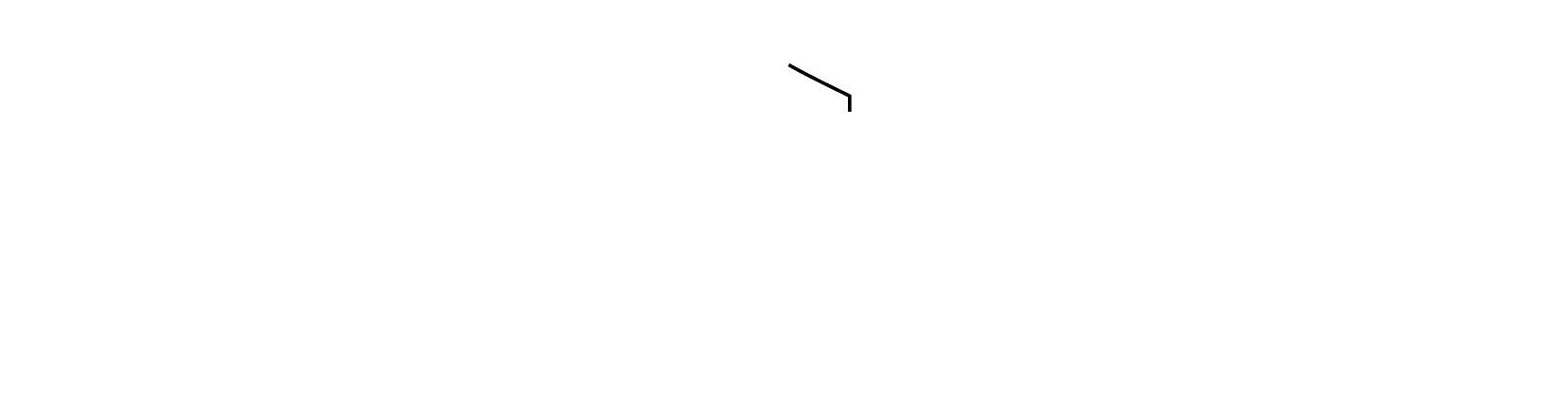}}
\end{equation}

The decorated graph $f_{s,r,\mathbf{a}}$ has shape $\raisebox{-0.65\baselineskip}{\def\svgwidth{25pt}
\begingroup%
  \makeatletter%
  \providecommand\color[2][]{%
    \errmessage{(Inkscape) Color is used for the text in Inkscape, but the package 'color.sty' is not loaded}%
    \renewcommand\color[2][]{}%
  }%
  \providecommand\transparent[1]{%
    \errmessage{(Inkscape) Transparency is used (non-zero) for the text in Inkscape, but the package 'transparent.sty' is not loaded}%
    \renewcommand\transparent[1]{}%
  }%
  \providecommand\rotatebox[2]{#2}%
  \newcommand*\fsize{\dimexpr\f@size pt\relax}%
  \newcommand*\lineheight[1]{\fontsize{\fsize}{#1\fsize}\selectfont}%
  \ifx\svgwidth\undefined%
    \setlength{\unitlength}{106.93405798bp}%
    \ifx\svgscale\undefined%
      \relax%
    \else%
      \setlength{\unitlength}{\unitlength * \real{\svgscale}}%
    \fi%
  \else%
    \setlength{\unitlength}{\svgwidth}%
  \fi%
  \global\let\svgwidth\undefined%
  \global\let\svgscale\undefined%
  \makeatother%
  \begin{picture}(1,0.89745272)%
    \lineheight{1}%
    \setlength\tabcolsep{0pt}%
    \put(0.49828874,0.08545245){\makebox(0,0)[lt]{\lineheight{1.25}\smash{\begin{tabular}[t]{l}\tiny $k+1$\end{tabular}}}}%
    \put(0.51480024,0.74820292){\makebox(0,0)[lt]{\lineheight{1.25}\smash{\begin{tabular}[t]{l}\tiny $k+1$\end{tabular}}}}%
    \put(0.67262153,0.42099247){\makebox(0,0)[lt]{\lineheight{1.25}\smash{\begin{tabular}[t]{l}\tiny $1$\end{tabular}}}}%
    \put(0,0.42448297){\makebox(0,0)[lt]{\lineheight{1.25}\smash{\begin{tabular}[t]{l}\tiny $k$\end{tabular}}}}%
    \put(0,0){\includegraphics[width=\unitlength,page=1]{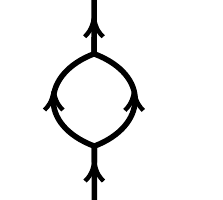}}%
  \end{picture}%
\endgroup%
} \sqcup \Theta_{\ell}$ and, by hypothesis, $\textit{DUR}(\Theta_{\ell}) = \{\theta_\mathbf{a}\}_{\mathbf{a} \in \{0,1\}^{\ell}}$ are a $\Z$-basis of $\mathscr{S}_1(\Theta_{\ell})$. Using \cref{lem:monoidal}, we can therefore express $f_{s,r,\mathbf{a}}$ as 
$$f_{s,r,\mathbf{a}}=\sum_{|\widetilde{\mathbf{a}}|=|\mathbf{a}|+r} b_{\widetilde{\mathbf{a}}} \cdot \gamma_{s,\theta_{\widetilde{\mathbf{a}}}}(\raisebox{-0.65\baselineskip}{\def\svgwidth{25pt}} \sqcup \Theta_{\ell})$$
(see \cref{fig:dotrewriting}), with $b_{\widetilde{\mathbf{a}}} \in \Z$.
\begin{equation} \label{fig:dotrewriting}
{\scriptsize \raisebox{-7\baselineskip}{\def\svgwidth{300pt}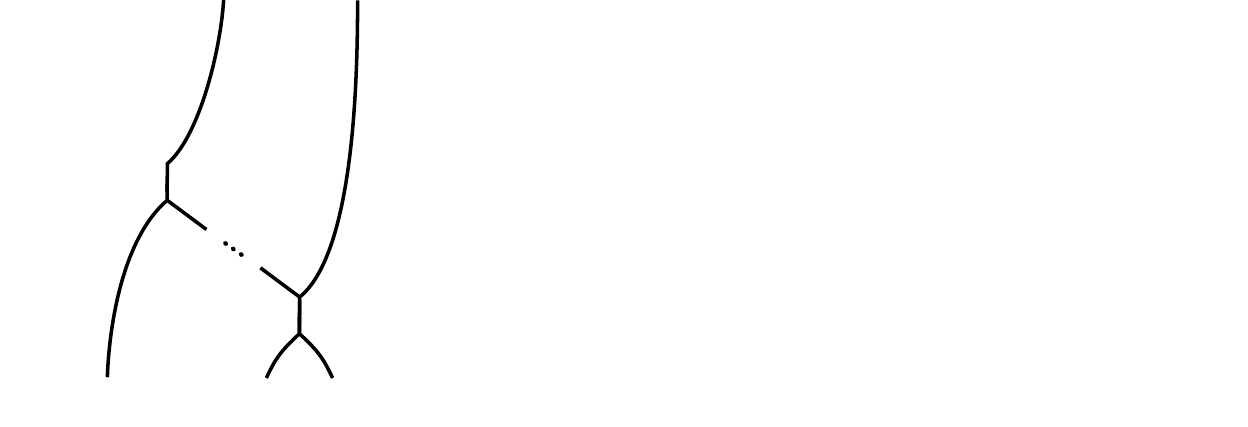}}
\end{equation}

Point (1) of the statement then follows from this:
\[
\textit{zip} \left( \sum_{|\widetilde{\mathbf{a}}|=|\mathbf{a}|+r} b_{\widetilde{\mathbf{a}}} \cdot \gamma_{s,\theta_{\widetilde{\mathbf{a}}}}(\raisebox{-0.65\baselineskip}{\def\svgwidth{25pt}} \sqcup \Theta_{\ell}) \right) = \sum_{|\widetilde{\mathbf{a}}|=|\mathbf{a}|+r} b_{\widetilde{\mathbf{a}}}(\gamma_{s,1,\theta_{\widetilde{\mathbf{a}}}}-\gamma_{s+1,0,\theta_{\widetilde{\mathbf{a}}}}).
\] 
Lastly, we observe the following:
\begin{remark} \label{rem:whenr+jtoobig}
If $|\mathbf{a}|+r >\ell$, by \cref{eq:degreemaxdecoration} (since $\ell=t_{\Theta_\ell}$), the decorated graph $f_{s,r,\mathbf{a}}$ is equal to zero, therefore its image under the zip map is zero as well.
\end{remark}

As for the second part of the statement, we observe that, using point (1) recursively: 
\begin{equation*}
\begin{split}
\gamma_{0,r+1,\theta_\mathbf{a}} & = \gamma_{1,r,\theta_{\mathbf{a}}} + \sum_{|\widetilde{\mathbf{a}}|=|\mathbf{a}|+r} b_{\widetilde{\mathbf{a}}}(\gamma_{0,1,\theta_{\widetilde{\mathbf{a}}}}-\gamma_{1,0,\theta_{\widetilde{\mathbf{a}}}}) = \\
& = \gamma_{2,r-1,\theta_{\mathbf{a}}} + \sum_{|\widetilde{\mathbf{a}}|=|\mathbf{a}|+r-1} b_{\widetilde{\mathbf{a}}}(\gamma_{1,1,\theta_{\widetilde{\mathbf{a}}}}-\gamma_{2,0,\theta_{\widetilde{\mathbf{a}}}}) + \sum_{|\widetilde{\mathbf{a}}|=|\mathbf{a}|+r} b_{\widetilde{\mathbf{a}}}(\gamma_{0,1,\theta_{\widetilde{\mathbf{a}}}}-\gamma_{1,0,\theta_{\widetilde{\mathbf{a}}}}) = \\
& = \ldots = \gamma_{r,1,\theta_\mathbf{a}} + \sum_{t=1}^r \sum_{|\widetilde{\mathbf{a}}|=|\mathbf{a}|+t} b_{\widetilde{\mathbf{a}}}(\gamma_{r-t,1,\theta_{\widetilde{\mathbf{a}}}}-\gamma_{r-t+1,0,\theta_{\widetilde{\mathbf{a}}}}).    
\end{split}
\end{equation*} 

However, by \cref{rem:whenr+jtoobig}, $b_{\widetilde{\mathbf{a}}}=0$ if $|\widetilde{\mathbf{a}}| >\ell$. Thus
\[
\sum_{t=1}^r \sum_{|\widetilde{\mathbf{a}}|=|\mathbf{a}|+t} b_{\widetilde{\mathbf{a}}}(\gamma_{r-t,1,\theta_{\widetilde{\mathbf{a}}}}-\gamma_{r-t+1,0,\theta_{\widetilde{\mathbf{a}}}}) = \sum_{t=1}^m \sum_{|\widetilde{\mathbf{a}}|=|\mathbf{a}|+t} b_{\widetilde{\mathbf{a}}}(\gamma_{r-t,1,\theta_{\widetilde{\mathbf{a}}}}-\gamma_{r-t+1,0,\theta_{\widetilde{\mathbf{a}}}})
\]
where $m=\min  \{r,\ell-|\mathbf{a}|\}$.

\end{proof}

We are now ready to show that the vinyl graphs $\Gamma_{k,l}$ are d.u.r.\ for all $k$ and $l$.

\begin{proof}[Proof of \cref{prop:gammaklisdur}]
We proceed by induction over $\ell$. Let first $\ell=0$ and $k\in \N$. The vinyl graph $\Gamma_{k,0}$ is the closure of $\; \raisebox{-0.65\baselineskip}{\def\svgwidth{25pt}
\begingroup%
  \makeatletter%
  \providecommand\color[2][]{%
    \errmessage{(Inkscape) Color is used for the text in Inkscape, but the package 'color.sty' is not loaded}%
    \renewcommand\color[2][]{}%
  }%
  \providecommand\transparent[1]{%
    \errmessage{(Inkscape) Transparency is used (non-zero) for the text in Inkscape, but the package 'transparent.sty' is not loaded}%
    \renewcommand\transparent[1]{}%
  }%
  \providecommand\rotatebox[2]{#2}%
  \newcommand*\fsize{\dimexpr\f@size pt\relax}%
  \newcommand*\lineheight[1]{\fontsize{\fsize}{#1\fsize}\selectfont}%
  \ifx\svgwidth\undefined%
    \setlength{\unitlength}{85.68269251bp}%
    \ifx\svgscale\undefined%
      \relax%
    \else%
      \setlength{\unitlength}{\unitlength * \real{\svgscale}}%
    \fi%
  \else%
    \setlength{\unitlength}{\svgwidth}%
  \fi%
  \global\let\svgwidth\undefined%
  \global\let\svgscale\undefined%
  \makeatother%
  \begin{picture}(1,0.78804226)%
    \lineheight{1}%
    \setlength\tabcolsep{0pt}%
    \put(0.44526273,0.35965513){\makebox(0,0)[lt]{\lineheight{1.25}\smash{\begin{tabular}[t]{l}\tiny $k+1$\end{tabular}}}}%
    \put(0.54770936,0.61287336){\makebox(0,0)[lt]{\lineheight{1.25}\smash{\begin{tabular}[t]{l}\tiny $1$\end{tabular}}}}%
    \put(0.01157408,0.61599963){\makebox(0,0)[lt]{\lineheight{1.25}\smash{\begin{tabular}[t]{l}\tiny $k$\end{tabular}}}}%
    \put(0,0){\includegraphics[width=\unitlength,page=1]{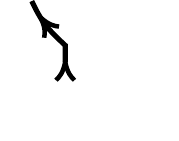}}%
    \put(0.56889486,0.10772896){\makebox(0,0)[lt]{\lineheight{1.25}\smash{\begin{tabular}[t]{l}\tiny $1$\end{tabular}}}}%
    \put(-0.00319129,0.1092921){\makebox(0,0)[lt]{\lineheight{1.25}\smash{\begin{tabular}[t]{l}\tiny $k$\end{tabular}}}}%
    \put(0,0){\includegraphics[width=\unitlength,page=2]{gammak.pdf}}%
  \end{picture}%
\endgroup%
}$, or equivalently, the closure of $\; \raisebox{-0.70\baselineskip}{\def\svgwidth{25pt}}$. Since, by \cref{rem:unionstrandsisdur}, the closure of $\;\;\raisebox{-0.65\baselineskip}{\def\svgwidth{14pt}
\begingroup%
  \makeatletter%
  \providecommand\color[2][]{%
    \errmessage{(Inkscape) Color is used for the text in Inkscape, but the package 'color.sty' is not loaded}%
    \renewcommand\color[2][]{}%
  }%
  \providecommand\transparent[1]{%
    \errmessage{(Inkscape) Transparency is used (non-zero) for the text in Inkscape, but the package 'transparent.sty' is not loaded}%
    \renewcommand\transparent[1]{}%
  }%
  \providecommand\rotatebox[2]{#2}%
  \newcommand*\fsize{\dimexpr\f@size pt\relax}%
  \newcommand*\lineheight[1]{\fontsize{\fsize}{#1\fsize}\selectfont}%
  \ifx\svgwidth\undefined%
    \setlength{\unitlength}{65.66262688bp}%
    \ifx\svgscale\undefined%
      \relax%
    \else%
      \setlength{\unitlength}{\unitlength * \real{\svgscale}}%
    \fi%
  \else%
    \setlength{\unitlength}{\svgwidth}%
  \fi%
  \global\let\svgwidth\undefined%
  \global\let\svgscale\undefined%
  \makeatother%
  \begin{picture}(1,1.4640853)%
    \lineheight{1}%
    \setlength\tabcolsep{0pt}%
    \put(0.29421222,1.00364701){\makebox(0,0)[lt]{\lineheight{1.25}\smash{\begin{tabular}[t]{l}\tiny $k+1$\end{tabular}}}}%
    \put(0,0){\includegraphics[width=\unitlength,page=1]{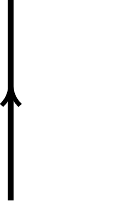}}%
  \end{picture}%
\endgroup%
}\;$ is d.u.r., we conclude that $\Gamma_{k,0}$ is d.u.r.\ as well by \cref{lem:un-bubblepreservesdur}.

We now assume that the statement is true for $\ell>0$ and for all $k\in\N$, i.e.\ that $\textit{DUR}(\Gamma_{k,\ell})$ is a basis for $\mathscr{S}_{1}(\Gamma_{k,\ell})$. Let us show that $\Gamma_{k,\ell+1}$ is d.u.r.\ for all $k \geq 0$. Since $\Gamma_{k,\ell+1}=\Gamma_{k,\ell+1}^{\textnormal{sq}}$ (\cref{rem:carreequalsgammakl}), we will actually show that the decorated graphs $\gamma_{i,j,\theta_\mathbf{a}}(\Gamma_{k,\ell+1}^{\textnormal{sq}})$ are a basis over $\Z$ of $\mathscr{S}_{1}(\Gamma_{k,\ell+1}^{\textnormal{sq}})$, for all $k \geq 0$.

\cref{lem:id17general} yields the isomorphism of \cref{fig:identity17gammakl}: 
\[
\mathscr{S}_{1}(\Gamma_{k,\ell+1})=\mathscr{S}_{1}(\Gamma_{k,\ell+1}^{\textnormal{sq}})\simeq [k] \; \mathscr{S}_{1}(\mathbb{S}_{k+1} \sqcup \Gamma_{0,\ell}) \oplus \mathscr{S}_{1}(\Gamma_{k+1,\ell}).
\]
\begin{equation} \label{fig:identity17gammakl}
{\scriptsize \raisebox{-6\baselineskip}{\def\svgwidth{310pt}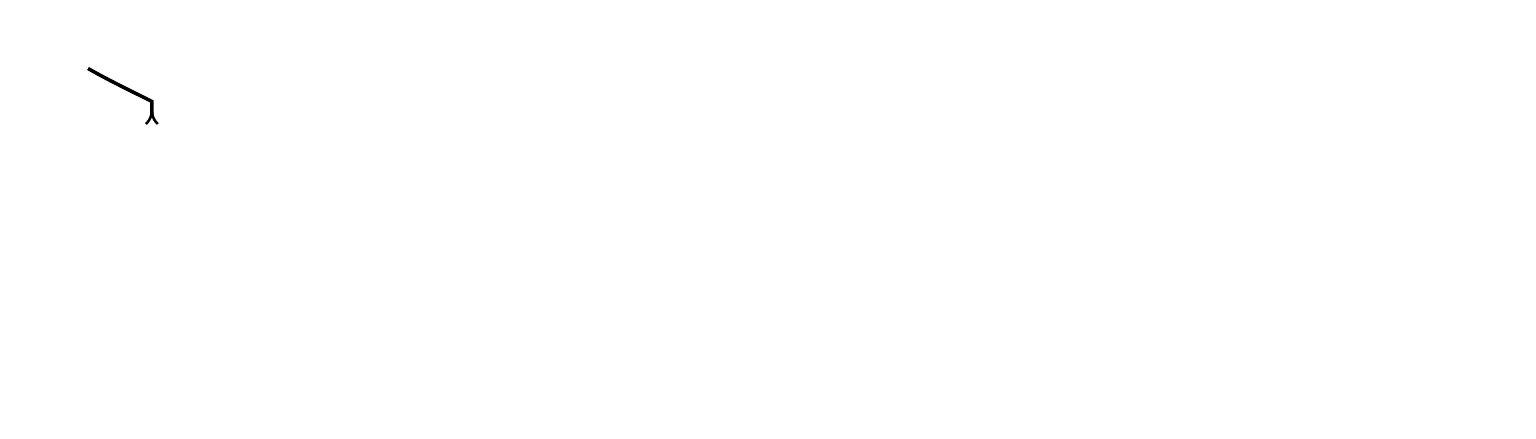}}
\end{equation}

This isomorphism increases $k$, but reduces $\ell$, so by the induction hypothesis and \cref{lem:monoidal}, both vinyl graphs on the RHS are d.u.r., and their d.u.r.\ bases are 
\[
\textit{DUR}(\Gamma_{k+1,\ell}) = \bigl\{ \gamma_{i,\theta_\mathbf{a}}(\Gamma_{k+1,\ell}),\; \text{ for } 0 \leq i \leq k+1, \; \mathbf{a} \in \{0,1\}^{\ell} \bigr\}
\]
and
\[
\textit{DUR}(\mathbb{S}_{k+1} \sqcup \Gamma_{0,\ell}) = \textit{DUR}(\mathbb{S}_{k+1} \sqcup \Theta_{\ell}) = \bigl\{\mathbb{S}_{k+1} \sqcup \theta_\mathbf{a},\; \text{ for } \mathbf{a} \in \{0,1\}^{\ell} \bigr\}.
\]
For simplicity, we will denote $\mathbb{S}_{k+1} \sqcup \theta_\mathbf{a}$ again by $\theta_\mathbf{a}$.

Our strategy will be to apply the isomorphisms $\phi_0^{\alpha}\colon \; \mathscr{S}_{1}(\mathbb{S}_{k+1} \sqcup \Gamma_{0,\ell}) \longrightarrow \mathscr{S}_{1}(\Gamma_{k,\ell+1}^{\textnormal{sq}})$ and $\phi_1 \colon\; \mathscr{S}_{1}(\Gamma_{k+1,\ell}) \longrightarrow \mathscr{S}_{1}(\Gamma_{k,\ell+1}^{\textnormal{sq}})$, realising the identity of \cref{lem:id17general}, to the d.u.r.\ basis $\textit{DUR}(\mathbb{S}_{k+1} \sqcup \Gamma_{0,\ell}) \sqcup \textit{DUR}(\Gamma_{k+1,\ell})$ of the RHS. This will yield a basis of the LHS, i.e.\ of $\mathscr{S}_{1}(\Gamma_{k,\ell+1}^{\textnormal{sq}})$, and we will then use \cref{lem:rewritingdots} and a change of basis to turn the obtained basis into the d.u.r.\ set.

Let's start by studying the image of the d.u.r.\ set of $\Gamma_{k+1,\ell}$ under $\phi_1$.
From its movie \cref{eq:moviephi1} we obtain
$$\gamma_{i,\theta_\mathbf{a}}(\Gamma_{k+1,\ell}) \overset{\phi_1}{\longmapsto} \gamma_{0,i,\theta_\mathbf{a}}(\Gamma_{k,\ell+1}^{\textnormal{sq}}).$$
Using \cref{lem:rewritingdots}, point (2), to express $\gamma_{0,i,\theta_\mathbf{a}}(\Gamma_{k,\ell+1}^{\textnormal{sq}})$ in terms of the decorated graphs of $\textit{DUR}(\Gamma_{k,\ell+1}^{\textnormal{sq}})$, we get:
\begin{equation*}
\begin{split}
\gamma_{0,\theta_\mathbf{a}} & \overset{\phi_1}{\longmapsto} \gamma_{0,0,\theta_\mathbf{a}} \\
\gamma_{1,\theta_\mathbf{a}} & \overset{\phi_1}{\longmapsto} \gamma_{0,1,\theta_\mathbf{a}} \\
\gamma_{i+1,\theta_\mathbf{a}} & \overset{\phi_1}{\longmapsto} \gamma_{i,1,\theta_{\mathbf{a}}} + \sum_{t=1}^m \sum_{|\widetilde{\mathbf{a}}|=|\mathbf{a}|+t} b_{\widetilde{\mathbf{a}}}(\gamma_{i-t,1,\theta_{\widetilde{\mathbf{a}}}}-\gamma_{i-t+1,0,\theta_{\widetilde{\mathbf{a}}}}) \; \text{ for } 1 \leq i \leq k.
\end{split}
\end{equation*}

As for the maps $\phi_0^{\alpha}$, for $0 \leq \alpha \leq k-1$, we find, for $\mathbf{a} \in \{0,1\}^{\ell}$
\[
\theta_\mathbf{a} \overset{\phi_0^{\alpha}}{\longmapsto} \sum_{r+s=\alpha} \gamma_{s,r+1,\theta_\mathbf{a}}(\Gamma_{k,\ell+1}^{\textnormal{sq}}) - \gamma_{s+1,r,\theta_\mathbf{a}}(\Gamma_{k,\ell+1}^{\textnormal{sq}}).
\]
Again, by \cref{lem:rewritingdots}, point (1):

\begin{equation*}
\begin{split}
\theta_\mathbf{a} & \overset{\phi_0^0}{\longmapsto} \gamma_{0,1,\theta_\mathbf{a}} - \gamma_{1,0,\theta_\mathbf{a}} \\
\theta_\mathbf{a} & \overset{\phi_0^i}{\longmapsto} \gamma_{i,1,\theta_{\mathbf{a}}} - \gamma_{i+1,0,\theta_{\mathbf{a}}} + \sum_{t=1}^m \sum_{|\widetilde{\mathbf{a}}|=|\mathbf{a}|+t} b_{\widetilde{\mathbf{a}}}(\gamma_{i-t,1,\theta_{\widetilde{\mathbf{a}}}} - \gamma_{i-t+1,0,\theta_{\widetilde{\mathbf{a}}}}) \; \text{ for } 1 \leq i \leq k-1.
\end{split}
\end{equation*}

 where $m=\min\{i,\ell -|\mathbf{a}|\}$.

We have expressed the images of the d.u.r.\ bases under $\phi_1$ and $\phi_0^{\alpha}$ as linear combinations of the d.u.r.\ set of $\Gamma_{k,\ell + 1}^{\text{sq}}$. Since we are working over $\Z$, however, in order to show that $\Gamma_{k,\ell + 1}^{\text{sq}}$ is d.u.r.\, we need to make sure that there is a change of basis that turns the basis we just obtained into the d.u.r.\ set. We start by calling, for $0 \leq i \leq k-1$ and $\mathbf{a} \in \{0,1\}^{\ell}$:
\[
\chi_{\mathbf{a}}^{i} = \gamma_{i,1,\theta_{\mathbf{a}}}(\Gamma_{k,\ell+1}^{\textnormal{sq}}) - \gamma_{i+1,0,\theta_{\mathbf{a}}}(\Gamma_{k,\ell+1}^{\textnormal{sq}})
\]
Then:

\begin{equation*}
\begin{gathered}
\phi_0^0(\theta_{\mathbf{a}})=\chi_{\mathbf{a}}^0 \\
\phi_0^1(\theta_{\mathbf{a}})-\sum_{|\widetilde{\mathbf{a}}|=|\mathbf{a}|+1} b_{\widetilde{\mathbf{a}}}\chi_{\widetilde{\mathbf{a}}}^0=\chi_{\mathbf{a}}^1 \\
\vdots \\
\phi_0^i(\theta_{\mathbf{a}}) - \sum_{t=1}^m \sum_{|\widetilde{\mathbf{a}}|=|\mathbf{a}|+t} b_{\widetilde{\mathbf{a}}}\chi_{\widetilde{\mathbf{a}}}^{i-t} =\chi_{\mathbf{a}}^i.
\end{gathered}
\end{equation*}

Now, in order to obtain the $\gamma_{i,1,\theta_{\mathbf{a}}}(\Gamma_{k,\ell+1}^{\textnormal{sq}})$ for all $0 \leq i \leq k$ and $\mathbf{a} \in \{0,1\}^{\ell}$, we do:

\begin{equation*}
\phi_1(\gamma_{i+1,\theta_{\mathbf{a}}}(\Gamma_{k+1,\ell})) - \sum_{t=1}^m \sum_{|\widetilde{\mathbf{a}}|=|\mathbf{a}|+t} b_{\widetilde{\mathbf{a}}}\chi_{\widetilde{\mathbf{a}}}^{i-t} = \gamma_{i,1,\theta_{\mathbf{a}}}(\Gamma_{k,\ell+1}^{\textnormal{sq}}).
\end{equation*}

Finally, we obtain the $\gamma_{i,0,\theta_{\mathbf{a}}}(\Gamma_{k,\ell+1}^{\textnormal{sq}})$:

\begin{equation*}
\begin{gathered}
\phi_1(\gamma_{0,\theta_{\mathbf{a}}}(\Gamma_{k+1,\ell})) = \gamma_{0,0,\theta_{\mathbf{a}}}(\Gamma_{k,\ell+1}^{\textnormal{sq}}) \\
\chi_{\mathbf{a}}^{i-1} - \gamma_{i-1,1,\theta_{\mathbf{a}}}(\Gamma_{k,\ell+1}^{\textnormal{sq}}) = \gamma_{i,0,\theta_{\mathbf{a}}}(\Gamma_{k,\ell+1}^{\textnormal{sq}}) \; \text{ for } 1 \leq i \leq k.
\end{gathered}
\end{equation*}
This ends the proof of the proposition.
\end{proof}

\section{Basis for elementary graphs}
\label{sec:basisgeneral}

When $\Gamma$ is an elementary graph, calculations inside $\mathscr{S}_1(\Gamma)$ become much easier. The dot migration identity \cref{eq:dotmigration}, for instance, becomes:
\begin{equation} \label{eq:dotmigrationelementary}
{\small\raisebox{-0.8\baselineskip}{\def\svgwidth{80pt}
\begingroup%
  \makeatletter%
  \providecommand\color[2][]{%
    \errmessage{(Inkscape) Color is used for the text in Inkscape, but the package 'color.sty' is not loaded}%
    \renewcommand\color[2][]{}%
  }%
  \providecommand\transparent[1]{%
    \errmessage{(Inkscape) Transparency is used (non-zero) for the text in Inkscape, but the package 'transparent.sty' is not loaded}%
    \renewcommand\transparent[1]{}%
  }%
  \providecommand\rotatebox[2]{#2}%
  \newcommand*\fsize{\dimexpr\f@size pt\relax}%
  \newcommand*\lineheight[1]{\fontsize{\fsize}{#1\fsize}\selectfont}%
  \ifx\svgwidth\undefined%
    \setlength{\unitlength}{119.04948114bp}%
    \ifx\svgscale\undefined%
      \relax%
    \else%
      \setlength{\unitlength}{\unitlength * \real{\svgscale}}%
    \fi%
  \else%
    \setlength{\unitlength}{\svgwidth}%
  \fi%
  \global\let\svgwidth\undefined%
  \global\let\svgscale\undefined%
  \makeatother%
  \begin{picture}(1,0.24597832)%
    \lineheight{1}%
    \setlength\tabcolsep{0pt}%
    \put(0,0){\includegraphics[width=\unitlength,page=1]{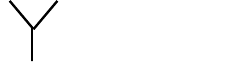}}%
    \put(-0.00278902,0.07548136){\color[rgb]{0.90588235,0.4,0.41568627}\makebox(0,0)[lt]{\lineheight{1.25}\smash{\begin{tabular}[t]{l}$e_1$\end{tabular}}}}%
    \put(0,0){\includegraphics[width=\unitlength,page=2]{e1migration.pdf}}%
    \put(0.48271505,0.21866826){\color[rgb]{0.90588235,0.4,0.41568627}\makebox(0,0)[lt]{\lineheight{1.25}\smash{\begin{tabular}[t]{l}$1$\end{tabular}}}}%
    \put(0.93865799,0.11118063){\color[rgb]{0.90588235,0.4,0.41568627}\makebox(0,0)[lt]{\lineheight{1.25}\smash{\begin{tabular}[t]{l}$1$\end{tabular}}}}%
    \put(0.31049348,0.09754734){\color[rgb]{0,0,0}\makebox(0,0)[lt]{\lineheight{1.25}\smash{\begin{tabular}[t]{l}$=$\end{tabular}}}}%
    \put(0.66509637,0.09801458){\color[rgb]{0,0,0}\makebox(0,0)[lt]{\lineheight{1.25}\smash{\begin{tabular}[t]{l}$+$\end{tabular}}}}%
    \put(0,0){\includegraphics[width=\unitlength,page=3]{e1migration.pdf}}%
  \end{picture}%
\endgroup%
} \quad , \qquad\qquad \raisebox{-0.8\baselineskip}{\def\svgwidth{56pt}
\begingroup%
  \makeatletter%
  \providecommand\color[2][]{%
    \errmessage{(Inkscape) Color is used for the text in Inkscape, but the package 'color.sty' is not loaded}%
    \renewcommand\color[2][]{}%
  }%
  \providecommand\transparent[1]{%
    \errmessage{(Inkscape) Transparency is used (non-zero) for the text in Inkscape, but the package 'transparent.sty' is not loaded}%
    \renewcommand\transparent[1]{}%
  }%
  \providecommand\rotatebox[2]{#2}%
  \newcommand*\fsize{\dimexpr\f@size pt\relax}%
  \newcommand*\lineheight[1]{\fontsize{\fsize}{#1\fsize}\selectfont}%
  \ifx\svgwidth\undefined%
    \setlength{\unitlength}{82.68436739bp}%
    \ifx\svgscale\undefined%
      \relax%
    \else%
      \setlength{\unitlength}{\unitlength * \real{\svgscale}}%
    \fi%
  \else%
    \setlength{\unitlength}{\svgwidth}%
  \fi%
  \global\let\svgwidth\undefined%
  \global\let\svgscale\undefined%
  \makeatother%
  \begin{picture}(1,0.35349429)%
    \lineheight{1}%
    \setlength\tabcolsep{0pt}%
    \put(-0.00401565,0.10867833){\color[rgb]{0.90588235,0.4,0.41568627}\makebox(0,0)[lt]{\lineheight{1.25}\smash{\begin{tabular}[t]{l}$e_2$\end{tabular}}}}%
    \put(0.59750692,0.21198528){\color[rgb]{0.90588235,0.4,0.41568627}\makebox(0,0)[lt]{\lineheight{1.25}\smash{\begin{tabular}[t]{l}$1$\end{tabular}}}}%
    \put(0.4289092,0.14044907){\color[rgb]{0,0,0}\makebox(0,0)[lt]{\lineheight{1.25}\smash{\begin{tabular}[t]{l}$=$\end{tabular}}}}%
    \put(0.91167938,0.21495709){\color[rgb]{0.90588235,0.4,0.41568627}\makebox(0,0)[lt]{\lineheight{1.25}\smash{\begin{tabular}[t]{l}$1$\end{tabular}}}}%
    \put(0,0){\includegraphics[width=\unitlength,page=1]{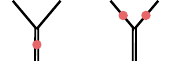}}%
  \end{picture}%
\endgroup%
}}
\end{equation}

and the \emph{zip} map is now given by:
\begin{equation} \label{eq:zipelementary}
{\small \raisebox{-1\baselineskip}{\def\svgwidth{100pt}
\begingroup%
  \makeatletter%
  \providecommand\color[2][]{%
    \errmessage{(Inkscape) Color is used for the text in Inkscape, but the package 'color.sty' is not loaded}%
    \renewcommand\color[2][]{}%
  }%
  \providecommand\transparent[1]{%
    \errmessage{(Inkscape) Transparency is used (non-zero) for the text in Inkscape, but the package 'transparent.sty' is not loaded}%
    \renewcommand\transparent[1]{}%
  }%
  \providecommand\rotatebox[2]{#2}%
  \newcommand*\fsize{\dimexpr\f@size pt\relax}%
  \newcommand*\lineheight[1]{\fontsize{\fsize}{#1\fsize}\selectfont}%
  \ifx\svgwidth\undefined%
    \setlength{\unitlength}{164.33212535bp}%
    \ifx\svgscale\undefined%
      \relax%
    \else%
      \setlength{\unitlength}{\unitlength * \real{\svgscale}}%
    \fi%
  \else%
    \setlength{\unitlength}{\svgwidth}%
  \fi%
  \global\let\svgwidth\undefined%
  \global\let\svgscale\undefined%
  \makeatother%
  \begin{picture}(1,0.21708712)%
    \lineheight{1}%
    \setlength\tabcolsep{0pt}%
    \put(0,0){\includegraphics[width=\unitlength,page=1]{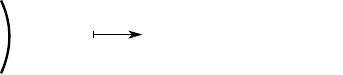}}%
    \put(0.28047065,0.16103886){\color[rgb]{0,0,0}\makebox(0,0)[lt]{\lineheight{1.25}\smash{\begin{tabular}[t]{l}$\textit{zip}$\end{tabular}}}}%
    \put(0,0){\includegraphics[width=\unitlength,page=2]{zipelementary.pdf}}%
    \put(0.71553273,0.16328559){\color[rgb]{0.96470588,0.4,0.36470588}\makebox(0,0)[lt]{\lineheight{1.25}\smash{\begin{tabular}[t]{l}$1$\end{tabular}}}}%
    \put(0.84219987,0.02935725){\color[rgb]{0.96470588,0.4,0.36470588}\makebox(0,0)[lt]{\lineheight{1.25}\smash{\begin{tabular}[t]{l}$1$\end{tabular}}}}%
    \put(0.78259245,0.09649856){\color[rgb]{0,0,0}\makebox(0,0)[lt]{\lineheight{1.25}\smash{\begin{tabular}[t]{l}$-$\end{tabular}}}}%
  \end{picture}%
\endgroup%
}}   
\end{equation}

The following also holds.
\begin{lem} \label{lem:unbubbleidentities}
When $k=1$, the \textit{unbubble} map \cref{eq:unbubble} yields the following, for all $i \in \N$:
\begin{equation} \label{eq:unbubbleidentitiesk}
\textit{unbubble }(\quad{\tiny \raisebox{-1.25\baselineskip}{\def\svgwidth{20pt}
\begingroup%
  \makeatletter%
  \providecommand\color[2][]{%
    \errmessage{(Inkscape) Color is used for the text in Inkscape, but the package 'color.sty' is not loaded}%
    \renewcommand\color[2][]{}%
  }%
  \providecommand\transparent[1]{%
    \errmessage{(Inkscape) Transparency is used (non-zero) for the text in Inkscape, but the package 'transparent.sty' is not loaded}%
    \renewcommand\transparent[1]{}%
  }%
  \providecommand\rotatebox[2]{#2}%
  \newcommand*\fsize{\dimexpr\f@size pt\relax}%
  \newcommand*\lineheight[1]{\fontsize{\fsize}{#1\fsize}\selectfont}%
  \ifx\svgwidth\undefined%
    \setlength{\unitlength}{37.1168615bp}%
    \ifx\svgscale\undefined%
      \relax%
    \else%
      \setlength{\unitlength}{\unitlength * \real{\svgscale}}%
    \fi%
  \else%
    \setlength{\unitlength}{\svgwidth}%
  \fi%
  \global\let\svgwidth\undefined%
  \global\let\svgscale\undefined%
  \makeatother%
  \begin{picture}(1,1.12966821)%
    \lineheight{1}%
    \setlength\tabcolsep{0pt}%
    \put(0,0){\includegraphics[width=\unitlength,page=1]{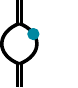}}%
    \put(0.48426192,0.81726358){\color[rgb]{0,0.53333333,0.63137255}\makebox(0,0)[lt]{\lineheight{1.25}\smash{\begin{tabular}[t]{l}\small $i$\end{tabular}}}}%
  \end{picture}%
\endgroup%
}})= - \textit{ unbubble }(\quad{\tiny \raisebox{-1.25\baselineskip}{\def\svgwidth{15pt}
\begingroup%
  \makeatletter%
  \providecommand\color[2][]{%
    \errmessage{(Inkscape) Color is used for the text in Inkscape, but the package 'color.sty' is not loaded}%
    \renewcommand\color[2][]{}%
  }%
  \providecommand\transparent[1]{%
    \errmessage{(Inkscape) Transparency is used (non-zero) for the text in Inkscape, but the package 'transparent.sty' is not loaded}%
    \renewcommand\transparent[1]{}%
  }%
  \providecommand\rotatebox[2]{#2}%
  \newcommand*\fsize{\dimexpr\f@size pt\relax}%
  \newcommand*\lineheight[1]{\fontsize{\fsize}{#1\fsize}\selectfont}%
  \ifx\svgwidth\undefined%
    \setlength{\unitlength}{27.5763251bp}%
    \ifx\svgscale\undefined%
      \relax%
    \else%
      \setlength{\unitlength}{\unitlength * \real{\svgscale}}%
    \fi%
  \else%
    \setlength{\unitlength}{\svgwidth}%
  \fi%
  \global\let\svgwidth\undefined%
  \global\let\svgscale\undefined%
  \makeatother%
  \begin{picture}(1,1.52049769)%
    \lineheight{1}%
    \setlength\tabcolsep{0pt}%
    \put(0,0){\includegraphics[width=\unitlength,page=1]{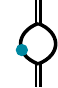}}%
    \put(0,0.36094429){\color[rgb]{0,0.53333333,0.63137255}\makebox(0,0)[lt]{\lineheight{1.25}\smash{\begin{tabular}[t]{l}\small $i$\end{tabular}}}}%
  \end{picture}%
\endgroup%
}}\quad)\, , \qquad \textit{unbubble }(\quad{\tiny \raisebox{-1.25\baselineskip}{\def\svgwidth{24pt}
\begingroup%
  \makeatletter%
  \providecommand\color[2][]{%
    \errmessage{(Inkscape) Color is used for the text in Inkscape, but the package 'color.sty' is not loaded}%
    \renewcommand\color[2][]{}%
  }%
  \providecommand\transparent[1]{%
    \errmessage{(Inkscape) Transparency is used (non-zero) for the text in Inkscape, but the package 'transparent.sty' is not loaded}%
    \renewcommand\transparent[1]{}%
  }%
  \providecommand\rotatebox[2]{#2}%
  \newcommand*\fsize{\dimexpr\f@size pt\relax}%
  \newcommand*\lineheight[1]{\fontsize{\fsize}{#1\fsize}\selectfont}%
  \ifx\svgwidth\undefined%
    \setlength{\unitlength}{45.49945057bp}%
    \ifx\svgscale\undefined%
      \relax%
    \else%
      \setlength{\unitlength}{\unitlength * \real{\svgscale}}%
    \fi%
  \else%
    \setlength{\unitlength}{\svgwidth}%
  \fi%
  \global\let\svgwidth\undefined%
  \global\let\svgscale\undefined%
  \makeatother%
  \begin{picture}(1,0.92154384)%
    \lineheight{1}%
    \setlength\tabcolsep{0pt}%
    \put(0,0){\includegraphics[width=\unitlength,page=1]{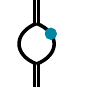}}%
    \put(0.5792789,0.66981223){\color[rgb]{0,0.53333333,0.63137255}\makebox(0,0)[lt]{\lineheight{1.25}\smash{\begin{tabular}[t]{l}\small $i$\end{tabular}}}}%
    \put(0,0){\includegraphics[width=\unitlength,page=2]{unbubblek-k.pdf}}%
    \put(0,0.1757352){\color[rgb]{0,0.53333333,0.63137255}\makebox(0,0)[lt]{\lineheight{1.25}\smash{\begin{tabular}[t]{l}\small $i$\end{tabular}}}}%
  \end{picture}%
\endgroup%
}})=0
\end{equation}
\begin{equation} \label{eq:unbubbleidentities1}
{\tiny \raisebox{-1.25\baselineskip}{\def\svgwidth{73pt}
\begingroup%
  \makeatletter%
  \providecommand\color[2][]{%
    \errmessage{(Inkscape) Color is used for the text in Inkscape, but the package 'color.sty' is not loaded}%
    \renewcommand\color[2][]{}%
  }%
  \providecommand\transparent[1]{%
    \errmessage{(Inkscape) Transparency is used (non-zero) for the text in Inkscape, but the package 'transparent.sty' is not loaded}%
    \renewcommand\transparent[1]{}%
  }%
  \providecommand\rotatebox[2]{#2}%
  \newcommand*\fsize{\dimexpr\f@size pt\relax}%
  \newcommand*\lineheight[1]{\fontsize{\fsize}{#1\fsize}\selectfont}%
  \ifx\svgwidth\undefined%
    \setlength{\unitlength}{145.09920257bp}%
    \ifx\svgscale\undefined%
      \relax%
    \else%
      \setlength{\unitlength}{\unitlength * \real{\svgscale}}%
    \fi%
  \else%
    \setlength{\unitlength}{\svgwidth}%
  \fi%
  \global\let\svgwidth\undefined%
  \global\let\svgscale\undefined%
  \makeatother%
  \begin{picture}(1,0.28911018)%
    \lineheight{1}%
    \setlength\tabcolsep{0pt}%
    \put(0,0){\includegraphics[width=\unitlength,page=1]{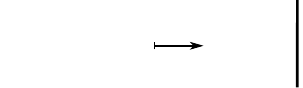}}%
    \put(0.39425012,0.18131002){\color[rgb]{0,0,0}\makebox(0,0)[lt]{\lineheight{1.25}\smash{\begin{tabular}[t]{l}\small $\textit{unbubble}$\end{tabular}}}}%
    \put(0,0){\includegraphics[width=\unitlength,page=2]{unbubble1.pdf}}%
    \put(0,0.11835025){\color[rgb]{0,0.53333333,0.63137255}\makebox(0,0)[lt]{\lineheight{1.25}\smash{\begin{tabular}[t]{l}\small $1$\end{tabular}}}}%
    \put(0,0){\includegraphics[width=\unitlength,page=3]{unbubble1.pdf}}%
  \end{picture}%
\endgroup%
}}    
\end{equation}    
\end{lem}

\begin{remark} \label{rem:graphrepsarelocal}

\begin{itemize}
    \item Unlike in the previous section, the graphs of this section are only represented locally, in the tangles where they differ from one another. For instance, the figure $\reflectbox {\raisebox{-0.8\baselineskip}{\def\svgwidth{20pt}
\begingroup%
  \makeatletter%
  \providecommand\color[2][]{%
    \errmessage{(Inkscape) Color is used for the text in Inkscape, but the package 'color.sty' is not loaded}%
    \renewcommand\color[2][]{}%
  }%
  \providecommand\transparent[1]{%
    \errmessage{(Inkscape) Transparency is used (non-zero) for the text in Inkscape, but the package 'transparent.sty' is not loaded}%
    \renewcommand\transparent[1]{}%
  }%
  \providecommand\rotatebox[2]{#2}%
  \newcommand*\fsize{\dimexpr\f@size pt\relax}%
  \newcommand*\lineheight[1]{\fontsize{\fsize}{#1\fsize}\selectfont}%
  \ifx\svgwidth\undefined%
    \setlength{\unitlength}{32.43885113bp}%
    \ifx\svgscale\undefined%
      \relax%
    \else%
      \setlength{\unitlength}{\unitlength * \real{\svgscale}}%
    \fi%
  \else%
    \setlength{\unitlength}{\svgwidth}%
  \fi%
  \global\let\svgwidth\undefined%
  \global\let\svgscale\undefined%
  \makeatother%
  \begin{picture}(1,1.21529287)%
    \lineheight{1}%
    \setlength\tabcolsep{0pt}%
    \put(0,0){\includegraphics[width=\unitlength,page=1]{graphid17-1l.pdf}}%
  \end{picture}%
\endgroup%
}}$ stands for a generic vinyl graph $\Gamma$ that contains a tangle of shape $\reflectbox {\raisebox{-0.8\baselineskip}{\def\svgwidth{20pt}}}$.

    \item Similarly, when writing $\gamma_{i,j}(\,\reflectbox {\raisebox{-0.8\baselineskip}{\def\svgwidth{20pt}}}\,)$, we are really referring to the subset of $\textit{DUR}(\Gamma)$ containing all decorated graphs $\gamma_{i,j}^{\mathbf{a}}$ that have decoration dictated by $\gamma_{i,j}$ on the tangle $\reflectbox {\raisebox{-0.8\baselineskip}{\def\svgwidth{20pt}}}$ and by the tuple $\mathbf{a}$ outside of it.
\end{itemize}

\end{remark}

\begin{lem} \label{lem:id17preservesdur}
\begin{enumerate}
    \item When $k=1$, the LHS of identity \cref{eq:id17general} is d.u.r.\ if and only if the two vinyl graphs on the RHS are d.u.r.
    \item The same statement holds for the mirror of \cref{eq:id17general}.
\end{enumerate}
\end{lem}

While there is an apparent symmetry of statement (1) and (2), the d.u.r.\ set does not intertwine this symmetry: the left and right edges adjacent to a split vertex play a different role. Hence the proof of (1) and (2) are different.

\begin{proof}
Let us start by proving point (2) of the statement. The mirror of identity \cref{eq:id17general}, for $k=1$, is given in \cref{fig:mapsid17}. 

\begin{figure}[ht]
\centering
\begin{tikzcd}[column sep=1em]
	\mathscr{S}_1(\,\reflectbox {\raisebox{-0.8\baselineskip}{\def\svgwidth{20pt}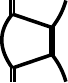}}\,)
	\arrow[rr, orange, swap, shift right=3,  start anchor={[xshift=-2ex]}, end anchor={[xshift=2ex]}, bend right, "\psi_0"] \arrow[rrrr, orange, swap, shift right=3,  start anchor={[xshift=-3ex]}, end anchor={[xshift=2ex]}, bend right=40, "\psi_1"]
	& \simeq & \mathscr{S}_1(\,\reflectbox{\raisebox{-0.8\baselineskip}{\def\svgwidth{10pt}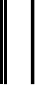}} \,)
	\arrow[ll, blue, swap, shift right=3,  start anchor={[xshift=2ex]}, end anchor={[xshift=-2ex]}, bend right, "\phi_0"]
	& \oplus
	& \mathscr{S}_1(\,\reflectbox{\raisebox{-0.8\baselineskip}{\def\svgwidth{13pt}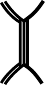}}\,)
	\arrow[llll, blue, swap, shift right=3,  start anchor={[xshift=2ex]}, end anchor={[xshift=-3ex]}, bend right=40, "\phi_1"]
\end{tikzcd}
\caption{}
\label{fig:mapsid17}
\end{figure}

We first assume that $\reflectbox {\raisebox{-0.8\baselineskip}{\def\svgwidth{20pt}}}$ is d.u.r. The d.u.r\ decorated graphs are locally given by $\gamma_{0,0},\gamma_{0,1},\gamma_{1,0},\gamma_{1,1}$ (see left-hand side of \cref{eq:imagebasisid17r}). When $k=1$ the isomorphism $\psi$ realising identity \cref{eq:id17general} is much simpler. Following the movie for $\psi_1$ on the d.u.r.\ basis one obtains
\begin{equation} \label{eq:phi1moviebasis}
{\scriptsize \raisebox{-2.2\baselineskip}{\def\svgwidth{320pt}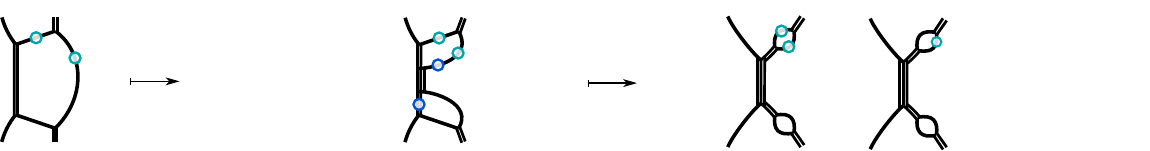}}
\end{equation}

Applying the \emph{unbubble} map to \cref{eq:phi1moviebasis}, and using \cref{lem:unbubbleidentities}, we find the images of the d.u.r.\ basis under $\psi_1$. A similar approach can be used for $\psi_0$. We get:
\begin{equation} \label{eq:imagebasisid17r}
{\small \raisebox{-6.1\baselineskip}{\def\svgwidth{230pt}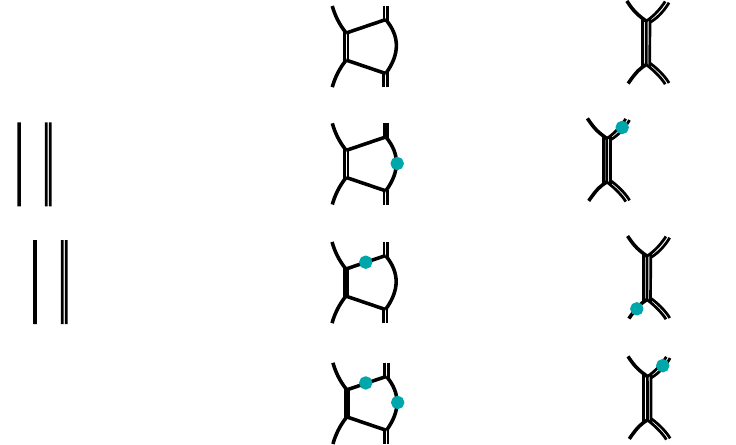}}
\end{equation}

The $\Z$-module map $\psi_0$ is surjective, therefore we deduce from \cref{eq:imagebasisid17r} that $\mathscr{S}_{1}(\reflectbox{\raisebox{-0.8\baselineskip}{\def\svgwidth{10pt}
\begingroup%
  \makeatletter%
  \providecommand\color[2][]{%
    \errmessage{(Inkscape) Color is used for the text in Inkscape, but the package 'color.sty' is not loaded}%
    \renewcommand\color[2][]{}%
  }%
  \providecommand\transparent[1]{%
    \errmessage{(Inkscape) Transparency is used (non-zero) for the text in Inkscape, but the package 'transparent.sty' is not loaded}%
    \renewcommand\transparent[1]{}%
  }%
  \providecommand\rotatebox[2]{#2}%
  \newcommand*\fsize{\dimexpr\f@size pt\relax}%
  \newcommand*\lineheight[1]{\fontsize{\fsize}{#1\fsize}\selectfont}%
  \ifx\svgwidth\undefined%
    \setlength{\unitlength}{16.49477618bp}%
    \ifx\svgscale\undefined%
      \relax%
    \else%
      \setlength{\unitlength}{\unitlength * \real{\svgscale}}%
    \fi%
  \else%
    \setlength{\unitlength}{\svgwidth}%
  \fi%
  \global\let\svgwidth\undefined%
  \global\let\svgscale\undefined%
  \makeatother%
  \begin{picture}(1,2.45038976)%
    \lineheight{1}%
    \setlength\tabcolsep{0pt}%
    \put(0,0){\includegraphics[width=\unitlength,page=1]{graphid17-3l.pdf}}%
  \end{picture}%
\endgroup%
}})$ is spanned by the decorated graph (actually, the family of decorated graphs) $\gamma = \reflectbox{\raisebox{-0.8\baselineskip}{\def\svgwidth{10pt}}}$, which thus constitutes its d.u.r.\ basis. We further observe that the three elements $\gamma_{0,0},\gamma_{0,1}+\gamma_{1,0}$ and $\gamma_{1,1}$ are in the kernel of $\psi_0$ and are linearly independent. They form a generating set for $\ker(\psi_0)$ since they span a space of the right dimension and they can be completed to a basis of $\reflectbox {\raisebox{-0.8\baselineskip}{\def\svgwidth{20pt}}}$ by adding $\gamma_{1,0}$. Their image under $\psi=\left(\begin{array}{c}
     \psi_0 \\
     \psi_1
\end{array}\right)$, 
which gives exactly $\textit{DUR}(\reflectbox{\raisebox{-0.8\baselineskip}{\def\svgwidth{13pt}
\begingroup%
  \makeatletter%
  \providecommand\color[2][]{%
    \errmessage{(Inkscape) Color is used for the text in Inkscape, but the package 'color.sty' is not loaded}%
    \renewcommand\color[2][]{}%
  }%
  \providecommand\transparent[1]{%
    \errmessage{(Inkscape) Transparency is used (non-zero) for the text in Inkscape, but the package 'transparent.sty' is not loaded}%
    \renewcommand\transparent[1]{}%
  }%
  \providecommand\rotatebox[2]{#2}%
  \newcommand*\fsize{\dimexpr\f@size pt\relax}%
  \newcommand*\lineheight[1]{\fontsize{\fsize}{#1\fsize}\selectfont}%
  \ifx\svgwidth\undefined%
    \setlength{\unitlength}{21.51854517bp}%
    \ifx\svgscale\undefined%
      \relax%
    \else%
      \setlength{\unitlength}{\unitlength * \real{\svgscale}}%
    \fi%
  \else%
    \setlength{\unitlength}{\svgwidth}%
  \fi%
  \global\let\svgwidth\undefined%
  \global\let\svgscale\undefined%
  \makeatother%
  \begin{picture}(1,1.89361977)%
    \lineheight{1}%
    \setlength\tabcolsep{0pt}%
    \put(0,0){\includegraphics[width=\unitlength,page=1]{graphid17-2l.pdf}}%
  \end{picture}%
\endgroup%
}})$, is therefore a basis of $\mathscr{S}_{1}(\reflectbox{\raisebox{-0.8\baselineskip}{\def\svgwidth{13pt}}})$. This proves that both $\reflectbox{\raisebox{-0.8\baselineskip}{\def\svgwidth{10pt}}}$ and $\reflectbox{\raisebox{-0.8\baselineskip}{\def\svgwidth{13pt}}}$ are d.u.r.

Conversely, if $\reflectbox{\raisebox{-0.8\baselineskip}{\def\svgwidth{10pt}}}$ and $\reflectbox{\raisebox{-0.8\baselineskip}{\def\svgwidth{13pt}}}$ are both d.u.r., let us group (see \cref{rem:graphrepsarelocal}) their d.u.r.\ bases into the four families $\gamma_0,\gamma_{e_1},\gamma_{e_2},\gamma$ shown on the left-hand side of \cref{eq:imagebasisid17reverse}. We will show that the d.u.r.\ bases are mapped to $\textit{DUR}(\reflectbox {\raisebox{-0.8\baselineskip}{\def\svgwidth{20pt}}})$ by the isomorphism $\phi$ of \cref{fig:mapsid17}. Using dot migration (identity \cref{eq:dotmigration}), we obtain the images of the d.u.r.\ bases under $\phi_0$ \cref{eq:moviephi0alpha} and $\phi_1$ \cref{eq:moviephi1}:
\begin{equation} \label{eq:imagebasisid17reverse}
{\small \raisebox{-6.1\baselineskip}{\def\svgwidth{200pt}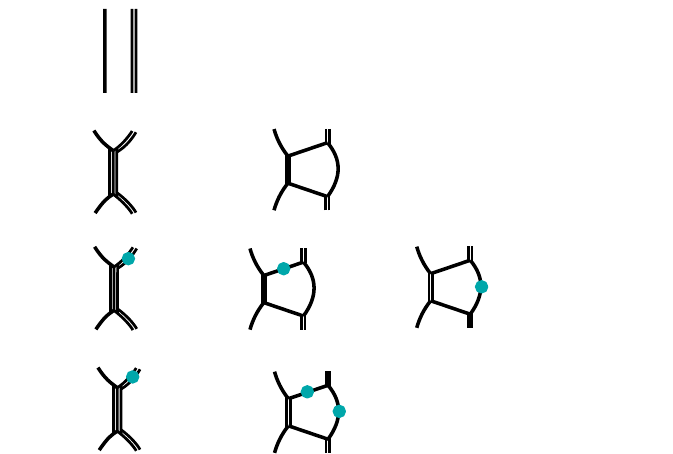}}
\end{equation}

We would like to express $\phi_0(\gamma)$ in terms of the d.u.r.\ set. We observe that $\phi_1(\raisebox{-0.8\baselineskip}{\def\svgwidth{13pt}
\begingroup%
  \makeatletter%
  \providecommand\color[2][]{%
    \errmessage{(Inkscape) Color is used for the text in Inkscape, but the package 'color.sty' is not loaded}%
    \renewcommand\color[2][]{}%
  }%
  \providecommand\transparent[1]{%
    \errmessage{(Inkscape) Transparency is used (non-zero) for the text in Inkscape, but the package 'transparent.sty' is not loaded}%
    \renewcommand\transparent[1]{}%
  }%
  \providecommand\rotatebox[2]{#2}%
  \newcommand*\fsize{\dimexpr\f@size pt\relax}%
  \newcommand*\lineheight[1]{\fontsize{\fsize}{#1\fsize}\selectfont}%
  \ifx\svgwidth\undefined%
    \setlength{\unitlength}{21.51854517bp}%
    \ifx\svgscale\undefined%
      \relax%
    \else%
      \setlength{\unitlength}{\unitlength * \real{\svgscale}}%
    \fi%
  \else%
    \setlength{\unitlength}{\svgwidth}%
  \fi%
  \global\let\svgwidth\undefined%
  \global\let\svgscale\undefined%
  \makeatother%
  \begin{picture}(1,1.89361977)%
    \lineheight{1}%
    \setlength\tabcolsep{0pt}%
    \put(0,0){\includegraphics[width=\unitlength,page=1]{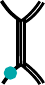}}%
  \end{picture}%
\endgroup%
})= \raisebox{-0.8\baselineskip}{\def\svgwidth{20pt}
\begingroup%
  \makeatletter%
  \providecommand\color[2][]{%
    \errmessage{(Inkscape) Color is used for the text in Inkscape, but the package 'color.sty' is not loaded}%
    \renewcommand\color[2][]{}%
  }%
  \providecommand\transparent[1]{%
    \errmessage{(Inkscape) Transparency is used (non-zero) for the text in Inkscape, but the package 'transparent.sty' is not loaded}%
    \renewcommand\transparent[1]{}%
  }%
  \providecommand\rotatebox[2]{#2}%
  \newcommand*\fsize{\dimexpr\f@size pt\relax}%
  \newcommand*\lineheight[1]{\fontsize{\fsize}{#1\fsize}\selectfont}%
  \ifx\svgwidth\undefined%
    \setlength{\unitlength}{32.43884546bp}%
    \ifx\svgscale\undefined%
      \relax%
    \else%
      \setlength{\unitlength}{\unitlength * \real{\svgscale}}%
    \fi%
  \else%
    \setlength{\unitlength}{\svgwidth}%
  \fi%
  \global\let\svgwidth\undefined%
  \global\let\svgscale\undefined%
  \makeatother%
  \begin{picture}(1,1.21529308)%
    \lineheight{1}%
    \setlength\tabcolsep{0pt}%
    \put(0,0){\includegraphics[width=\unitlength,page=1]{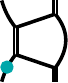}}%
  \end{picture}%
\endgroup%
}$. By hypothesis $\reflectbox{\raisebox{-0.8\baselineskip}{\def\svgwidth{13pt}}}$ is d.u.r., so we can write $\raisebox{-0.8\baselineskip}{\def\svgwidth{13pt}}$ as a linear combination of the d.u.r.\ decorated graphs: $\raisebox{-0.8\baselineskip}{\def\svgwidth{13pt}} = \displaystyle \sum_{\mathbf{a_0}} b_{\mathbf{a_0}}\gamma_0^{\mathbf{a_0}} + \sum_{\mathbf{a_1}} b_{\mathbf{a_1}}\gamma_{e_1}^{\mathbf{a_1}} + \sum_{\mathbf{a_2}} b_{\mathbf{a_2}}\gamma_{e_2}^{\mathbf{a_2}}$, with $b_{\mathbf{a_i}} \in \Z$ (see \cref{rem:graphrepsarelocal}).
Applying $\phi_1$:
$$\raisebox{-0.8\baselineskip}{\def\svgwidth{20pt}} = \sum_{\mathbf{a_0}} b_{\mathbf{a_0}}\gamma_{0,0}^{\mathbf{a_0}} + \sum_{\mathbf{a_1}} b_{\mathbf{a_1}}(\gamma_{0,1}^{\mathbf{a_1}}+\gamma_{1,0}^{\mathbf{a_1}}) + \sum_{\mathbf{a_2}} b_{\mathbf{a_2}}\gamma_{1,1}^{\mathbf{a_2}}.$$ 
Now \cref{eq:imagebasisid17reverse} yields a basis of $\reflectbox {\raisebox{-0.8\baselineskip}{\def\svgwidth{20pt}}}$ where each basis element is a linear combination of d.u.r.\ decorated graphs. One can then easily find a change of basis turning this basis into the d.u.r.\ set of $\reflectbox {\raisebox{-0.8\baselineskip}{\def\svgwidth{20pt}}}$, which proves that this graph is indeed d.u.r.

We now prove point (1) of the statement. Suppose that $\raisebox{-0.8\baselineskip}{\def\svgwidth{20pt}}$ is d.u.r., with d.u.r.\ basis grouped into the four families $\gamma_{0,0},\gamma_{0,1},\gamma_{1,0},\gamma_{1,1}$ (see left-hand side of \cref{eq:imagebasisid17l}). Similarly to the previous point, applying the isomorphism $\psi=\left(\begin{array}{c}
    \psi_0 \\
    \psi_1
\end{array}\right)$ 
that realises identity \cref{eq:id17general} yields
\begin{equation} \label{eq:imagebasisid17l}
{\small \raisebox{-6.1\baselineskip}{\def\svgwidth{200pt}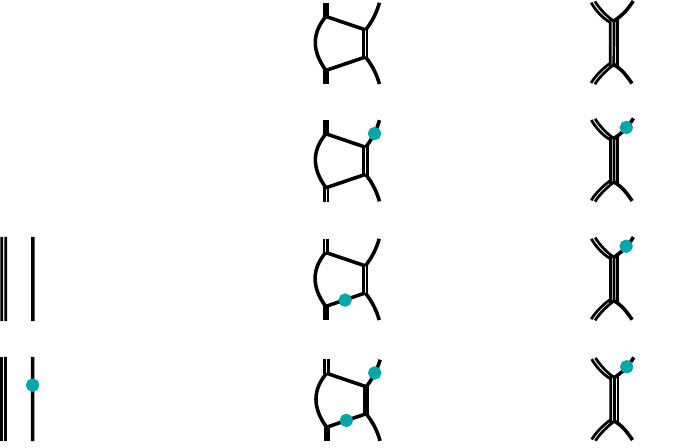}}
\end{equation}

We would like to express $\raisebox{-0.8\baselineskip}{\def\svgwidth{23pt}
\begingroup%
  \makeatletter%
  \providecommand\color[2][]{%
    \errmessage{(Inkscape) Color is used for the text in Inkscape, but the package 'color.sty' is not loaded}%
    \renewcommand\color[2][]{}%
  }%
  \providecommand\transparent[1]{%
    \errmessage{(Inkscape) Transparency is used (non-zero) for the text in Inkscape, but the package 'transparent.sty' is not loaded}%
    \renewcommand\transparent[1]{}%
  }%
  \providecommand\rotatebox[2]{#2}%
  \newcommand*\fsize{\dimexpr\f@size pt\relax}%
  \newcommand*\lineheight[1]{\fontsize{\fsize}{#1\fsize}\selectfont}%
  \ifx\svgwidth\undefined%
    \setlength{\unitlength}{39.08274274bp}%
    \ifx\svgscale\undefined%
      \relax%
    \else%
      \setlength{\unitlength}{\unitlength * \real{\svgscale}}%
    \fi%
  \else%
    \setlength{\unitlength}{\svgwidth}%
  \fi%
  \global\let\svgwidth\undefined%
  \global\let\svgscale\undefined%
  \makeatother%
  \begin{picture}(1,1.03418102)%
    \lineheight{1}%
    \setlength\tabcolsep{0pt}%
    \put(0,0){\includegraphics[width=\unitlength,page=1]{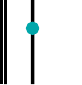}}%
    \put(0.53286808,0.60135367){\color[rgb]{0,0.65490196,0.66666667}\makebox(0,0)[lt]{\lineheight{1.25}\smash{\begin{tabular}[t]{l}$1$\end{tabular}}}}%
  \end{picture}%
\endgroup%
} = \psi_0(\gamma_{1,1})$ in terms of the d.u.r.\ set of $\mathscr{S}_1(\raisebox{-0.8\baselineskip}{\def\svgwidth{10pt}})$. 
We already know this is possible over $\Q$: $\psi(\gamma_{0,1}+\gamma_{1,0})=\raisebox{-0.8\baselineskip}{\def\svgwidth{10pt}}=\gamma$ constitutes the d.u.r.\ set of $\mathscr{S}_1(\raisebox{-0.8\baselineskip}{\def\svgwidth{10pt}})$, and it is a basis over $\Q$ since it's a generator and has the right dimension. So 
$$\raisebox{-0.8\baselineskip}{\def\svgwidth{23pt}} = \sum_{\mathbf{a}} b_{\mathbf{a}} \cdot \gamma^{\mathbf{a}}$$ 
with $b_{\mathbf{a}} \in \Q$. 
Suppose now that $\gamma$ is not a basis of $\mathscr{S}_1(\raisebox{-0.8\baselineskip}{\def\svgwidth{10pt}})$ over $\Z$. Then $b_{\mathbf{a}} \in \Q - \Z$ for some $\mathbf{a}$. This implies that 
$$\phi_0(\quad\raisebox{-0.8\baselineskip}{\def\svgwidth{23pt}}) = \phi_0(\sum_{\mathbf{a}} b_{\mathbf{a}} \cdot \gamma^{\mathbf{a}}) = \sum_{\mathbf{a}} b_{\mathbf{a}} \cdot \phi_0(\gamma^{\mathbf{a}})=\sum_{\mathbf{a}} b_{\mathbf{a}} (\gamma_{0,1}^{\mathbf{a}}-\gamma_{1,0}^{\mathbf{a}}),$$
where the last equality is obtained by following the movie \cref{eq:moviephi0alpha} for $\phi_0$, and using the simplified version of the \emph{zip} map given in \cref{eq:zipelementary}.
But the $\gamma_{i,j}^{\mathbf{a}}$ are a d.u.r.\ basis of $\mathscr{S}_1(\raisebox{-0.8\baselineskip}{\def\svgwidth{20pt}})$ over $\Z$, so $b_{\mathbf{a}} \in \Z$ for all $\mathbf{a}$, which contradicts our assumption. Therefore $\raisebox{-0.8\baselineskip}{\def\svgwidth{10pt}}=\psi(\gamma_{0,1}+\gamma_{1,0})$ is indeed a basis of $\mathscr{S}_1(\raisebox{-0.8\baselineskip}{\def\svgwidth{10pt}})$ over $\Z$, and the images under $\psi$ of $\gamma_{0,0},\, \gamma_{0,1}$ and $\gamma_{1,1}-\sum_{\mathbf{a}} b_{\mathbf{a}} (\gamma_{0,1}+\gamma_{1,0})$ give the d.u.r.\ basis of $\mathscr{S}_1(\raisebox{-0.8\baselineskip}{\def\svgwidth{13pt}})$.

If, conversely, $\raisebox{-0.8\baselineskip}{\def\svgwidth{13pt}}$ and $\raisebox{-0.8\baselineskip}{\def\svgwidth{10pt}}$ are d.u.r., 
the images of their d.u.r.\ bases under $\phi_0$ and $\phi_1$ are:
\begin{equation} \label{eq:imagebasisid17rereverse}
{\small \raisebox{-6.1\baselineskip}{\def\svgwidth{200pt}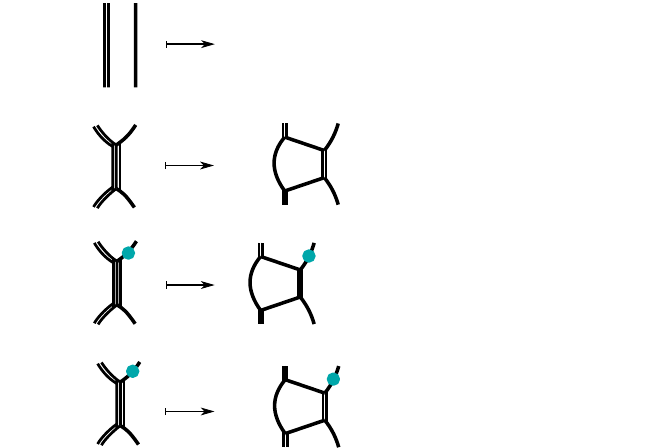}}
\end{equation}

By similar arguments to the previous cases, we can express $\gamma_{0,2}$ in terms of the d.u.r.\ set, by noting that $\gamma_{0,2}-\gamma_{1,1} = \phi_0(\quad\raisebox{-0.8\baselineskip}{\def\svgwidth{23pt}})$, and expressing $\quad \raisebox{-0.8\baselineskip}{\def\svgwidth{23pt}}$ in terms of the d.u.r.\ basis of $\raisebox{-0.8\baselineskip}{\def\svgwidth{10pt}}$. Then one easily finds a change of basis turning the basis of \cref{eq:imagebasisid17rereverse} into the d.u.r.\ set.

\end{proof}

\begin{lem} \label{lem:treeisopreservesdur}
When $a=b=c=1$, the LHS of the \textit{tree isomorphism} \cref{eq:treeiso} is d.u.r.\ if and only if the RHS is.
\end{lem}

\begin{proof}
D.u.r.\ sets of the graphs on the LHS and RHS of \cref{eq:treeiso} are given locally by
\begin{equation*}
{\small \raisebox{-1\baselineskip}{\def\svgwidth{170pt}
\begingroup%
  \makeatletter%
  \providecommand\color[2][]{%
    \errmessage{(Inkscape) Color is used for the text in Inkscape, but the package 'color.sty' is not loaded}%
    \renewcommand\color[2][]{}%
  }%
  \providecommand\transparent[1]{%
    \errmessage{(Inkscape) Transparency is used (non-zero) for the text in Inkscape, but the package 'transparent.sty' is not loaded}%
    \renewcommand\transparent[1]{}%
  }%
  \providecommand\rotatebox[2]{#2}%
  \newcommand*\fsize{\dimexpr\f@size pt\relax}%
  \newcommand*\lineheight[1]{\fontsize{\fsize}{#1\fsize}\selectfont}%
  \ifx\svgwidth\undefined%
    \setlength{\unitlength}{229.14147574bp}%
    \ifx\svgscale\undefined%
      \relax%
    \else%
      \setlength{\unitlength}{\unitlength * \real{\svgscale}}%
    \fi%
  \else%
    \setlength{\unitlength}{\svgwidth}%
  \fi%
  \global\let\svgwidth\undefined%
  \global\let\svgscale\undefined%
  \makeatother%
  \begin{picture}(1,0.14217686)%
    \lineheight{1}%
    \setlength\tabcolsep{0pt}%
    \put(0,0){\includegraphics[width=\unitlength,page=1]{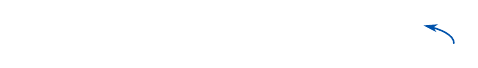}}%
    \put(0.92721217,0.00612902){\color[rgb]{0,0.32156863,0.67058824}\makebox(0,0)[lt]{\lineheight{1.25}\smash{\begin{tabular}[t]{l}$0,1$\end{tabular}}}}%
    \put(0.52226646,0.02863401){\makebox(0,0)[lt]{\lineheight{1.25}\smash{\begin{tabular}[t]{l}$,$\end{tabular}}}}%
    \put(0.67594008,0.03309918){\color[rgb]{0,0.32156863,0.67058824}\makebox(0,0)[lt]{\lineheight{1.25}\smash{\begin{tabular}[t]{l}$0,1$\end{tabular}}}}%
    \put(0,0){\includegraphics[width=\unitlength,page=2]{basisidentity13.pdf}}%
    \put(0.37169813,0.0588792){\color[rgb]{0,0.32156863,0.67058824}\makebox(0,0)[lt]{\lineheight{1.25}\smash{\begin{tabular}[t]{l}$0,1$\end{tabular}}}}%
    \put(-0.00217353,0.00773275){\color[rgb]{0,0.32156863,0.67058824}\makebox(0,0)[lt]{\lineheight{1.25}\smash{\begin{tabular}[t]{l}$e_0,e_1,e_2$\end{tabular}}}}%
    \put(0,0){\includegraphics[width=\unitlength,page=3]{basisidentity13.pdf}}%
  \end{picture}%
\endgroup%
}}
\end{equation*}

Dot migration \cref{eq:dotmigration} and the \emph{tree} isomorphism \cref{eq:treeiso} yield the following identities, which prove the statement.
\begin{equation*}
\raisebox{-1.8\baselineskip}{\def\svgwidth{330pt}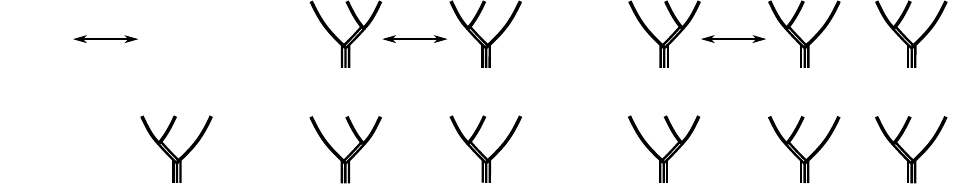}
\end{equation*}

\end{proof}

In the previous section we proved that $\Theta_{\ell}=\Gamma_{0,\ell}$ is d.u.r. The following result from \cite{BPRW} holds\footnote{It can be proved by induction on the number of strands of a graph, by using Jones' basis of the Hecke algebra.}.

\begin{prop}[\cite{BPRW}, Proposition 2.25] \label{lem:moveselementary}
Let $\Gamma$ be an elementary vinyl graph.
For any $\mathbf{u}=(\ell^1,\ldots , \ell^n)$, with $\ell^i \in \N$ $\forall i$, define $\Theta_{\mathbf{u}}=\bigsqcup_i \Theta_{\ell^i}$.
Then 
\begin{enumerate}
    \item There exist two $\N[q,q^{-1}]$-linear combinations $\sum_{i=1}^m a_i\Theta_{\mathbf{u_i}}$ and $\sum_{i=1}^{m'} a'_i\Theta_{\mathbf{u'_i}}$ such that 
    \begin{equation}
    \mathscr{S}_1(\Gamma) \oplus \bigoplus_{i=1}^m a_i\mathscr{S}_1(\Theta_{\mathbf{u_i}}) \oplus b\mathscr{S}_1(C) \simeq \bigoplus_{i=1}^{m'} a'_i\mathscr{S}_1(\Theta_{\mathbf{u'_i}}) \oplus b'\mathscr{S}_1(C')
    \end{equation}
    for some $b,b' \in \N[q,q^{-1}]$ and $C,C'$ disjoint unions of circles (the coefficients $a_i,a'_i,b,b'$ represent degree shifts).
    \item This isomorphism is achieved via a finite sequence of the moves \cref{eq:move1} and \cref{eq:move2} below.

    \begin{equation} \label{eq:move1}
    \mathscr{S}_1(\;\raisebox{-1.3\baselineskip}{\def\svgwidth{11pt}
\begingroup%
  \makeatletter%
  \providecommand\color[2][]{%
    \errmessage{(Inkscape) Color is used for the text in Inkscape, but the package 'color.sty' is not loaded}%
    \renewcommand\color[2][]{}%
  }%
  \providecommand\transparent[1]{%
    \errmessage{(Inkscape) Transparency is used (non-zero) for the text in Inkscape, but the package 'transparent.sty' is not loaded}%
    \renewcommand\transparent[1]{}%
  }%
  \providecommand\rotatebox[2]{#2}%
  \newcommand*\fsize{\dimexpr\f@size pt\relax}%
  \newcommand*\lineheight[1]{\fontsize{\fsize}{#1\fsize}\selectfont}%
  \ifx\svgwidth\undefined%
    \setlength{\unitlength}{14.37280568bp}%
    \ifx\svgscale\undefined%
      \relax%
    \else%
      \setlength{\unitlength}{\unitlength * \real{\svgscale}}%
    \fi%
  \else%
    \setlength{\unitlength}{\svgwidth}%
  \fi%
  \global\let\svgwidth\undefined%
  \global\let\svgscale\undefined%
  \makeatother%
  \begin{picture}(1,3.48439208)%
    \lineheight{1}%
    \setlength\tabcolsep{0pt}%
    \put(0,0){\includegraphics[width=\unitlength,page=1]{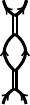}}%
  \end{picture}%
\endgroup%
}\;) \;\simeq\; [2] \mathscr{S}_1(\raisebox{-1.1\baselineskip}{\def\svgwidth{15pt}
\begingroup%
  \makeatletter%
  \providecommand\color[2][]{%
    \errmessage{(Inkscape) Color is used for the text in Inkscape, but the package 'color.sty' is not loaded}%
    \renewcommand\color[2][]{}%
  }%
  \providecommand\transparent[1]{%
    \errmessage{(Inkscape) Transparency is used (non-zero) for the text in Inkscape, but the package 'transparent.sty' is not loaded}%
    \renewcommand\transparent[1]{}%
  }%
  \providecommand\rotatebox[2]{#2}%
  \newcommand*\fsize{\dimexpr\f@size pt\relax}%
  \newcommand*\lineheight[1]{\fontsize{\fsize}{#1\fsize}\selectfont}%
  \ifx\svgwidth\undefined%
    \setlength{\unitlength}{22.06656412bp}%
    \ifx\svgscale\undefined%
      \relax%
    \else%
      \setlength{\unitlength}{\unitlength * \real{\svgscale}}%
    \fi%
  \else%
    \setlength{\unitlength}{\svgwidth}%
  \fi%
  \global\let\svgwidth\undefined%
  \global\let\svgscale\undefined%
  \makeatother%
  \begin{picture}(1,2.27114957)%
    \lineheight{1}%
    \setlength\tabcolsep{0pt}%
    \put(0,0){\includegraphics[width=\unitlength,page=1]{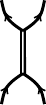}}%
  \end{picture}%
\endgroup%
})
    \end{equation}
    \begin{equation} \label{eq:move2}
    \mathscr{S}_1(\;\raisebox{-1.33\baselineskip}{\def\svgwidth{22pt}
\begingroup%
  \makeatletter%
  \providecommand\color[2][]{%
    \errmessage{(Inkscape) Color is used for the text in Inkscape, but the package 'color.sty' is not loaded}%
    \renewcommand\color[2][]{}%
  }%
  \providecommand\transparent[1]{%
    \errmessage{(Inkscape) Transparency is used (non-zero) for the text in Inkscape, but the package 'transparent.sty' is not loaded}%
    \renewcommand\transparent[1]{}%
  }%
  \providecommand\rotatebox[2]{#2}%
  \newcommand*\fsize{\dimexpr\f@size pt\relax}%
  \newcommand*\lineheight[1]{\fontsize{\fsize}{#1\fsize}\selectfont}%
  \ifx\svgwidth\undefined%
    \setlength{\unitlength}{36.43002027bp}%
    \ifx\svgscale\undefined%
      \relax%
    \else%
      \setlength{\unitlength}{\unitlength * \real{\svgscale}}%
    \fi%
  \else%
    \setlength{\unitlength}{\svgwidth}%
  \fi%
  \global\let\svgwidth\undefined%
  \global\let\svgscale\undefined%
  \makeatother%
  \begin{picture}(1,1.79697212)%
    \lineheight{1}%
    \setlength\tabcolsep{0pt}%
    \put(0,0){\includegraphics[width=\unitlength,page=1]{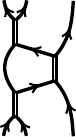}}%
  \end{picture}%
\endgroup%
}\;) \oplus \mathscr{S}_1(\;\raisebox{-1.16\baselineskip}{\def\svgwidth{24pt}
\begingroup%
  \makeatletter%
  \providecommand\color[2][]{%
    \errmessage{(Inkscape) Color is used for the text in Inkscape, but the package 'color.sty' is not loaded}%
    \renewcommand\color[2][]{}%
  }%
  \providecommand\transparent[1]{%
    \errmessage{(Inkscape) Transparency is used (non-zero) for the text in Inkscape, but the package 'transparent.sty' is not loaded}%
    \renewcommand\transparent[1]{}%
  }%
  \providecommand\rotatebox[2]{#2}%
  \newcommand*\fsize{\dimexpr\f@size pt\relax}%
  \newcommand*\lineheight[1]{\fontsize{\fsize}{#1\fsize}\selectfont}%
  \ifx\svgwidth\undefined%
    \setlength{\unitlength}{44.39299239bp}%
    \ifx\svgscale\undefined%
      \relax%
    \else%
      \setlength{\unitlength}{\unitlength * \real{\svgscale}}%
    \fi%
  \else%
    \setlength{\unitlength}{\svgwidth}%
  \fi%
  \global\let\svgwidth\undefined%
  \global\let\svgscale\undefined%
  \makeatother%
  \begin{picture}(1,1.48598566)%
    \lineheight{1}%
    \setlength\tabcolsep{0pt}%
    \put(0,0){\includegraphics[width=\unitlength,page=1]{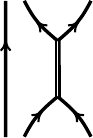}}%
  \end{picture}%
\endgroup%
}) \;\simeq\; \mathscr{S}_1(\reflectbox{\raisebox{-1.16\baselineskip}{\def\svgwidth{24pt}}}\;) \oplus \mathscr{S}_1(\;\reflectbox{\raisebox{-1.33\baselineskip}{\def\svgwidth{22pt}}}\;)
    \end{equation}
\end{enumerate}

\end{prop}

It follows from \cref{lem:moveselementary} that in order to prove that elementary graphs are d.u.r., it is sufficient to show that the two moves \cref{eq:move1} and \cref{eq:move2} preserve the d.u.r.\ property. We are ready to prove \cref{thm:d.u.r.}.

\begin{proof}[Proof of \cref{thm:d.u.r.}]
We showed in \cref{prop:gammaklisdur} that the vinyl graph $\Theta_{\ell}$ is d.u.r. By \cref{lem:moveselementary} the only thing left to do is prove that the isomorphisms of move \cref{eq:move1} and move \cref{eq:move2} map d.u.r.\ bases to d.u.r.\ bases. For the first move, this follows directly from \cref{lem:un-bubblepreservesdur}. As for the second move, isomorphism \cref{eq:move2} is obtained by composing identities \cref{eq:id17general}, \cref{eq:treeiso} and the mirror of \cref{eq:id17general}. It is therefore enough to show that the property of having a d.u.r.\ basis is preserved by these two identities (and their mirrors). This is the object of \cref{lem:id17preservesdur} and \cref{lem:treeisopreservesdur}. 

\end{proof}

\section{Computing the homology}
\label{sec:comp-gll_1-homol}

In this section we describe an algorithm and a program which computes the non-equivariant symmetric
$\mathfrak{gl}_1$-homology of uncolored links. The program is made public in \cite{foam-gl1}.
First, let us give an outline of the definition of the symmetric $\mathfrak{gl}_1$-homology. 

\subsection{The symmetric $\mathfrak{gl}_1$-homology} \label{sec:homology}

In \cref{sec:gl1eval} we reviewed the construction of the $R$-modules $\mathscr{S}_1(\Gamma)$ associated to vinyl graphs $\Gamma$, for $R=\Q,\,\Z$ or $\mathbb{F}_p$, with $p$ prime. These modules are a key component in the definition of the symmetric $\mathfrak{gl}_1$-homology of links. We sketch the construction of this homological invariant here for uncolored links.  More details can be found in \cite{zbMATH07206869,zbMATH07632766}. 

Start with a braid diagram $\beta$ whose closure $\widehat{\beta}$ is a representative of an uncolored link $L$. If $\beta$ has $n$ crossings $x_1,\ldots, x_n$, let $n_+$ (resp.\ $n_-$) be the number of positive (resp.\ negative) ones. For each crossing, consider the two resolutions below:

\begin{equation*}
{\scriptsize \raisebox{-0\baselineskip}{\def\svgwidth{180pt}
\begingroup%
  \makeatletter%
  \providecommand\color[2][]{%
    \errmessage{(Inkscape) Color is used for the text in Inkscape, but the package 'color.sty' is not loaded}%
    \renewcommand\color[2][]{}%
  }%
  \providecommand\transparent[1]{%
    \errmessage{(Inkscape) Transparency is used (non-zero) for the text in Inkscape, but the package 'transparent.sty' is not loaded}%
    \renewcommand\transparent[1]{}%
  }%
  \providecommand\rotatebox[2]{#2}%
  \newcommand*\fsize{\dimexpr\f@size pt\relax}%
  \newcommand*\lineheight[1]{\fontsize{\fsize}{#1\fsize}\selectfont}%
  \ifx\svgwidth\undefined%
    \setlength{\unitlength}{607.4238756bp}%
    \ifx\svgscale\undefined%
      \relax%
    \else%
      \setlength{\unitlength}{\unitlength * \real{\svgscale}}%
    \fi%
  \else%
    \setlength{\unitlength}{\svgwidth}%
  \fi%
  \global\let\svgwidth\undefined%
  \global\let\svgscale\undefined%
  \makeatother%
  \begin{picture}(1,0.43646625)%
    \lineheight{1}%
    \setlength\tabcolsep{0pt}%
    \put(0,0){\includegraphics[width=\unitlength,page=1]{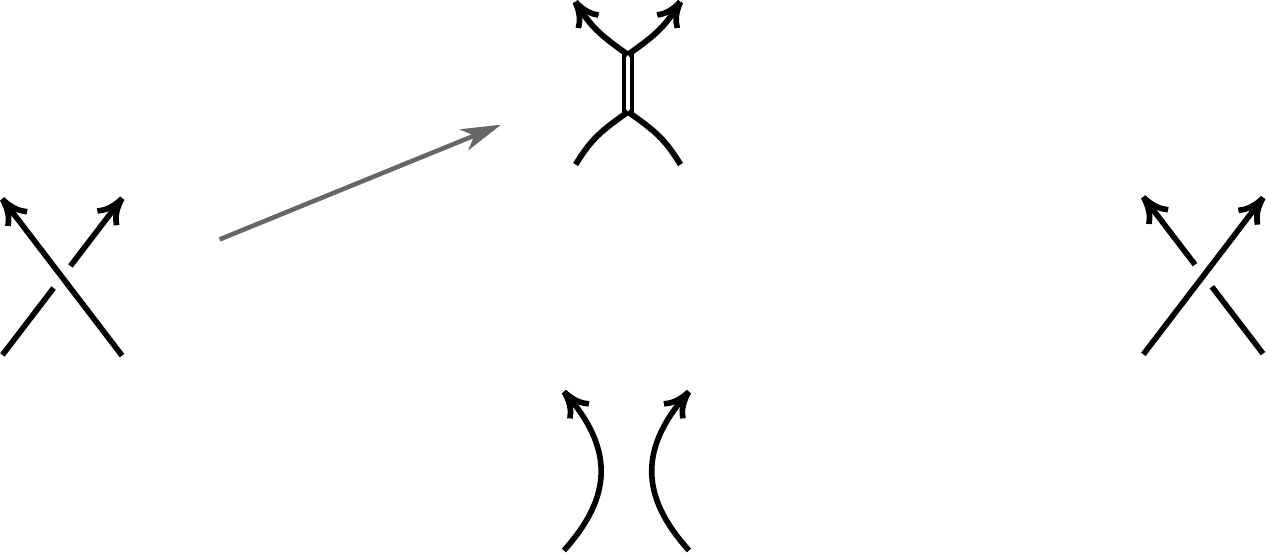}}%
    \put(0.16773049,0.2798184){\color[rgb]{0.4,0.4,0.4}\rotatebox{22.445115}{\makebox(0,0)[lt]{\lineheight{1.25}\smash{\begin{tabular}[t]{l}0-resolution\end{tabular}}}}}%
    \put(0.15806163,0.13594517){\color[rgb]{0.4,0.4,0.4}\rotatebox{-22.408164}{\makebox(0,0)[lt]{\lineheight{1.25}\smash{\begin{tabular}[t]{l}1-resolution\end{tabular}}}}}%
    \put(0,0){\includegraphics[width=\unitlength,page=2]{resolutions.pdf}}%
    \put(0.63054925,0.0468867){\color[rgb]{0.4,0.4,0.4}\rotatebox{22.445115}{\makebox(0,0)[lt]{\lineheight{1.25}\smash{\begin{tabular}[t]{l}0-resolution\end{tabular}}}}}%
    \put(0,0){\includegraphics[width=\unitlength,page=3]{resolutions.pdf}}%
    \put(0.62232331,0.36273587){\color[rgb]{0.4,0.4,0.4}\rotatebox{-22.408164}{\makebox(0,0)[lt]{\lineheight{1.25}\smash{\begin{tabular}[t]{l}1-resolution\end{tabular}}}}}%
  \end{picture}%
\endgroup%
} }   
\end{equation*}

Let $v=(v_1,\ldots,v_n) \in \{0,1\}^{n}$ and $|v|=\displaystyle\sum_i v_i$. Replacing each crossing $x_i$ of $\widehat{\beta}$ with its $v_i$-resolution produces an elementary vinyl graph $\widehat{\beta_v}$, called a \emph{complete resolution} of $\widehat{\beta}$. One constructs a \emph{hypercube of resolutions} whose vertices are the state spaces $\widetilde{\mathscr{S}}_1(\widehat{\beta_v}) = q^{2n_- - n_+ -|v|}\mathscr{S}_1(\widehat{\beta_v})$, and whose edges are instances of the \emph{zip} \cref{eq:zipelementary} or the \emph{unzip} \cref{eq:unzip} maps. More precisely, if $v,v'$ differ only on one entry $v_i \ne v_i'$, and $|v'| = |v| + 1$, we define the $R$-module morphism $d_{v,v'}\thinspace\colon \widetilde{\mathscr{S}}_1(\widehat{\beta_v})) \to \widetilde{\mathscr{S}}_1(\widehat{\beta_v}))$ as
\[
\begin{split}
d_{v,v'} = \textit{zip } \qquad \text{ if $x_i$ is positive,} \\
d_{v,v'} = \textit{unzip } \qquad \text{ if $x_i$ is negative.}
\end{split}
\]
Flattening the hypercube of resolutions gives rise to a chain complex of graded $R$-modules: the chain spaces and differentials are
\[
\bigoplus_{|v|=i+n_-} \widetilde{\mathscr{S}}_1(\widehat{\beta_v})) = C_{i}(\widehat{\beta})), \qquad \qquad \sum_{\substack{|v|=i+n_-, \\ |v'|=i+1+n_-}} d_{v,v'}\; = \; d_i \; \colon \; C_{i}(\widehat{\beta})) \to C_{i+1}(\widehat{\beta}))
\]
When $R=\Q$ this complex is a link invariant up to chain homotopy equivalence. Its homology is thus a link invariant, the (non-equivariant) \emph{symmetric $\mathfrak{gl}_1$-homology} of $L$. For $R=\Z$ or $\mathbb{F}_p$, however, as we will see in \cref{ex:notinvariant}, the chain complex is not invariant under the Reidemeister I move, and is thus only a framed link invariant.

\subsection{The program {\tt foam-gl1.py}}

The goal of the rest of this section is to give details on how our computer program is constructed and on the algorithm that underlies it. The program {\tt foam-gl1.py} is written in python3 and uses the
libraries NumPy \cite{harris2020array} and PyTorch \cite{torch}. It takes as input a braid representative $\beta$ of an uncolored link $L$, expressed either as a sequence of integers or as a word made up of uppercase (positive crossings) and lowercase (negative crossings) letters. A typical use is:
\[
  \text{\tt> python3 foam-gl1.py AbAb}
\]
In this line, \texttt{AbAb} represents the braid
$\sigma_1\sigma_2^{-1}\sigma_1\sigma_2^{-1}$ and the program computes
the (non-equivariant) symmetric $\mathfrak{gl}_1$-homology of its closure, that is the figure-eight knot. 

The main task of the program is to compute the differentials of the chain complex $C_*(\widehat{\beta}))$ and their ranks. 
Due to the nature of $\mathscr{S}_1(\cdot)$ (see \cref{def:statespaces}), the first step is to compute the evaluation, and therefore the induced pairing, of elementary decorated graphs (i.e.\ decorated graphs whose underlying vinyl graphs are elementary).

\subsection{Computing the evaluation}
\label{sec:computing-evaluation}

Let $\Gamma^d$ be an elementary decorated graph of level $k$ which is a complete resolution of $\widehat{\beta}$. The 1-evaluation $\langle \Gamma^d \rangle_1$ is obtained by setting all variables of the $\infty$-evaluation to zero (see \cref{def:evaldecoratedgraphs}). \cref{rem:degreemaxdecoration} allows us to significantly reduce the number of calculations we need to carry out: indeed, it follows from \cref{eq:degreeeval1} that we only need to compute the evaluation for decorated graphs of degree $t_\Gamma$ (as $\langle \Gamma^d \rangle_1$ is zero otherwise). Moreover, for such decorated graphs the $\infty$-evaluation is already a degree $0$ polynomial, and is therefore equal to the 1-evaluation.

In our program, $\Gamma^d_{\text{open}}$ is viewed as a composition of pieces of two kinds: \emph{monomial slices} are $k$ vertical 1-labelled strands with decorations by homogeneous symmetric polynomials (since the strands all have label 1, the decorations are all powers of $x$), and \emph{dumbbells} are $k-2$ vertical 1-labelled strands with a $\raisebox{-0.7\baselineskip}{\def\svgwidth{11pt}
\begingroup%
  \makeatletter%
  \providecommand\color[2][]{%
    \errmessage{(Inkscape) Color is used for the text in Inkscape, but the package 'color.sty' is not loaded}%
    \renewcommand\color[2][]{}%
  }%
  \providecommand\transparent[1]{%
    \errmessage{(Inkscape) Transparency is used (non-zero) for the text in Inkscape, but the package 'transparent.sty' is not loaded}%
    \renewcommand\transparent[1]{}%
  }%
  \providecommand\rotatebox[2]{#2}%
  \newcommand*\fsize{\dimexpr\f@size pt\relax}%
  \newcommand*\lineheight[1]{\fontsize{\fsize}{#1\fsize}\selectfont}%
  \ifx\svgwidth\undefined%
    \setlength{\unitlength}{19.27925477bp}%
    \ifx\svgscale\undefined%
      \relax%
    \else%
      \setlength{\unitlength}{\unitlength * \real{\svgscale}}%
    \fi%
  \else%
    \setlength{\unitlength}{\svgwidth}%
  \fi%
  \global\let\svgwidth\undefined%
  \global\let\svgscale\undefined%
  \makeatother%
  \begin{picture}(1,1.83831597)%
    \lineheight{1}%
    \setlength\tabcolsep{0pt}%
    \put(0,0){\includegraphics[width=\unitlength,page=1]{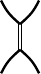}}%
  \end{picture}%
\endgroup%
}$ in a specified position. The monomial slice of \cref{eq:mondumb} is encoded by the list $[3,0,0,1,2,0,0]$, while the dumbbell is represented by the integer $4$.

\begin{equation} \label{eq:mondumb}
\raisebox{-1.1\baselineskip}{\def\svgwidth{300pt}
\begingroup%
  \makeatletter%
  \providecommand\color[2][]{%
    \errmessage{(Inkscape) Color is used for the text in Inkscape, but the package 'color.sty' is not loaded}%
    \renewcommand\color[2][]{}%
  }%
  \providecommand\transparent[1]{%
    \errmessage{(Inkscape) Transparency is used (non-zero) for the text in Inkscape, but the package 'transparent.sty' is not loaded}%
    \renewcommand\transparent[1]{}%
  }%
  \providecommand\rotatebox[2]{#2}%
  \newcommand*\fsize{\dimexpr\f@size pt\relax}%
  \newcommand*\lineheight[1]{\fontsize{\fsize}{#1\fsize}\selectfont}%
  \ifx\svgwidth\undefined%
    \setlength{\unitlength}{264.4613716bp}%
    \ifx\svgscale\undefined%
      \relax%
    \else%
      \setlength{\unitlength}{\unitlength * \real{\svgscale}}%
    \fi%
  \else%
    \setlength{\unitlength}{\svgwidth}%
  \fi%
  \global\let\svgwidth\undefined%
  \global\let\svgscale\undefined%
  \makeatother%
  \begin{picture}(1,0.0995188)%
    \lineheight{1}%
    \setlength\tabcolsep{0pt}%
    \put(0,0){\includegraphics[width=\unitlength,page=1]{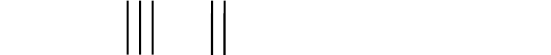}}%
    \put(-0.0012555,0.04403753){\color[rgb]{0,0,0}\makebox(0,0)[lt]{\lineheight{1.25}\smash{\begin{tabular}[t]{l}$\text{a dumbbell:}$\end{tabular}}}}%
    \put(0,0){\includegraphics[width=\unitlength,page=2]{mondumb.pdf}}%
    \put(0.52104879,0.04403744){\color[rgb]{0,0,0}\makebox(0,0)[lt]{\lineheight{1.25}\smash{\begin{tabular}[t]{l}$\text{a monomial slice:}$\end{tabular}}}}%
    \put(0,0){\includegraphics[width=\unitlength,page=3]{mondumb.pdf}}%
    \put(0.79118637,0.07570485){\color[rgb]{0.96470588,0.4,0.36470588}\makebox(0,0)[lt]{\lineheight{1.25}\smash{\begin{tabular}[t]{l}$3$\end{tabular}}}}%
    \put(0.91787969,0.046914){\color[rgb]{0.96470588,0.4,0.36470588}\makebox(0,0)[lt]{\lineheight{1.25}\smash{\begin{tabular}[t]{l}$1$\end{tabular}}}}%
    \put(0,0){\includegraphics[width=\unitlength,page=4]{mondumb.pdf}}%
    \put(0.94618111,0.08228153){\color[rgb]{0.96470588,0.4,0.36470588}\makebox(0,0)[lt]{\lineheight{1.25}\smash{\begin{tabular}[t]{l}$2$\end{tabular}}}}%
  \end{picture}%
\endgroup%
}
\end{equation}

These pieces are stacked vertically, by alternating a monomial slice and a dumbbell, so that the first and last pieces are monomial slices.

For each piece $\Psi$, the program (in the functions {\tt matrix\_dumbbell}, {\tt matrix\_monomial} and {\tt my\_matrix\_monomial}) computes a \emph{coloring matrix}: the rows (resp. columns) of this matrix are indexed by all possible colorings (\cref{def:coloring}) of the bottom (resp. top) strands of $\Psi$ (i.e.\ all permutations of $k$ elements), and the $(i,j)$-th entry encodes $\langle \Psi, c \rangle_{\infty}$ (see \cref{def:evaldecoratedgraphs}), where $c$ is the coloring of $\Psi$ determined by $i$ and $j$. 

Coloring matrices behave well with respect to the splitting of $\Gamma^d_{\text{open}}$: the coloring matrix of $\Gamma^d_{\text{open}}$ is just the composition (on the left as one stacks up) of the matrices of its pieces. Finally, the evaluation $\langle \Gamma^d \rangle_1$ is obtained by taking the trace of the coloring matrix of $\Gamma^d_{\text{open}}$. This is because the admissible colorings of $\Gamma^d$ (the closure of $\Gamma^d_{\text{open}}$) correspond to the colorings of $\Gamma^d_{\text{open}}$ where the pigments on the bottom and those on the top strands coincide. This is what the functions {\tt mat\_evaluation} and {\tt evaluation} compute. The function {\tt pairing} then determines the pairing $(\cdot\,,\cdot)_1$ of two decorated graphs.

\subsection{Computing the differentials}

In \cref{thm:d.u.r.} we showed that the d.u.r.\ set is a basis of the $R$-module $\mathscr{S}_1(\Gamma)$ for all elementary vinyl graphs $\Gamma$. Since the resolutions $\widehat{\beta}_v$ are elementary graphs ($L$ is uncolored), this implies that the d.u.r.\ set is a basis for the chain spaces $C_{i}(\widehat{\beta})$. One can thus express the differentials as matrices in these bases. 
The main issue one has to face when computing these matrices is how to express the images of the \emph{zip} and \emph{unzip} maps as linear combinations of the basis elements. Given $\mathbf{y}=\{y_1,\ldots,y_n\}$ and $\mathbf{y'}=\{y'_1,\ldots,y'_n\}$, let $(\mathbf{y},\mathbf{y'})_1$ be the matrix having as $(i,j)$-th entry the pairing $( y_i,y'_j )_1$ of \cref{def:pairing}. We call it the \emph{pairing matrix} of $\mathbf{y},\mathbf{y'}$. The following trick comes in handy:

\begin{lem} \label{lem:diffpairing}
Let $\widehat{\beta}_v,\widehat{\beta}_{v'}$ be two complete resolutions related by an edge $\widetilde{\mathscr{S}}_1(\widehat{\beta}_v) \overset{d_{v,v'}}{\longrightarrow} \widetilde{\mathscr{S}}_1(\widehat{\beta}_{v'})$ of the hypercube of resolutions and let $\mathbf{b}=\{b_1,\ldots,b_n\}$ be a basis of $\widetilde{\mathscr{S}}_1(\widehat{\beta}_{v'})$. Consider $u \in \widetilde{\mathscr{S}}_1(\widehat{\beta}_v)$ and let $d_{v,v'}(u)=\displaystyle\sum_i a_ib_i$, with $a_i \in R$, be the decomposition of $d_{v,v'}(u)$ as a linear combination of basis elements. We have:
\begin{equation*}
\displaystyle\sum_i a_ib_i = ( \mathbf{b},\mathbf{b} )_1^{-1} \cdot (d_{v,v'}(u),\mathbf{b})_1,
\end{equation*}
\end{lem}

\begin{proof}
The vector $(d_{v,v'}(u),\mathbf{b})_1$ is given by 
\[
(d_{v,v'}(u),\mathbf{b})_1 = (\displaystyle\sum_i a_ib_i,\mathbf{b})_1 = a_1 
\left( \begin{array}{c}
     (b_1,b_1)_1 \\
     \vdots \\
     (b_1,b_n)_1
\end{array} \right) + \ldots + a_n 
\left( \begin{array}{c}
     (b_n,b_1)_1 \\
     \vdots \\
     (b_n,b_n)_1
\end{array} \right).
\]
Since $\left( \begin{array}{c}
     (b_i,b_1)_1 \\
     \vdots \\
     (b_i,b_n)_1
\end{array} \right)$ is the $i$-th column of the matrix $( \mathbf{b},\mathbf{b} )_1$, it is clear that 
\[
( \mathbf{b},\mathbf{b} )_1^{-1} \cdot (d_{v,v'}(u),\mathbf{b})_1=\left( \begin{array}{c}
     a_1 \\
     \vdots \\
     a_n
\end{array} \right), 
\]
which proves the statement.
\end{proof}

Using \cref{lem:diffpairing}, computing the differentials boils down to computing pairing matrices.

Of the \emph{zip} and \emph{unzip} maps constituting the differentials $d_{v,v'}$, the latter is easier to work with, as it does not increase the number of decorations. The following statement shows that the pairing matrix corresponding to a \emph{zip} map is simply the transpose of that of an \emph{unzip}, with opposite signs.

\begin{lem} \label{lem:positivecx}
Let $v,v' \in \{0,1\}^n$ be equal on all entries but the $k$-th one, where 
\[
\begin{split}
v_k=1 \quad \text{and} \quad v'_k=0 \qquad \text{if } x_k \text{ is a positive crossing of } \beta  \\
v_k=0 \quad \text{and} \quad v'_k=1 \qquad \text{if } x_k \text{ is a negative crossing of } \beta 
\end{split}
\]
The state spaces $\widetilde{\mathscr{S}}_1(\widehat{\beta}_v)$ and $\widetilde{\mathscr{S}}_1(\widehat{\beta}_{v'})$ are thus related by an edge of the hypercube of resolutions. given by:
\[
\begin{split}
\widetilde{\mathscr{S}}_1(\widehat{\beta}_{v'}) \overset{\textit{zip}}{\to} \widetilde{\mathscr{S}}_1(\widehat{\beta}_{v}) \qquad \text{if } x_k \text{ is positive}  \\
\widetilde{\mathscr{S}}_1(\widehat{\beta}_{v}) \overset{\textit{unzip}}{\to} \widetilde{\mathscr{S}}_1(\widehat{\beta}_{v'}) \qquad \text{if } x_k \text{ is negative}
\end{split}
\]
Consider the d.u.r.\ bases $\mathbf{b}$ and $\mathbf{b'}$ of $\widetilde{\mathscr{S}}_1(\widehat{\beta}_v)$ and $\widetilde{\mathscr{S}}_1(\widehat{\beta}_{v'})$ respectively. 
The pairing matrices $(\textit{unzip}(\mathbf{b}),\mathbf{b'})_1$ and $-(\textit{zip}(\mathbf{b'}),\mathbf{b})_1$ are the transpose of one another.
\end{lem}

\begin{proof}
It is enough\footnote{Viewing the maps \emph{zip} and \emph{unzip} as foams, this statement is trivial: the \emph{zip} foam is the \emph{unzip} foam flipped upside-down, and one clearly sees that the two pairing matrices compute the evaluations of the same foams} to show that
\[
\langle \; \raisebox{-0.7\baselineskip}{\def\svgwidth{11pt}
\begingroup%
  \makeatletter%
  \providecommand\color[2][]{%
    \errmessage{(Inkscape) Color is used for the text in Inkscape, but the package 'color.sty' is not loaded}%
    \renewcommand\color[2][]{}%
  }%
  \providecommand\transparent[1]{%
    \errmessage{(Inkscape) Transparency is used (non-zero) for the text in Inkscape, but the package 'transparent.sty' is not loaded}%
    \renewcommand\transparent[1]{}%
  }%
  \providecommand\rotatebox[2]{#2}%
  \newcommand*\fsize{\dimexpr\f@size pt\relax}%
  \newcommand*\lineheight[1]{\fontsize{\fsize}{#1\fsize}\selectfont}%
  \ifx\svgwidth\undefined%
    \setlength{\unitlength}{18.38020952bp}%
    \ifx\svgscale\undefined%
      \relax%
    \else%
      \setlength{\unitlength}{\unitlength * \real{\svgscale}}%
    \fi%
  \else%
    \setlength{\unitlength}{\svgwidth}%
  \fi%
  \global\let\svgwidth\undefined%
  \global\let\svgscale\undefined%
  \makeatother%
  \begin{picture}(1,1.93944322)%
    \lineheight{1}%
    \setlength\tabcolsep{0pt}%
    \put(0,0){\includegraphics[width=\unitlength,page=1]{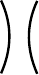}}%
  \end{picture}%
\endgroup%
} \; \rangle_\infty = \; - \; \langle \; \raisebox{-0.7\baselineskip}{\def\svgwidth{40pt}
\begingroup%
  \makeatletter%
  \providecommand\color[2][]{%
    \errmessage{(Inkscape) Color is used for the text in Inkscape, but the package 'color.sty' is not loaded}%
    \renewcommand\color[2][]{}%
  }%
  \providecommand\transparent[1]{%
    \errmessage{(Inkscape) Transparency is used (non-zero) for the text in Inkscape, but the package 'transparent.sty' is not loaded}%
    \renewcommand\transparent[1]{}%
  }%
  \providecommand\rotatebox[2]{#2}%
  \newcommand*\fsize{\dimexpr\f@size pt\relax}%
  \newcommand*\lineheight[1]{\fontsize{\fsize}{#1\fsize}\selectfont}%
  \ifx\svgwidth\undefined%
    \setlength{\unitlength}{66.01123985bp}%
    \ifx\svgscale\undefined%
      \relax%
    \else%
      \setlength{\unitlength}{\unitlength * \real{\svgscale}}%
    \fi%
  \else%
    \setlength{\unitlength}{\svgwidth}%
  \fi%
  \global\let\svgwidth\undefined%
  \global\let\svgscale\undefined%
  \makeatother%
  \begin{picture}(1,0.5368989)%
    \lineheight{1}%
    \setlength\tabcolsep{0pt}%
    \put(0,0){\includegraphics[width=\unitlength,page=1]{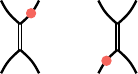}}%
    \put(0.29183112,0.40649264){\color[rgb]{0.96470588,0.4,0.36470588}\makebox(0,0)[lt]{\lineheight{1.25}\smash{\begin{tabular}[t]{l}$1$\end{tabular}}}}%
    \put(0.60716348,0.07308317){\color[rgb]{0.96470588,0.4,0.36470588}\makebox(0,0)[lt]{\lineheight{1.25}\smash{\begin{tabular}[t]{l}$1$\end{tabular}}}}%
    \put(0.45877338,0.24022888){\color[rgb]{0,0,0}\makebox(0,0)[lt]{\lineheight{1.25}\smash{\begin{tabular}[t]{l}$-$\end{tabular}}}}%
  \end{picture}%
\endgroup%
} \; \rangle_\infty.
\]
Let $c$ be a coloring of $\raisebox{-0.7\baselineskip}{\def\svgwidth{11pt}}$ (see \cref{eq:colorings}) and call $i$ and $j$ the pigments on the left and right strand respectively (by definition of $c$, we have $i \ne j$). Let $\langle \; \raisebox{-0.7\baselineskip}{\def\svgwidth{11pt}} \; , c \rangle_\infty = \frac{P}{Q}$. The coloring $c$ induces the coloring $c'$ on $\raisebox{-0.7\baselineskip}{\def\svgwidth{11pt}}$, and possibly (depending on the shape of the rest of the graph) colorings $c''$ and $c'''$.

\begin{equation} \label{eq:colorings}
{\scriptsize \raisebox{-2.3\baselineskip}{\def\svgwidth{170pt}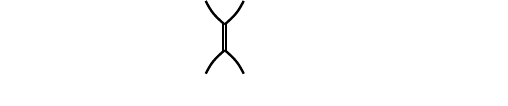}}
\end{equation}

The following computations prove the statement (the polynomials $P'$ and $Q'$ encode the $\infty$-evaluation at $c''$ on the portion of the graph that is outside of the tangle $\raisebox{-0.7\baselineskip}{\def\svgwidth{11pt}}$):
\begin{gather*}
\langle \; \raisebox{-0.7\baselineskip}{\def\svgwidth{40pt}} \; , c' \rangle_\infty = \frac{P\cdot x_j}{Q \cdot (x_i-x_j)} - \frac{P\cdot x_i}{Q \cdot (x_i-x_j)} = - \frac{P}{Q} \\
\langle \; \raisebox{-0.7\baselineskip}{\def\svgwidth{40pt}} \; , c'' \rangle_\infty = \frac{P'\cdot x_i}{Q' \cdot (x_j-x_i)} - \frac{P'\cdot x_i}{Q' \cdot (x_j-x_i)} = 0, \\
\langle \; \raisebox{-0.7\baselineskip}{\def\svgwidth{40pt}} \; , c''' \rangle_\infty = 0.
\end{gather*}

\end{proof}

Consider the positive braid $\beta^+$ obtained from $\beta$ by turning all crossings to positive ones. By construction, the differentials of $C_*(\widehat{\beta}^+)$ are now only \emph{zip} maps. The program first finds the pairing matrices for $\beta^+$. This is done by computing transposes of the pairing matrices of \emph{unzip} maps, using \cref{lem:positivecx} (functions {\tt compute\_all\_diff\_unsmoothing} and {\tt all\_diff}). In order to recover the complex and pairing matrices for $\beta$ one then simply needs to reorder the vertices of the hypercube of resolutions (this is encoded in {\tt res\_by\_hom\_deg} in {\tt homology}) and take transposes of some of the pairing matrices ({\tt piece\_of\_matrix} in the function {\tt piece\_of\_real\_diff}), according to the signs of the crossings of $\beta$. Finally, using \cref{lem:diffpairing}, the program computes the differentials: this is done in {\tt inverse\_basis\_matrix}, {\tt all\_inverse\_basis\_matrices} and {\tt piece\_of\_real\_diff}.

\begin{remark}
Other than computing the differentials, \cref{lem:diffpairing} has a wider function. It allows one to express any element of $\mathscr{S}_1(\widehat{\beta}_v)$ as a linear combination of elements of the d.u.r.\ basis. This is the purpose of the function {\tt el\_as\_basis\_comb}.
\end{remark}

\subsection{Taking homology}
\label{sec:computing-homology}

Over a field, computing the homology of a finite
dimensional chain complex boils down to computing the ranks of the
differentials. In our context, since we consider a complex of ($q$)-graded
vector spaces, the computation can be split according to $q$-degree.

Over $\Q$, the rank is computed\footnote{Over the rationals, in principle it would be enough to compute the rank of the pairing matrices (by \cref{lem:diffpairing}). This however is no longer true in positive characteristic.} using the package {\tt torch}, which is more efficient than {\tt numpy} when dealing with very large matrices. This is done in the functions {\tt compute\_rank} and {\tt homology}.

Since the d.u.r.\ set is a basis of state spaces over any $R$, this allows in principle to compute homology over $\Z$ and over any field. As we shall see in \cref{ex:notinvariant}, we can disprove
invariance of such theories using our program. Over the integers, the rank of the differential provides enough information to compute the torsion free part of the homology, but does not see the torsion part. Our program, however, can compute the homology over finite fields. If one wishes to compute in characteristic $3$, for instance, one has to insert {\tt -p 3} between the name of the program and the braid. In this setting, the rank is computed using the software SpaSM developed by Bouillaguet \cite{spasm}.

\subsection{Some computations}

The example below shows that the symmetric $\mathfrak{gl}_n$-homology is not a link invariant over the integers.

\begin{example} \label{ex:notinvariant}
Consider the three braids $\sigma_1$, $\sigma_1\sigma_2^{-1}$ and $\sigma_1\sigma_2$. Their closures all represent the unknot, and they are obtained from the trivial diagram by performing respectively one R1$_+$, one R1$_+$ and one R1$_-$, and two R1$_+$ (here R1$_+$ (resp.\ R1$_-$) stands for a positive (resp.\ negative) Reidemeister I move). Over $\mathbb{F}_3$, our program yields three distinct results for the three braids above, showing that the symmetric $\mathfrak{gl}_1$-homology is not a link invariant over fields of positive characteristic. It follows that it is also not an invariant of links over $\Z$.
\end{example}

Using \texttt{foam-gl1.py}, we have computed the non-equivariant symmetric $\mathfrak{gl}_1$-homology of almost all prime knots with up to 10 crossings. The outputs of the computations are listed in the appendix. These results coincide with the reduced triply graded homology computations made by Nakagane--Sano \cite{computedHOMFLY}. Based on this, we formulate the following conjecture.

\begin{conjecture}
The total rank of the symmetric $\mathfrak{gl}_1$-homology is equal to the total rank of the reduced triply graded homology.   
\end{conjecture}

\subsection{Room for improvement}
\label{sec:improvement}

The program is based on a rather naive approach and is therefore far
from being fast and optimal. On an ordinary computer, it is efficient for braid length up to 12 and braid index $\leq 6$. Among all kinds of possible improvements
let us list a few things:

\begin{itemize}
\item \emph{Gaussian elimination:} Computing the rank of large
  matrices takes a lot of RAM resources. It would be
  far more efficient to proceed by Gaussian elimination as explained
  by Bar-Natan \cite{BN2} and implemented, for instance, by Lewark \cite{khoca}.
\item Many of the braids we consider feature powers of generators. A factor
  $\sigma_i^j$ contributes as a $|j|$-dimensional hypercube. However,
  we know \cite{Krasner} that, in this case, the complex can be simplified.
\item The program is entirely written in Python which is by essence
  not very fast. However its critical parts (such as computation of ranks) are
  implemented using the libraries NumPy and Torch, which are entirely coded in
  \texttt{C} and therefore do not suffer too much from the slowness
  of Python. The part which computes the pairing matrices relies entirely on
  Python and could certainly be substantially faster if rewritten in \texttt{C}. 
\end{itemize}

\appendix

\section{Computations for small knots} \label{sec:computations}

We list the Poincar\'e polynomial $P_{\mathfrak{gl}_1}$ of (non-equivariant) symmetric $\mathfrak{gl}_1$-homology over the rationals for all prime knots with up to 10 crossings, braid index at most 4 and braid length up to 12. These polynomials have been computed using the program \texttt{foam-gl1.py} \cite{foam-gl1}. On these knots, the correspondence between $P_{\mathfrak{gl}_1}=P_{\mathfrak{gl}_1}(t,q)$ and the Poincar\'e polynomial $P$ corresponding to triply graded homology, with the conventions of Nakagane--Sano \cite{computedHOMFLY}, is given by the following equality:
\begin{equation}
P_{\mathfrak{gl}_1}(t,q^{-1}) =  P(qt,qt,t).
\end{equation}

\begin{center}
\begin{adjustwidth}{-2cm}{}
\begin{tabular}{|c|p{15cm}|} 
 \hline
 Knot & Poincar\'e polynomial \\
 \hline\hline
 $3_1$ & $1 + t^2q^{-4} + tq^{-4}$ \\ 
 \hline
 $4_1$ & $q^2 + q^2t^{-1} + 1 + tq^{-2} + q^{-2}$ \\
 \hline
 $5_1$ & $1 + t^2q^{-4} + tq^{-4} + t^4q^{-8} + t^3q^{-8}$ \\
 \hline
 $5_2$ & $1 + tq^{-2} + q^{-2} + t^2q^{-4} + tq^{-4} + t^2q^{-6} + tq^{-6}$ \\
 \hline
 $6_1$ & $q^2 + q^2t^{-1} + 2 + t^{-1} + tq^{-2} + q^{-2} + tq^{-4} + q^{-4}$ \\ 
 \hline
  $6_2$ & $q^2 + q^2t^{-1} + 1 + t^2q^{-2} + 2tq^{-2} + q^{-2} + t^2q^{-4} + tq^{-4} + t^3q^{-6} + t^2q^{-6}
$ \\ 
 \hline
  $6_3$ & $q^4t^{-1} + q^4t^{-2} + q^2 + q^2t^{-1} + t + 3 + t^{-1} + tq^{-2} + q^{-2} + t^2q^{-4} + tq^{-4}
$ \\
\hline
 
\end{tabular}
\end{adjustwidth}
\end{center}

\begin{center}
\begin{adjustwidth}{-2cm}{}
\begin{tabular}{|c|p{15cm}|} 
 \hline
 Knot & Poincar\'e polynomial \\
 \hline\hline
  $8_{20}$ & $q^6t^{-1} + q^6t^{-2} + q^2t + 2q^2 + q^2t^{-1} + 1 + t^2q^{-2} + tq^{-2}
$ \\ 
 \hline
  $8_{21}$ & $2 + t^{-1} + tq^{-2} + q^{-2} + 2t^2q^{-4} + 3tq^{-4} + q^{-4} + t^2q^{-6} + tq^{-6} + t^3q^{-8} + t^2q^{-8}
$ \\ 
 \hline
  $9_1$ & $1 + t^2q^{-4} + tq^{-4} + t^4q^{-8} + t^3q^{-8} + t^6q^{-12} + t^5q^{-12} + t^8q^{-16} + t^7q^{-16}
$ \\ 
 \hline
  $9_2$ & $1 + tq^{-2} + q^{-2} + t^2q^{-4} + 2tq^{-4} + q^{-4} + t^2q^{-6} + 2tq^{-6} + q^{-6} + t^2q^{-8} + tq^{-8} + t^2q^{-10} + tq^{-10}
$ \\ 
 \hline
  $9_3$ & $1 + tq^{-2} + q^{-2} + t^2q^{-4} + tq^{-4} + t^3q^{-6} + 2t^2q^{-6} + tq^{-6} + t^4q^{-8} + t^3q^{-8} + t^5q^{-10} + 2t^4q^{-10} + t^3q^{-10} + t^6q^{-12} + t^5q^{-12} + t^6q^{-14} + t^5q^{-14}
$ \\ 
 \hline
  $9_4$ & $
1 + tq^{-2} + q^{-2} + t^2q^{-4} + 2tq^{-4} + q^{-4} + t^3q^{-6} + 2t^2q^{-6} + tq^{-6} + t^4q^{-8} + 2t^3q^{-8} + 2t^2q^{-8} + tq^{-8} + t^4q^{-10} + t^3q^{-10} + t^4q^{-12} + t^3q^{-12}
$ \\ 
 \hline
  $9_5$ & $1 + 2tq^{-2} + 2q^{-2} + t^2q^{-4} + 3tq^{-4} + 2q^{-4} + 2t^2q^{-6} + 3tq^{-6} + q^{-6} + 2t^2q^{-8} + 2tq^{-8} + t^2q^{-10} + tq^{-10}
$ \\
\hline

\end{tabular}
\end{adjustwidth}
\end{center}

\begin{center}
\begin{adjustwidth}{-2cm}{}
\begin{tabular}{|c|p{15cm}|} 
 \hline
 Knot & Poincar\'e polynomial \\
 \hline\hline
  $9_{27}$ & $q^6t^{-2} + q^6t^{-3} + 2q^4t^{-1} + 2q^4t^{-2} + 3q^2 + 5q^2t^{-1} + 2q^2t^{-2} + 2t + 6 + 3t^{-1} + t^2q^{-2} + 5tq^{-2} + 5q^{-2} + q^{-2}t^{-1} + 2t^2q^{-4} + 3tq^{-4} + q^{-4} + t^3q^{-6} + 2t^2q^{-6} + tq^{-6}
$ \\ 
 \hline
  $9_{28}$ & $q^4t^{-1} + q^4t^{-2} + 2q^2 + 2q^2t^{-1} + 3t + 6 + 2t^{-1} + 2t^2q^{-2} + 5tq^{-2} + 3q^{-2} + t^3q^{-4} + 6t^2q^{-4} + 6tq^{-4} + q^{-4} + 2t^3q^{-6} + 3t^2q^{-6} + tq^{-6} + t^4q^{-8} + 2t^3q^{-8} + t^2q^{-8}
$ \\ 
 \hline
  $9_{29}$ & $q^8t^{-3} + q^8t^{-4} + 2q^6t^{-2} + 2q^6t^{-3} + 4q^4t^{-1} + 6q^4t^{-2} + 2q^4t^{-3} + 2q^2 + 5q^2t^{-1} + 3q^2t^{-2} + t + 6 + 5t^{-1} + t^{-2} + 2tq^{-2} + 3q^{-2} + q^{-2}t^{-1} + t^2q^{-4} + 2tq^{-4} + q^{-4}
$ \\ 
 \hline
  $9_{30}$ & $q^6t^{-1} + 2q^6t^{-2} + q^6t^{-3} + q^4 + 3q^4t^{-1} + 2q^4t^{-2} + q^2t + 6q^2 + 6q^2t^{-1} + q^2t^{-2} + 3t + 6 + 2t^{-1} + 2t^2q^{-2} + 6tq^{-2} + 4q^{-2} + 2t^2q^{-4} + 2tq^{-4} + t^3q^{-6} + t^2q^{-6}
$ \\ 
 \hline
  $9_{31}$ & $q^4t^{-1} + q^4t^{-2} + 2q^2 + 2q^2t^{-1} + 3t + 6 + 2t^{-1} + 2t^2q^{-2} + 6tq^{-2} + 4q^{-2} + t^3q^{-4} + 6t^2q^{-4} + 6tq^{-4} + q^{-4} + 2t^3q^{-6} + 4t^2q^{-6} + 2tq^{-6} + t^4q^{-8} + 2t^3q^{-8} + t^2q^{-8}
$ \\ 
 \hline
  $9_{32}$ & $q^8t^{-2} + 2q^8t^{-3} + q^8t^{-4} + q^6t^{-1} + 4q^6t^{-2} + 3q^6t^{-3} + q^4 + 6q^4t^{-1} + 6q^4t^{-2} + q^4t^{-3} + 4q^2 + 7q^2t^{-1} + 3q^2t^{-2} + 2t + 6 + 3t^{-1} + 3tq^{-2} + 3q^{-2} + t^2q^{-4} + tq^{-4}
$ \\ 
 \hline
  $9_{33}$ & $q^6t^{-1} + 2q^6t^{-2} + q^6t^{-3} + q^4 + 4q^4t^{-1} + 3q^4t^{-2} + q^2t + 6q^2 + 6q^2t^{-1} + q^2t^{-2} + 4t + 8 + 3t^{-1} + 2t^2q^{-2} + 6tq^{-2} + 4q^{-2} + 3t^2q^{-4} + 3tq^{-4} + t^3q^{-6} + t^2q^{-6}
$ \\
\hline

\end{tabular}
\end{adjustwidth}
\end{center}

\begin{center}
\begin{adjustwidth}{-2cm}{}
\begin{tabular}{|c|p{15cm}|} 
 \hline
 Knot & Poincar\'e polynomial \\
 \hline\hline
  $10_7$ & $q^2 + q^2t^{-1} + t + 3 + t^{-1} + t^2q^{-2} + 4tq^{-2} + 4q^{-2} + q^{-2}t^{-1} + 3t^2q^{-4} + 5tq^{-4} + 2q^{-4} + t^3q^{-6} + 4t^2q^{-6} + 4tq^{-6} + q^{-6} + t^3q^{-8} + 2t^2q^{-8} + tq^{-8} + t^3q^{-10} + t^2q^{-10}
$ \\ 
 \hline
  $10_8$ & $q^4 + q^4t^{-1} + q^2 + q^2t^{-1} + t^2 + 2t + 2 + t^2q^{-2} + 2tq^{-2} + q^{-2} + t^4q^{-4} + 2t^3q^{-4} + 2t^2q^{-4} + tq^{-4} + t^4q^{-6} + 2t^3q^{-6} + t^2q^{-6} + t^5q^{-8} + 2t^4q^{-8} + t^3q^{-8} + t^5q^{-10} + t^4q^{-10}
$ \\ 
 \hline
  $10_9$ & $q^6t^{-2} + q^6t^{-3} + q^4t^{-1} + q^4t^{-2} + 2q^2 + 3q^2t^{-1} + q^2t^{-2} + t + 3 + t^{-1} + 2t^2q^{-2} + 4tq^{-2} + 2q^{-2} + t^3q^{-4} + 3t^2q^{-4} + 2tq^{-4} + t^4q^{-6} + 3t^3q^{-6} + 2t^2q^{-6} + t^4q^{-8} + t^3q^{-8} + t^5q^{-10} + t^4q^{-10}
$ \\ 
 \hline
  $10_{10}$ & $q^8t^{-1} + q^8t^{-2} + q^6 + 2q^6t^{-1} + q^6t^{-2} + q^4t + 4q^4 + 4q^4t^{-1} + q^4t^{-2} + 2q^2t + 5q^2 + 3q^2t^{-1} + t^2 + 4t + 5 + t^{-1} + t^2q^{-2} + 3tq^{-2} + 2q^{-2} + t^2q^{-4} + tq^{-4}
$ \\ 
 \hline
  $10_{11}$ & $q^4 + q^4t^{-1} + 2q^2 + 2q^2t^{-1} + t^2 + 3t + 4 + t^{-1} + 2t^2q^{-2} + 5tq^{-2} + 3q^{-2} + t^3q^{-4} + 4t^2q^{-4} + 4tq^{-4} + q^{-4} + 2t^3q^{-6} + 3t^2q^{-6} + tq^{-6} + t^3q^{-8} + t^2q^{-8}
$ \\ 
 \hline
  $10_{12}$ & $q^{10}t^{-3} + q^{10}t^{-4} + q^8t^{-2} + 2q^8t^{-3} + q^8t^{-4} + 2q^6t^{-1} + 4q^6t^{-2} + 2q^6t^{-3} + q^4 + 4q^4t^{-1} + 4q^4t^{-2} + q^4t^{-3} + q^2t + 4q^2 + 4q^2t^{-1} + q^2t^{-2} + 2t + 4 + t^{-1} + t^2q^{-2} + 2tq^{-2} + q^{-2} + t^2q^{-4} + tq^{-4}
$ \\ 
 \hline
  $10_{13}$ & - \\
\hline

\end{tabular}
\end{adjustwidth}
\end{center}

\begin{center}
\begin{adjustwidth}{-2cm}{}
\begin{tabular}{|c|p{15cm}|}  
 \hline
 Knot & Poincar\'e polynomial \\
 \hline\hline
  $10_{31}$ & $q^6t^{-1} + q^6t^{-2} + q^4 + 3q^4t^{-1} + 2q^4t^{-2} + q^2t + 5q^2 + 5q^2t^{-1} + q^2t^{-2} + 3t + 8 + 4t^{-1} + t^2q^{-2} + 5tq^{-2} + 5q^{-2} + q^{-2}t^{-1} + 2t^2q^{-4} + 4tq^{-4} + 2q^{-4} + t^2q^{-6} + tq^{-6}$ \\ 
 \hline
  $10_{32}$ & $q^6t^{-2} + q^6t^{-3} + 2q^4t^{-1} + 3q^4t^{-2} + q^4t^{-3} + 3q^2 + 6q^2t^{-1} + 3q^2t^{-2} + 2t + 7 + 5t^{-1} + t^{-2} + t^2q^{-2} + 6tq^{-2} + 7q^{-2} + 2q^{-2}t^{-1} + 3t^2q^{-4} + 5tq^{-4} + 2q^{-4} + t^3q^{-6} + 3t^2q^{-6} + 2tq^{-6} + t^3q^{-8} + t^2q^{-8}$ \\ 
 \hline
  $10_{33}$ & $q^6t^{-1} + q^6t^{-2} + 2q^4 + 4q^4t^{-1} + 2q^4t^{-2} + q^2t + 6q^2 + 6q^2t^{-1} + q^2t^{-2} + 4t + 10 + 5t^{-1} + t^2q^{-2} + 6tq^{-2} + 6q^{-2} + q^{-2}t^{-1} + 2t^2q^{-4} + 4tq^{-4} + 2q^{-4} + t^2q^{-6} + tq^{-6}$ \\ 
 \hline
  $10_{34}$ & $q^8t^{-1} + q^8t^{-2} + q^6 + 2q^6t^{-1} + q^6t^{-2} + q^4t + 3q^4 + 3q^4t^{-1} + q^4t^{-2} + 2q^2t + 4q^2 + 3q^2t^{-1} + q^2t^{-2} + t^2 + 3t + 4 + t^{-1} + t^2q^{-2} + 2tq^{-2} + q^{-2} + t^2q^{-4} + tq^{-4}
$ \\ 
 \hline
  $10_{35}$ & - \\ 
 \hline
  $10_{36}$ & $q^2 + q^2t^{-1} + t + 3 + t^{-1} + t^2q^{-2} + 5tq^{-2} + 5q^{-2} + q^{-2}t^{-1} + 3t^2q^{-4} + 7tq^{-4} + 4q^{-4} + t^3q^{-6} + 5t^2q^{-6} + 5tq^{-6} + q^{-6} + t^3q^{-8} + 3t^2q^{-8} + 2tq^{-8} + t^3q^{-10} + t^2q^{-10}
$ \\ 
 \hline
  $10_{37}$ & $q^6t^{-1} + q^6t^{-2} + q^4 + 3q^4t^{-1} + 2q^4t^{-2} + q^2t + 5q^2 + 5q^2t^{-1} + q^2t^{-2} + 3t + 7 + 3t^{-1} + t^2q^{-2} + 5tq^{-2} + 5q^{-2} + q^{-2}t^{-1} + 2t^2q^{-4} + 3tq^{-4} + q^{-4} + t^2q^{-6} + tq^{-6}$ \\
\hline

\end{tabular}
\end{adjustwidth}
\end{center}

\begin{center}
\begin{adjustwidth}{-2cm}{}
\begin{tabular}{|c|p{15cm}|} 
 \hline
 Knot & Poincar\'e polynomial \\
 \hline\hline
  $10_{55}$ & $1 + 2tq^{-2} + 2q^{-2} + 3t^2q^{-4} + 5tq^{-4} + 2q^{-4} + 2t^3q^{-6} + 7t^2q^{-6} + 5tq^{-6} + t^4q^{-8} + 5t^3q^{-8} + 7t^2q^{-8} + 3tq^{-8} + 2t^4q^{-10} + 5t^3q^{-10} + 3t^2q^{-10} + 2t^4q^{-12} + 3t^3q^{-12} + t^2q^{-12}
$ \\ 
 \hline
  $10_{56}$ & $q^2 + q^2t^{-1} + t + 3 + t^{-1} + 2t^2q^{-2} + 5tq^{-2} + 3q^{-2} + t^3q^{-4} + 6t^2q^{-4} + 6tq^{-4} + q^{-4} + t^4q^{-6} + 5t^3q^{-6} + 6t^2q^{-6} + 2tq^{-6} + 3t^4q^{-8} + 6t^3q^{-8} + 3t^2q^{-8} + t^5q^{-10} + 3t^4q^{-10} + 2t^3q^{-10} + t^5q^{-12} + t^4q^{-12}
$ \\ 
 \hline
  $10_{57}$ & $q^{10}t^{-3} + q^{10}t^{-4} + 2q^8t^{-2} + 3q^8t^{-3} + q^8t^{-4} + 4q^6t^{-1} + 7q^6t^{-2} + 3q^6t^{-3} + 2q^4 + 8q^4t^{-1} + 7q^4t^{-2} + q^4t^{-3} + q^2t + 7q^2 + 8q^2t^{-1} + 2q^2t^{-2} + 3t + 7 + 3t^{-1} + t^2q^{-2} + 3tq^{-2} + 2q^{-2} + t^2q^{-4} + tq^{-4}
$ \\ 
 \hline
  $10_{58}$ & - \\ 
 \hline
  $10_{59}$ & $q^4 + 2q^4t^{-1} + q^4t^{-2} + q^2t + 4q^2 + 3q^2t^{-1} + t^2 + 6t + 8 + 2t^{-1} + 4t^2q^{-2} + 8tq^{-2} + 4q^{-2} + 2t^3q^{-4} + 8t^2q^{-4} + 7tq^{-4} + q^{-4} + 3t^3q^{-6} + 4t^2q^{-6} + tq^{-6} + t^4q^{-8} + 2t^3q^{-8} + t^2q^{-8}
$ \\ 
 \hline
  $10_{60}$ & $q^6 + 3q^6t^{-1} + 3q^6t^{-2} + q^6t^{-3} + 3q^4 + 7q^4t^{-1} + 4q^4t^{-2} + 3q^2t + 15q^2 + 14q^2t^{-1} + 2q^2t^{-2} + 7t + 12 + 4t^{-1} + 3t^2q^{-2} + 8tq^{-2} + 5q^{-2} + 3t^2q^{-4} + 3tq^{-4} + t^3q^{-6} + t^2q^{-6}
$ \\ 
 \hline
  $10_{61}$ & $q^4 + q^4t^{-1} + q^2 + q^2t^{-1} + 2t^2 + 3t + 2 + t^2q^{-2} + 2tq^{-2} + q^{-2} + t^4q^{-4} + 3t^3q^{-4} + 3t^2q^{-4} + tq^{-4} + t^4q^{-6} + 2t^3q^{-6} + t^2q^{-6} + t^5q^{-8} + 2t^4q^{-8} + t^3q^{-8} + t^5q^{-10} + t^4q^{-10}
$ \\
\hline
 
\end{tabular}
\end{adjustwidth}
\end{center}

\begin{center}
\begin{adjustwidth}{-2cm}{}
\begin{tabular}{|c|p{15cm}|}  
 \hline
 Knot & Poincar\'e polynomial \\
 \hline\hline
  $10_{80}$ & $1 + 2tq^{-2} + 2q^{-2} + 4t^2q^{-4} + 4tq^{-4} + 4t^3q^{-6} + 7t^2q^{-6} + 3tq^{-6} + 4t^4q^{-8} + 7t^3q^{-8} + 3t^2q^{-8} + 2t^5q^{-10} + 7t^4q^{-10} + 6t^3q^{-10} + t^2q^{-10} + t^6q^{-12} + 4t^5q^{-12} + 3t^4q^{-12} + 2t^6q^{-14} + 3t^5q^{-14} + t^4q^{-14}
$ \\ 
 \hline
  $10_{81}$ & $q^6t^{-1} + 2q^6t^{-2} + q^6t^{-3} + q^4 + 5q^4t^{-1} + 4q^4t^{-2} + q^2t + 7q^2 + 8q^2t^{-1} + 2q^2t^{-2} + 5t + 11 + 5t^{-1} + 2t^2q^{-2} + 8tq^{-2} + 7q^{-2} + q^{-2}t^{-1} + 4t^2q^{-4} + 5tq^{-4} + q^{-4} + t^3q^{-6} + 2t^2q^{-6} + tq^{-6}
$ \\ 
 \hline
  $10_{82}$ & $q^6t^{-2} + q^6t^{-3} + 2q^4t^{-1} + 2q^4t^{-2} + 3q^2 + 4q^2t^{-1} + q^2t^{-2} + 3t + 6 + 2t^{-1} + 3t^2q^{-2} + 6tq^{-2} + 3q^{-2} + 2t^3q^{-4} + 6t^2q^{-4} + 4tq^{-4} + t^4q^{-6} + 4t^3q^{-6} + 3t^2q^{-6} + 2t^4q^{-8} + 2t^3q^{-8} + t^5q^{-10} + t^4q^{-10}
$ \\ 
 \hline
  $10_{83}$ & $q^{10}t^{-3} + q^{10}t^{-4} + 2q^8t^{-2} + 3q^8t^{-3} + q^8t^{-4} + 3q^6t^{-1} + 7q^6t^{-2} + 4q^6t^{-3} + 2q^4 + 8q^4t^{-1} + 7q^4t^{-2} + q^4t^{-3} + q^2t + 7q^2 + 9q^2t^{-1} + 3q^2t^{-2} + 3t + 7 + 3t^{-1} + t^2q^{-2} + 4tq^{-2} + 3q^{-2} + t^2q^{-4} + tq^{-4}
$ \\ 
 \hline
  $10_{84}$ & $q^4t^{-1} + q^4t^{-2} + 2q^2 + 3q^2t^{-1} + q^2t^{-2} + 3t + 8 + 4t^{-1} + 2t^2q^{-2} + 8tq^{-2} + 7q^{-2} + q^{-2}t^{-1} + t^3q^{-4} + 8t^2q^{-4} + 10tq^{-4} + 3q^{-4} + 3t^3q^{-6} + 7t^2q^{-6} + 4tq^{-6} + t^4q^{-8} + 4t^3q^{-8} + 3t^2q^{-8} + t^4q^{-10} + t^3q^{-10}
$ \\ 
 \hline
  $10_{85}$ & $q^{12}t^{-5} + q^{12}t^{-6} + 2q^{10}t^{-4} + 2q^{10}t^{-5} + 3q^8t^{-3} + 4q^8t^{-4} + q^8t^{-5} + 3q^6t^{-2} + 5q^6t^{-3} + 2q^6t^{-4} + 3q^4t^{-1} + 5q^4t^{-2} + 2q^4t^{-3} + 2q^2 + 5q^2t^{-1} + 3q^2t^{-2} + t + 4 + 2t^{-1} + 2tq^{-2} + 2q^{-2} + t^2q^{-4} + tq^{-4}
$ \\
\hline
 
\end{tabular}
\end{adjustwidth}
\end{center}

\begin{center}
\begin{adjustwidth}{-2cm}{}
\begin{tabular}{|c|p{15cm}|} 
 \hline
 Knot & Poincar\'e polynomial \\
 \hline\hline
  $10_{104}$ & $q^8t^{-3} + q^8t^{-4} + 2q^6t^{-2} + 2q^6t^{-3} + 4q^4t^{-1} + 5q^4t^{-2} + q^4t^{-3} + 5q^2 + 7q^2t^{-1} + 2q^2t^{-2} + 4t + 9 + 4t^{-1} + 2t^2q^{-2} + 7tq^{-2} + 5q^{-2} + t^3q^{-4} + 5t^2q^{-4} + 4tq^{-4} + 2t^3q^{-6} + 2t^2q^{-6} + t^4q^{-8} + t^3q^{-8}
$ \\ 
 \hline
  $10_{105}$ & $q^4 + 2q^4t^{-1} + q^4t^{-2} + q^2t + 5q^2 + 4q^2t^{-1} + t^2 + 7t + 9 + 2t^{-1} + 5t^2q^{-2} + 12tq^{-2} + 7q^{-2} + 2t^3q^{-4} + 9t^2q^{-4} + 8tq^{-4} + q^{-4} + 4t^3q^{-6} + 6t^2q^{-6} + 2tq^{-6} + t^4q^{-8} + 2t^3q^{-8} + t^2q^{-8}				
$ \\ 
 \hline
  $10_{106}$ & $q^6t^{-2} + q^6t^{-3} + 2q^4t^{-1} + 2q^4t^{-2} + 4q^2 + 5q^2t^{-1} + q^2t^{-2} + 4t + 7 + 2t^{-1} + 4t^2q^{-2} + 8tq^{-2} + 4q^{-2} + 2t^3q^{-4} + 7t^2q^{-4} + 5tq^{-4} + t^4q^{-6} + 5t^3q^{-6} + 4t^2q^{-6} + 2t^4q^{-8} + 2t^3q^{-8} + t^5q^{-10} + t^4q^{-10}
$ \\ 
 \hline
  $10_{107}$ & $q^6t^{-1} + 2q^6t^{-2} + q^6t^{-3} + q^4 + 5q^4t^{-1} + 4q^4t^{-2} + q^2t + 8q^2 + 9q^2t^{-1} + 2q^2t^{-2} + 5t + 13 + 7t^{-1} + 2t^2q^{-2} + 9tq^{-2} + 8q^{-2} + q^{-2}t^{-1} + 4t^2q^{-4} + 6tq^{-4} + 2q^{-4} + t^3q^{-6} + 2t^2q^{-6} + tq^{-6}
$ \\ 
 \hline
  $10_{108}$ & $q^6t^{-1} + q^6t^{-2} + 2q^4 + 3q^4t^{-1} + q^4t^{-2} + 2q^2t + 5q^2 + 3q^2t^{-1} + 2t^2 + 5t + 5 + t^{-1} + t^3q^{-2} + 5t^2q^{-2} + 6tq^{-2} + 2q^{-2} + 3t^3q^{-4} + 5t^2q^{-4} + 2tq^{-4} + t^4q^{-6} + 3t^3q^{-6} + 2t^2q^{-6} + t^4q^{-8} + t^3q^{-8}
$ \\ 
 \hline
  $10_{109}$ & $q^8t^{-3} + q^8t^{-4} + 2q^6t^{-2} + 2q^6t^{-3} + 5q^4t^{-1} + 6q^4t^{-2} + q^4t^{-3} + 5q^2 + 7q^2t^{-1} + 2q^2t^{-2} + 5t + 11 + 5t^{-1} + 2t^2q^{-2} + 7tq^{-2} + 5q^{-2} + t^3q^{-4} + 6t^2q^{-4} + 5tq^{-4} + 2t^3q^{-6} + 2t^2q^{-6} + t^4q^{-8} + t^3q^{-8}
$ \\
\hline
 
\end{tabular}
\end{adjustwidth}
\end{center}

\begin{center}
\begin{adjustwidth}{-2cm}{}
\begin{tabular}{|c|p{15cm}|} 
 \hline
 Knot & Poincar\'e polynomial \\
 \hline\hline
  $10_{128}$ & $1 + tq^{-2} + q^{-2} + t^2q^{-4} + tq^{-4} + t^3q^{-6} + 3t^2q^{-6} + 2tq^{-6} + t^4q^{-8} + t^3q^{-8} + t^2q^{-8} + tq^{-8} + t^5q^{-10} + 2t^4q^{-10} + 2t^3q^{-10} + t^2q^{-10} + t^6q^{-12} + t^5q^{-12} + t^3q^{-12} + t^2q^{-12} + t^6q^{-14} + t^5q^{-14}
$ \\ 
 \hline
  $10_{129}$ & $q^4t^{-1} + q^4t^{-2} + q^2 + 2q^2t^{-1} + q^2t^{-2} + t + 4 + 2t^{-1} + 2tq^{-2} + 3q^{-2} + q^{-2}t^{-1} + t^2q^{-4} + 2tq^{-4} + q^{-4} + t^2q^{-6} + tq^{-6}
$ \\ 
 \hline
  $10_{130}$ & $q^8t^{-1} + q^8t^{-2} + q^6t^{-1} + q^6t^{-2} + q^4t + 2q^4 + q^4t^{-1} + q^2t + 2q^2 + q^2t^{-1} + t^2 + t + 1 + t^2q^{-2} + tq^{-2}
$ \\ 
 \hline
  $10_{131}$ & $2 + t^{-1} + 2tq^{-2} + 3q^{-2} + q^{-2}t^{-1} + 2t^2q^{-4} + 4tq^{-4} + 2q^{-4} + 3t^2q^{-6} + 4tq^{-6} + q^{-6} + t^3q^{-8} + 2t^2q^{-8} + tq^{-8} + t^3q^{-10} + t^2q^{-10}
$ \\ 
 \hline
  $10_{132}$ & $q^8t^{-1} + q^8t^{-2} + q^4t + 2q^4 + 2q^4t^{-1} + q^4t^{-2} + t^2 + t + 1
$ \\ 
 \hline
  $10_{133}$ & $1 + tq^{-2} + 2q^{-2} + q^{-2}t^{-1} + t^2q^{-4} + 2tq^{-4} + q^{-4} + 2t^2q^{-6} + 3tq^{-6} + q^{-6} + t^2q^{-8} + tq^{-8} + t^3q^{-10} + t^2q^{-10}
$ \\ 
 \hline
  $10_{134}$ & $1 + tq^{-2} + q^{-2} + 2t^2q^{-4} + 2tq^{-4} + t^3q^{-6} + 2t^2q^{-6} + tq^{-6} + 2t^4q^{-8} + 3t^3q^{-8} + 2t^2q^{-8} + tq^{-8} + t^5q^{-10} + 2t^4q^{-10} + t^3q^{-10} + t^6q^{-12} + 2t^5q^{-12} + t^4q^{-12} + t^3q^{-12} + t^2q^{-12} + t^6q^{-14} + t^5q^{-14}
$ \\
\hline

\end{tabular}
\end{adjustwidth}
\end{center}

\begin{center}
\begin{adjustwidth}{-2cm}{}
\begin{tabular}{|c|p{15cm}|} 
 \hline
 Knot & Poincar\'e polynomial \\
 \hline\hline
  $10_{159}$ & $q^2t^{-1} + q^2t^{-2} + 3 + 2t^{-1} + 3tq^{-2} + 4q^{-2} + q^{-2}t^{-1} + 3t^2q^{-4} + 5tq^{-4} + 2q^{-4} + t^3q^{-6} + 4t^2q^{-6} + 3tq^{-6} + 2t^3q^{-8} + 2t^2q^{-8} + t^4q^{-10} + t^3q^{-10}
$ \\ 
 \hline
  $10_{160}$ & $q^2 + q^2t^{-1} + t + 2 + t^2q^{-2} + 2tq^{-2} + q^{-2} + t^3q^{-4} + 3t^2q^{-4} + 3tq^{-4} + q^{-4} + t^4q^{-6} + 2t^3q^{-6} + t^2q^{-6} + 2t^4q^{-8} + 2t^3q^{-8} + t^2q^{-8} + tq^{-8} + t^5q^{-10} + t^4q^{-10}
$ \\ 
 \hline
  $10_{161}$ & $1 + t^2q^{-4} + 2tq^{-4} + q^{-4} + t^2q^{-6} + tq^{-6} + t^4q^{-8} + 2t^3q^{-8} + 2t^2q^{-8} + tq^{-8} + t^3q^{-10} + t^2q^{-10} + t^6q^{-12} + t^5q^{-12} + t^4q^{-12} + t^3q^{-12}
$ \\ 
 \hline
  $10_{162}$ & $q^8t^{-2} + q^8t^{-3} + q^6t^{-1} + 3q^6t^{-2} + 2q^6t^{-3} + q^4 + 4q^4t^{-1} + 3q^4t^{-2} + 3q^2 + 5q^2t^{-1} + 2q^2t^{-2} + t + 3 + t^{-1} + 2tq^{-2} + 2q^{-2}
$ \\ 
 \hline
  $10_{163}$ & $q^4t^{-1} + q^4t^{-2} + 3q^2 + 3q^2t^{-1} + 3t + 5 + t^{-1} + 3t^2q^{-2} + 7tq^{-2} + 4q^{-2} + t^3q^{-4} + 5t^2q^{-4} + 4tq^{-4} + 3t^3q^{-6} + 4t^2q^{-6} + tq^{-6} + t^4q^{-8} + t^3q^{-8}
$ \\ 
 \hline
  $10_{164}$ & $q^6t^{-1} + q^6t^{-2} + 2q^4 + 4q^4t^{-1} + 2q^4t^{-2} + q^2t + 5q^2 + 4q^2t^{-1} + 4t + 7 + 2t^{-1} + t^2q^{-2} + 4tq^{-2} + 3q^{-2} + 2t^2q^{-4} + 2tq^{-4}
$ \\ 
 \hline
  $10_{165}$ & $q^{10}t^{-2} + q^{10}t^{-3} + 2q^8t^{-1} + 3q^8t^{-2} + q^8t^{-3} + q^6 + 5q^6t^{-1} + 4q^6t^{-2} + 3q^4 + 5q^4t^{-1} + 2q^4t^{-2} + q^2t + 4q^2 + 3q^2t^{-1} + t + 2
$ \\ 
 \hline

\end{tabular}
\end{adjustwidth}
\end{center}

\bibliographystyle{myamsalpha}
\bibliography{References}
\end{document}